\documentclass[11pt]{article}

\usepackage{microtype}
\usepackage{amsmath}
\usepackage{amssymb}
\usepackage{authblk}
\usepackage{algpseudocode,algorithm,algorithmicx}
\usepackage{mathrsfs}
\usepackage{graphicx}
\usepackage{ntheorem}
\usepackage{booktabs} 
\usepackage{subfiles}
\usepackage{geometry}
\usepackage{caption}
\usepackage{subcaption}
\usepackage{extsizes}
\usepackage{mathtools}
\usepackage{enumitem}
\usepackage{lscape}

\usepackage{hyperref}

\newtheorem{theorem}{Theorem}
\newtheorem{lemma}{Lemma}
\newtheorem{coro}{Corollary}
\newtheorem{remark}{Remark}
\newtheorem{assumption}{Assumption}

\newtheorem{proposition}{Proposition}
\newtheorem*{proof}{Proof}
\newcommand{\norm}[1]{\left\Vert#1\right\Vert}

\newcommand{\Mnorm}[1]{\left\Vert#1\right\Vert}

\newcommand{\R}{\mathbb{R}}
\newcommand{\bE}{\mathbb{E}}
\newcommand{\B}{\mathbb{B}}
\newcommand{\bN}{\mathbb{N}}
\newcommand{\cF}{\mathcal{F}}
\newcommand{\cK}{\mathcal{K}}
\newcommand{\cA}{\mathcal{A}}
\newcommand{\cC}{\mathcal{C}}
\newcommand{\cL}{\mathcal{L}}
\newcommand{\cB}{\mathcal{B}}
\newcommand{\cT}{\mathcal{T}}

\newcommand{\cI}{\mathcal{I}}

\newcommand{\dist}{\operatorname{dist}}
\newcommand{\dom}{\operatorname{dom}}

\newcommand{\trace}{\operatorname{trace}}
\def\<#1,#2>{\langle #1,#2\rangle}
\algrenewcommand\algorithmicrequire{\textbf{Parameters:}}
\algrenewcommand\algorithmicensure{\textbf{Initialize:}}

\begin{document}
	
	\title{An inexact proximal augmented Lagrangian framework  with arbitrary linearly convergent inner solver for  composite convex optimization}
	\author[]{Fei Li\thanks{\textit{Email:} \href{mailto:lifei16@connect.hku.hk}{lifei16@connect.hku.hk}. The author was supported by Hong Kong PhD Fellowship Scheme No. PF15-16399. }  }
	\author[]{Zheng Qu\thanks{\textit{Email:} \href{mailto:zhengqu@maths.hku.hk}{zhengqu@maths.hku.hk}.  The author was supported by Early Career Scheme from Hong Kong Research Grants Council No. 27302016. The computations were performed using research computing facilities offered by Information Technology Services, the University of Hong Kong.}}

	\affil[]{Department of Mathematics\\ The University of Hong Kong}

	\maketitle

	\begin{abstract}
		We propose an inexact proximal augmented Lagrangian  framework with explicit inner problem termination rule for composite convex optimization problems. We consider arbitrary linearly convergent inner solver including in particular stochastic algorithms,  making the resulting framework more scalable facing the ever-increasing problem dimension.
		Each subproblem is solved inexactly with an explicit and self-adaptive stopping criterion, without requiring to set an a priori target accuracy. 
		When the primal and dual domain are bounded, our method achieves $O(1/\sqrt{\epsilon})$ and $O(1/{\epsilon})$  complexity bound in terms of number of inner solver iterations, respectively for the strongly convex and non-strongly convex case. Without the boundedness assumption, only logarithm terms need to be added and the above two complexity bounds increase  respectively to $\tilde O(1/\sqrt{\epsilon})$ and $\tilde O(1/{\epsilon})$, which hold both for obtaining $\epsilon$-optimal and $\epsilon$-KKT solution. Within the  general framework that we propose, we also obtain $\tilde O(1/{\epsilon})$ and $\tilde O(1/{\epsilon^2})$ complexity bounds  under relative smoothness assumption on the differentiable component of the objective function. We show through theoretical analysis as well as numerical experiments the computational speedup possibly achieved by the use of randomized inner solvers for large-scale problems.

	\end{abstract}
	
	\section{Introduction}\label{sec:intro}

	We consider the following optimization problem:
	\begin{align}
		\label{prime}
		\min_{x\in \mathbb{R}^n} ~~& f(x)+ g(x)+ h_1(p_1(x))+h_2(p_2(x)).
	\end{align} Here $g:\R^{n}\rightarrow \R\cup\{+\infty\}$, $h_1:\R^{d_1}\rightarrow\R$   are proper, convex and closed functions.  The function
	$h_2:\R^{d_2}\rightarrow \R\cup\{+\infty\}$ is the indicator function of a convex and closed set $\cK\subset\R^{d_2}$: \begin{equation}\label{eq:h2}
	h_2(u_2)= \left\{\begin{array}{ll}0 & \mathrm{if~} u_2\in \cK \\
	+\infty & \mathrm{otherwise}
	\end{array}\right. 
	\end{equation}
	The function $f:\R^n\rightarrow \R\cup\{+\infty\}$ is  convex and differentiable on an open set containing $\dom(g)$. The functions $p_1:\R^n\rightarrow \R^{d_1}$ and $p_2: \R^n\rightarrow \R^{d_2}$ are differentiable. In addition, we assume that $g, h_1, h_2$ are \textit{simple} functions, in the sense that their proximal operator are easily computable. With some other standard assumptions stated in the later discussion, the model that we consider covers a wide range of optimization problems.
	As an example, the following linearly constrained convex optimization problem
	\begin{align}
		\label{primefff}
		&\min_{x\in \mathbb{R}^n} ~~~ f(x)+ g(x)\\
		&\enspace s.t. \qquad Ax=b \notag
	\end{align}
	is a special case of~\eqref{prime} by letting $h_1\equiv 0$, $\cK=\{b\}$ and $p_2(x)\equiv Ax$.  Important applications of~\eqref{primefff} include  model predictive control~\cite{751369} and basis pursuit problem~\cite{ChenDonohoSaubders}.
	When $h_1\equiv 0$, $\cK$ is a closed convex cone in $\R^{d_2}$ and $p_2(\cdot)$ is convex with respect to $\cK$, problem~\eqref{prime} reduces to the convex conic programming model~\cite{NecoaraPatrascuGlineur,lu18} and in particular contains the constrained convex programming problem~\cite{rock76}:
	\begin{align}
		\label{primeff33f}
		&\min_{x\in \mathbb{R}^n} ~~~ f(x)+g(x)\\
		&\enspace s.t. \quad f_1(x)\leq 0 ,\enspace \cdots, f_m(x)\leq 0 \notag .
	\end{align}
	Apart from constrained programs,  problem~\eqref{prime} also covers many popular models in machine learning, including the sparse-group LASSO~\cite{Simon13asparse-group}, the fused LASSO~\cite{Tibshirani05sparsityand}, the square root LASSO~\cite{articlesqa}, and the support vector machine problem~\cite{Zhu:2003:SVM:2981345.2981352}. 
	
	In~\cite{rock76}, Rockafellar built an inexact augmented Lagrangian method (ALM) framework for solving~\eqref{primeff33f}.
	At each iteration of  the inexact ALM, one needs to solve a convex optimization problem (referred to as \textit{inner problem}) presumed easier than the original constrained problem~\eqref{primeff33f}, up to a certain accuracy.  
	Rockafellar~\cite{rock76} gave some stopping criteria for  the test of the inner problem solution accuracy, as well as some sufficient conditions guaranteeing the convergence of the inexact ALM method.  
	In~\cite{Nesterov2005}, Nesterov proposed a smoothing technique to deal with the unconstrained case ($h_2\equiv 0$). The idea is again to replace the original problem by an easier subproblem and solve it up to a desired accuracy.
	Although existing work usually consider either  $h_1\equiv 0$ or $h_2\equiv 0$~\cite{ber14,cham11}, 
	we can treat them in a unified way  because the augmented Lagrangian function corresponds to a smooth approximation of the function $h_2$.
	In fact, both Rockafellar's inexact ALM framework and Nesterov's smoothing technique are applications of  the inexact proximal point method~\cite{rock76ppa, BeckTeboulle12}.  
	There is no essential difficulty in extending existing results from the case $h_1\equiv 0$ or $h_2\equiv 0$ to the general model~\eqref{prime}.  For this reason,  in the following discussion, we do  not make particular difference between the papers dealing with the  two different cases (either  $h_1\equiv 0$ or $h_2\equiv 0$).

	We mainly consider two optimality criteria  for the complexity analysis of inexact ALM. 
	One  is based on the primal feasibility and the primal value optimality gap and the other on the KKT-residual. 
	A solution $x\in \dom(g)$ is said to be $\epsilon$-optimal if~\cite{rock76, Nesterov2005, BeckTeboulle12, NedelcuNecoaraTran,PatrascuNecoara15,NecoaraPatrascuGlineur,Xu2017IterationCO}
	\begin{align}\label{a:epsiol}
		|F(x)-F^\star|\leq \epsilon,\enspace  \dist(p_2(x), \cK)\leq \epsilon.
	\end{align} 
	Here,
	\begin{align}\label{a:defF}
		F(x):= f(x)+g(x)+h_1(p_1(x)),\enspace \forall x\in \R^n,
	\end{align}
	and $F^\star$ denotes the optimal value of~\eqref{prime}. A solution $x\in \dom(g)$  is said to be $\epsilon$-KKT optimal if there is $\lambda_1\in \dom(h_1^*)$ and $\lambda_2\in \dom(h_2^*)$ such that~\cite{Lan:2016:IFA:2874819.2874858,lu18}
	\begin{align}\label{a:erd}
		\dist(0,\partial_x L(x, \lambda_1,\lambda_2)) \leq \epsilon,\enspace\dist(0,\partial_{\lambda_1}L(x, \lambda_1,\lambda_2)) \leq \epsilon,\enspace\dist(0,\partial_{\lambda_2}L(x, \lambda_1,\lambda_2)) \leq \epsilon.
	\end{align}
	Here,
	$$
	L(x,\lambda_1,\lambda_2):=f(x)+g(x)+\<\lambda_1, p_1(x)>-h_1^*(\lambda_1)+\<\lambda_2, p_2(x)>-h_2^*(\lambda_2), \enspace \forall x\in \R^n, \lambda_1\in \R^{d_1}, \lambda_2\in \R^{d_2},
	$$
	denotes the Lagrangian function. A different criterion which can be derived from~\eqref{a:erd} under the boundedness of $\dom(g)$ was used in~\cite{LiuLiuMa19MoR}. Most of the previously cited papers studied the complexity bound of the inexact ALM, which is the number of inner iterations needed for computing an $\epsilon$-optimal solution or an $\epsilon$-KKT solution. The lowest known complexity bound is $O(\epsilon)$ for obtaining an $\epsilon$-optimal solution~\cite{Nesterov2005, BeckTeboulle12, PatrascuNecoara15,NecoaraPatrascuGlineur,Xu2017IterationCO,tran18ada}, and $\tilde O(\epsilon)$ for obtaining an $\epsilon$-KKT solution~\cite{lu18}.

	There are some variants of inexact ALM which avoid the solution of inner problems, including the linearized ALM~\cite{Xu2017IterationCO} and linearized ADMM~\cite{ouyang15} as well as their stochastic extensions~\cite{xu2017first,xu2018accelerated,chambolle2018stochastic}. These inner problem free methods are widely used in practice thanks to their simple implementation form and good practical convergence behavior.  However, $O(\epsilon)$ complexity bound of these methods are established only in an ergodic sense, not in the last iterate.
	In~\cite{tran18}, an accelerated smooth gap reduction method (ASGARD) was developed with a non-ergodic $O(\epsilon)$ complexity bound and showed superior numerical performance than linearized ADMM. The method has been extended to a stochastic block coordinate update version in~\cite{ala17}, called SMART-CD. In practice, it was observed that appropriately restarting ASGARD or SMART-CD can further speed up the convergence. In~\cite{tran18ada},  the authors analyzed a double-loop ASGARD (ASGARD-DL) which achieves the non-ergodic complexity bound $O(\epsilon)$ and has similar practical convergence behavior as ASGARD with restart. ASGARD-DL~\cite{tran18ada} can be seen as an inexact ALM. However, in contrast to a series of work on inexact ALM~\cite{Nesterov2005, BeckTeboulle12, PatrascuNecoara15,NecoaraPatrascuGlineur,Xu2017IterationCO}, ASGARD-DL has an explicit inner termination rule and does not require the boundedness of $\dom(g)$.
	

	The boundedness assumption of $\dom(g)$ seems to be crucial in the existing analysis for inexact ALM since it allows to directly control the number of iterations needed for the solution of each inner problem, using deterministic first-order solvers such as the accelerated proximal gradient (APG)~\cite{beck2009fista}. It was argued that such boundedness assumption is mild because  in some cases it is possible to find a bounded set including the optimal solution~\cite{LiuLiuMa19MoR}. Nevertheless, removing this assumption from the complexity analysis of inexact ALM seems to be challenging and requires different approaches from existing ones. ASGARD-DL~\cite{tran18ada} is among the first which  
	achieve the best complexity bound $O(\epsilon)$ without making the compactness assumption.
	However, their analysis builds on a very special  property  of APG  and thus excludes the possibility of other inner solvers.

	Allowing more flexible choice of inner solver is a very important feature in the large-scale setting. It is recognized that some randomized first-order methods can be more efficient than APG when the problem dimension is high. This includes for example the randomized coordinate descent variant of APG (\textit{a.k.a.} APPROX)~\cite{FR:2013approx} and the stochastic variance reduced variant of APG (\textit{a.k.a.} Katyusha)~\cite{Katyusha}.  Compared with their deterministic origin APG, APPROX can reduce the computation load when the number of coordinates $n$ is large, while Katyusha is more efficient when the number of constraints $m$ (in~\eqref{primeff33f}) is large.  
	With the ever-increasing scale of the problems to be solved, it is necessary to employ randomized methods for solving the inner problems. However, the
	complexity analysis of inexact ALM with randomized inner solvers seems not to have been fully investigated.
	
	In this paper, we develop an inexact proximal ALM which does not require the boundedness of $\dom(g)$ for the inner termination rule, and analyze its total complexity bound for any linearly convergent inner solver. Since randomized inner solvers are included, we will only require the optimality criteria~\eqref{a:epsiol} and~\eqref{a:erd} to be achieved in expectation, see~\eqref{a:ef} and~\eqref{a:tdff}. In addition, the complexity bound that we provide is an upper bound on the expectation of the number of total inner iterations.
	We summarize below our contributions.
	\begin{enumerate}
		\item We give a stochastic extension of Rockafellar's inexact  proximal ALM framework, see Algorithm~\ref{ipALM}. The  difference with the original framework lies in the inner problem stopping criteria, which only asks the   inner optimality gap to be smaller than a certain threshold in expectation.
		\item 
		For any linearly convergent inner solver $\cA$, we give an upper bound on the number of inner iterations required to satisfy the  stopping criteria, see~\eqref{eq:ms2}.  In contrast to  related work~\cite{NecoaraPatrascuGlineur,Lan:2016:IFA:2874819.2874858,NedelcuNecoaraTran,PatrascuNecoara15,LiuLiuMa19MoR,Xu2017IterationCO,lu18}, the upper bound computed by~\eqref{eq:ms2} does not depend  on the diameter of  $\dom(g)$ and in particular does not need to assume the boundedness of $\dom(g)$.
		Instead, our upper bound is adaptively computed based on the previous and current  primal and dual iterates, as well as the linear convergence rate of the inner solver $\cA$.  
		
		\item Based on the explicit upper bound computed by~\eqref{eq:ms2}, we propose an inexact proximal ALM with an explicit inner termination rule, see Algorithm~\ref{ipALM_m}. Compared with the previously mentioned work, our termination rule \begin{itemize}
			\item \textbf{does not} require the desired accuracy $\epsilon$ to be set a priori;
			\item \textbf{does not} need to assume the boundedness of $\dom(g)$.
		\end{itemize}

		\item We show that the complexity bound of Algorithm~\ref{ipALM_m} is $\tilde O(1/\epsilon^{\ell})$ to obtain an $\epsilon$-optimal solution where $\ell>0$ is a constant determined by the convergence rate of the inner solver $\cA$, see Theorem~\ref{main:theoemalgo2}. 
		Our approach can be easily extended to obtain $\tilde O(1/\epsilon^{\ell})$ complexity bound for $\epsilon$-KKT solution, see Section~\ref{sec:kkt}.  When both the primal and dual domains are bounded, the bound $\tilde O(1/\epsilon^{\ell})$ can be improved to $ O(1/\epsilon^{\ell})$ for  obtaining an $\epsilon$-optimal solution, see Section~\ref{sec:bound}.
		
		\item  We show how to apply Theorem~\ref{main:theoemalgo2} under different problem structures and assumptions. When $p_1(\cdot)$ and $p_2(\cdot)$ are linear,
		under the same assumptions as~\cite{NecoaraPatrascuGlineur,Lan:2016:IFA:2874819.2874858,NedelcuNecoaraTran,PatrascuNecoara15,LiuLiuMa19MoR,lu18} but without the boundedness  of $\dom(g)$, we obtain $\tilde O(1/\epsilon)$ and  $\tilde O(1/\sqrt{\epsilon})$ complexity bound respectively for the non-strongly convex and strongly convex case, see Corollary~\ref{coro:lAPG} and~\ref{coro:LL}. We also consider the case when $f$ is only relatively smooth, and establish $\tilde O(1/\epsilon)$ and $\tilde O(1/\epsilon^2)$ complexity bound respectively for the non-strongly convex and strongly convex case, see Corollary~\ref{coro:BPG}.

		\item 
		We provide theoretical justification  to support  the use of randomized solvers in large-scale setting, see Table~\ref{tab_0}. We give numerical evidence to show that with appropriate choice of inner solver, our algorithm outperforms ASGARD-DL and SMART-CD, see Figure~\ref{fig_1}, \ref{fig_2}, \ref{fig_2_inf}, \ref{fig_3}, \ref{fig_7}. Moreover,  compared with CVX, our algorithm often obtains a  solution  with medium accuracy
		within less computational time,  see Table~\ref{tab_2},~\ref{tab_3},~\ref{tab_4}.
	\end{enumerate}

	\textbf{Notations.} For any two vectors $\lambda_1 \in \R^{d_1}$ and $ \lambda_2 \in \R^{d_2}$ we denote by $(\lambda_1;\lambda_2)$ the vector in $\R^{d_1+d_2}$ obtained by concatenating $\lambda_1$ and $\lambda_2$. Inversely, for any $\lambda\in \R^{d_1+d_2}$ we denote by $\lambda_1\in \R^{d_1}$ the vector containing the first $d_1$ components of $\lambda$ and $\lambda_2\in \R^{d_2}$ the vector containing the last $d_2$ components of $\lambda$. We use $\|\cdot\|$ to denote the standard Euclidean norm for vector and spectral norm for matrix. For any matrix $A$, $A_{i,i}$ is the $i$th diagonal element of $A$. We denote by $e_i\in \R^n$ the $i$th standard basis vector in $\R^n$. For proper, closed and convex function $h(\cdot)$, $h^{*}(\cdot)$ denotes its Fenchel conjugate function. For any $x\in \R^{d_2}$, $\dist(x, \cK)$ denotes the distance from $x$ to $\cK$.
	For any integer $n$ we denote by $[n]$ the set $\{1,2,\cdots,n\}$.

	\textbf{Organization.} 
	In Section~\ref{sec:pre}, we study an inexact proximal ALM framework with expected inexactness condition. In Section~\ref{sec:rri}, we give an upper bound on the number of the inner iterations and obtain an instantiation of the general inexact proximal ALM. In Section~\ref{sec:total}, we briefly recall several  first order methods and their respective convergence rate. In Section~\ref{sec:rterr}, we apply  
	our main results to different structured problems. In Section~\ref{sec::dis}, we discuss some extension of our  work. In Section~\ref{sec::num}, we present numerical experiments. In Section~\ref{sec::con}, we make some concluding remarks.  Background knowledge used and missing proofs can be found in the Appendix.
	
	\section{Preliminaries}\label{sec:pre}
	\subsection{Problem and Assumptions}
	For ease of presentation we rewrite~\eqref{prime} as
	\begin{align}
		\label{prime2}
		\min_{x\in \mathbb{R}^n} f(x)+g(x)+h(p(x)),
	\end{align}
	where 
	$$
	h((u_1;u_2)):=h_1(u_1)+h_2(u_2),\enspace u_1\in\R^{d_1}, u_2\in \R^{d_2},
	$$
	and 
	$$
	p(x):= ( p_1(x); p_2(x)),\enspace x\in \R^n.
	$$
	Define the Lagrangian function
	\begin{align}
		\label{langrangian}
		L(x;\lambda):=f(x)+g(x)+\langle\lambda, p(x)\rangle-h^*(\lambda),
	\end{align}
	and consider  the Lagrange dual problem:
	\begin{align}\label{a:dfer}
		\max_{\lambda\in \R^d}  \left[D(\lambda) \equiv\inf_x L(x;\lambda)\right].
	\end{align}
	We shall call problem~\eqref{prime2} the \textit{primal problem} and~\eqref{a:dfer} the \textit{dual problem}.
	Apart from the structures mentioned in the very beginning of Section~\ref{sec:intro}, we make the following additional assumptions  throughout the paper.
	\begin{assumption}\label{ass:handp}
		~
		
		\begin{enumerate}[label=(\alph*)]
			\item  \label{ass:hp1} $h_1$ is  $L_{h_1}$-Lipschitz continuous.
			\item \label{ass:pq2} for any $x, y\in \R^n$, $u_1,v_1 \in \R^{d_1}$ and $\alpha\in(0,1)$
			$$
			h_1 \left(p_1(\alpha x+(1-\alpha)y)-\alpha u_1-(1-\alpha)v_1\right)\leq \alpha h_1(p_1(x)-u_1)+(1-\alpha) h_1(p_1(y)-v_1).
			$$
			\item \label{ass:pq3} for any $x, y\in \R^n$, $u_2,v_2\in \R^{d_2}$ and $\alpha\in(0,1)$ such that $
			p_2(x)-u_2\in \cK$ and $ p_2(y)-v_2\in \cK$, it holds that
			$$ p_2\left(\alpha x+(1-\alpha)y)\right)-\alpha u_2-(1-\alpha)v_2\in \cK.
			$$
			\item both the primal and the dual problem have optimal solution and the strong duality holds, i.e., there is $x^\star\in \dom(g)$ and $\lambda^\star\in \dom(h^*)$ such that $g(x^\star)\in \cK$ and
			\begin{align}\label{a:strd}
				F(x^\star)=L(x^\star;\lambda^\star)=D(\lambda^\star).
			\end{align}
		\end{enumerate}
	\end{assumption}	
	If $p_1(\cdot):\R^{n}\rightarrow \R^{d_1}$ is affine, then Assumption~\ref{ass:pq2} holds. Otherwise,~\ref{ass:pq2} holds if there is a partial order $\preceq_{\cC_1}$ on $\R^{d_1}$ induced by a closed convex cone  $\cC_1 \subset \R^{d_1}$ (i.e. $x\preceq_{\cC_1}y$ if and only if $y-x\in \cC_1$) such that the function $p_1(\cdot)$ is convex with respect to the order $\preceq_{\cC_1}$, i.e.,
	\begin{align}\label{a:pass}
		p_1(\alpha x+(1-\alpha)y)\preceq_{\cC_1} \alpha p_1(x)+(1-\alpha)p_1(y),
	\end{align}
	and the function $h_1(\cdot)$ is order preserving with respect to $\preceq$, i.e.,
	\begin{align}\label{a:hass}
		u_1\preceq_{\cC_1} v_1 \Longrightarrow h_1(u_1)\leq h_1(v_1).
	\end{align}
	Similarly, if  $p_2(\cdot):\R^{n}\rightarrow \R^{d_2}$ is affine, then Assumption~\ref{ass:pq3} holds. Otherwise,~\ref{ass:pq3} holds if there is a partial order $\preceq_{\cC_2}$ on $\R^{d_2}$ induced by a closed convex cone $\cC_2$ such that the function $p_2(\cdot)$ is convex with respect to the order $\preceq_{\cC_2}$, i.e.,
	\begin{align}\label{a:pass2}
		p_2(\alpha x+(1-\alpha)y)\preceq_{\cC_2} \alpha p_2(x)+(1-\alpha)p_2(y),
	\end{align}
	and the set $\cK$ is such that $u_2+z_2\in \cK$ for any $u_2\in \cK$ and $z_2\preceq_{\cC_2} 0$.
	\begin{remark}
		For example,  consider the partial order $\preceq$  induced by the nonnegative orthant $\R^{d_1}_+$, then~\eqref{a:pass} is satisfied if $p_1(x)=(q_1(x),\cdots,q_{d_1}(x))^\top$ with each $q_i:\R^n\rightarrow \R$  being convex. If the partial order $\preceq$  is induced by the cone of positive semidefinite matrices $S^m_+$, then~\eqref{a:pass} holds if
		$p_1(x)=\sum_{i=1}^t B_i q_i(x)$ with $B_1,\dots, B_t \in S^m_+$ and each $q_i:\R^n\rightarrow \R$  being convex, see, e.g.~\cite{Auslender:2005:IPM:3113613.3114011}. The same class of examples apply to~\eqref{a:pass2}.
	\end{remark}
	\begin{remark}
		A special case when~\eqref{a:hass} holds is when $h_1$ is the support function of some bounded set included in the dual cone of $\cC_1$, i.e.,
		$$
		h_1(x)\equiv \sup \{\<y, x>: y\in \B \cap \cC_1^{*} \},
		$$
		where $\B$ is a bounded set and $\cC_1^{*}:=\{y: \<y,x> \geq 0,\enspace \forall x\in \cC_1\}$ is the dual cone of $\cC_1$. For example, when $\B$ is the unit ball with respect to the standard Euclidean norm and $\cC_1=\R^{d_1}_+$ is the nonpositive orthant,  then $h_1(x)=\| \max(x,0)\|$, see e.g.~\cite{FriedlanderGoh:2016}.
	\end{remark}
	Let $d=d_1+d_2$. Condition~\ref{ass:pq2} and~\ref{ass:pq3} imply that for any $x, y\in \R^n$, $u,v \in \R^d$ and $\alpha\in(0,1)$
	\begin{align}\label{a:herffg}
		h\left(z\right)\leq \alpha h(p(x)-u)+(1-\alpha) h(p(y)-v),
	\end{align}
	where $z=p(\alpha x+(1-\alpha)y)-\alpha u-(1-\alpha)v$.  The latter condition guarantees the convexity of $h(p(\cdot)):\R^n\rightarrow \R\cup\{+\infty\}$.

	\subsection{Proximal ALM Revisited}\label{sec:prox_ALM}
	Let any $\lambda\in \R^d$ and $\beta>0$. Define
	\begin{align}\label{a:hbetadef}
		h(u;\lambda,\beta):=\max_{v\in \R^d}\left\{\<v, u>-h^{*}(v)- \frac{\beta}{2}\|v-\lambda\|^2 \right\},
	\end{align}
	and
	\begin{align}\label{a:Lambda}
		\Lambda(u;\lambda, \beta):=\arg\max_{v\in \R^d}\left\{\<v, u>-h^{*}(v)- \frac{\beta}{2}\|v-\lambda\|^2 \right\}.
	\end{align}
	The function $h(\cdot;\lambda,\beta)$ is known as an approximate smooth function of the possibly nonsmooth function $h(\cdot)$ with parameter $\lambda$ and $\beta$. We next recall some results needed later about the smooth function $h(\cdot;\lambda,\beta)$. 
	\begin{lemma}[\cite{Nesterov2005, BauschkeCombetter09,BeckTeboulle12}]\label{l:erdfgt}
		\begin{enumerate}
			\item 	The function $h(u;\lambda,\beta)$ is convex and differentiable with respect to $u$. Denote by $\nabla_1 h(u;\lambda,\beta)$ the gradient with respect to the variable $u$, then we have
			\begin{align}\label{a:nablehu}
				&\nabla_1 h(u;\lambda,\beta)=\Lambda(u;\lambda,\beta)\\ \label{a:nablahLip}
				&\|\nabla_1 h(u;\lambda,\beta)-\nabla_1 h(v;\lambda, \beta)\| \leq \beta^{-1} \| u-v\|
			\end{align}
			\item 	For any $u, \lambda\in \R^d$ and $\beta>0$ we have
			\begin{align}\label{a:dualityhbeta}
				& h(u;\lambda,\beta)=\min_w\left\{ h(u-w)+\frac{1}{2\beta}\|w\|^2+\<w,\lambda> \right\} \leq h(u) 
			\end{align}
			In addition, the optimal solution $w^\star$ for~\eqref{a:dualityhbeta} is given by
			\begin{align}\label{a:wstar}
				w^\star=\beta(\Lambda(u;\lambda,\beta)-\lambda),
			\end{align}
			and 
			\begin{align}
				& \label{a:hbeta}
				h(u;\lambda,\beta)=h(u-\beta(\Lambda(u;\lambda,\beta)-\lambda))+\frac{\beta}{2}\|\Lambda(u;\lambda,\beta)\|^2-\frac{\beta}{2}\|\lambda\|^2
			\end{align}
			\item For any $u,\lambda\in \R^d$ and $\beta>0$, we have
			\begin{align}\label{a:optimalitycondition}
				u-\beta(\Lambda(u;\lambda, \beta)-\lambda)\in \partial h^{*}(\Lambda(u;\lambda,\beta)).
			\end{align}
		\end{enumerate}
	\end{lemma}
	\begin{lemma}\label{l:infimalconv}
		Let $\psi(\cdot):\R^n\rightarrow\R \cup\{+\infty\}$ 
		be a convex function.  Define:
		$$
		\tilde \psi(x):=\inf_w\{h(p(x)-w)+\psi(w)\},
		$$
		Then condition~\eqref{a:herffg} ensures the convexity of   $\tilde \psi$.
	\end{lemma}
	\begin{lemma}\label{l:rterrr}
		Fix any $\lambda\in\R^d$ and $\beta>0$. Define $$\tilde \psi (x):=h(p(x);\lambda,\beta)$$ Then $\tilde \psi:\R^n\rightarrow \R$ is a convex  and differentiable function with $\nabla \tilde \psi(x)=\nabla p(x)\Lambda(p(x);\lambda,\beta)$.
	\end{lemma}
	Define 
	\begin{align}\label{a:pAL}
		L(x;y,\lambda, \beta):=f(x)+g(x)+h(p(x);\lambda,\beta)+\frac{\beta}{2}\|x-y\|^2.
	\end{align}
	It follows from Lemma~\ref{l:rterrr} that  for any $\beta>0$ the function $L(x;y,\lambda,\beta)$ is strongly convex with respect to the variable $x$. 
	Let $\{\beta_s:s\geq 0\}$  and $\{\epsilon_s:s\geq 0\}$ be two sequence of positive numbers. We recall the  inexact  proximal augmented Lagrangian framework in Algorithm~\ref{ipALM}.  The objective function at outer iteration $s$ is denoted by $H_s(\cdot)$ and $H_s^\star:=\min_x H_s(x)$.
	A small difference with the classical  inexact proximal ALM in~\cite{rock76} is that we only require to control  the  expectation of the subproblem objective value gap. In particular, note that the iterates $\{(x^s,\lambda^s)\}$ in Algorithm~\ref{ipALM} are random variables.  We denote by $\cF_s$ the $\sigma$-algebra generated by $\{x^t: t\leq s\}\cup \{\lambda^{t}:t\leq s\}$. 
	\begin{algorithm}[H]
		\caption{IPALM (compare with~\cite{rock76})}
		\label{ipALM}
		\begin{algorithmic}
			\Require   $\{\epsilon_s\}$, $\{\beta_s\}$;
			\Ensure  $x^{-1}\in \dom(g)$, $\lambda^0\in \dom(h^{*})$;
			\For {$s= 0, 1, \ldots$}
			\State Find $ x^s\simeq \arg\min H_s(x)\equiv L(x; x^{s-1}, \lambda^s, \beta_s)$ satisfying
			$
			\bE[{H_s}(x^s)- H_s^\star]\leq \epsilon_{s}
			$
			\State $\lambda^{s+ 1}\leftarrow \Lambda(p(x^{s});\lambda^s,\beta_s)$
			\EndFor
		\end{algorithmic}
	\end{algorithm}
	We now recall a few known results about inexact proximal ALM. Note that we are in a slightly more general setting than~\cite{rock76} due to our problem formulation~\eqref{prime} and the expected inexactness condition in Algorithm~\ref{ipALM}. A modification of the original proof is needed to obtain the  desired results. For completeness  proofs can be found in Appendix~\ref{app:pALM}. Hereinafter,  $x^\star$ is an arbitrary optimal solution of the primal problem~\eqref{prime2} and $\lambda^\star$ is an arbitrary optimal solution of the dual problem~\eqref{a:dfer}.

	\begin{lemma}[compare with~\cite{rock76ppa}]
		\label{bound_y_1.3}
		Let $\{x^{s}, \lambda^s\}$ be the sequence generated by Algorithm $\ref{ipALM}$. Then for any $s\geq 0$,
		\begin{align}
			&\mathbb{E}\left[\Mnorm{(x^{s},\lambda^{s+ 1})- (x^{s-1},\lambda^{s})}\right]\le \Mnorm{(x^{-1},\lambda^0)- (x^\star,\lambda^\star)}+ \sum_{i= 0}^s \sqrt{2\epsilon_{i}/\beta_i}\\ \label{a:erdfdfdfd}
			&\bE\left[ \Mnorm{(x^{s},\lambda^{s+ 1})- (x^\star,\lambda^\star)}^2\right] \leq \left(\Mnorm{(x^{-1},\lambda^0)- (x^\star,\lambda^\star)}+ \sum_{i= 0}^s \sqrt{2\epsilon_{i}/\beta_i} \right)^2 
		\end{align}
	\end{lemma}
	Lemma~\ref{bound_y_1.3} allows us to give a bound on the produced primal dual sequence $\{(x^s,\lambda^s)\}_s$.
	\begin{coro}\label{coro:bound}
		Consider Algorithm~\ref{ipALM} with $\beta_s=\beta_0 \rho^s$ and $\epsilon_s=\epsilon_0 \eta^s$ for some $0<\eta<\rho<1$. Define
		\begin{align}\label{a:defc0}
			c_0:=2\left(\Mnorm{(x^{-1},\lambda^0)- (x^\star,\lambda^\star)}+ \frac{2\sqrt{\epsilon_0/\beta_0}}{1-\sqrt{\eta/\rho}} \right)^2 +2 \|\lambda^\star\|^2+2\|x^\star\|^2.
		\end{align}
		Then we have
		$$
		\max\left(\bE\left[\|\lambda^s\|^2\right], \bE\left[\|x^s\|^2\right],\bE\left[ \Mnorm{x^{s}-x^\star}^2\right]\right)\leq c_0,\enspace \bE\left[\|\lambda^{s+1}-\lambda^s\|\right]\leq \sqrt{c_0},\enspace \forall s\geq 0.
		$$
	\end{coro}
	
	\begin{theorem}[compare with~\cite{rock76}]
		\label{rock2} Consider Algorithm~\ref{ipALM}.
		We have the following bounds:
		\begin{align}\label{a:rock2e1}
			&  F(x^s)-F^\star\leq H_s(x^s)-H_s^\star+ L_{h_1} \beta_s \|\lambda_1^{s+1}-\lambda_1^s\|+\frac{\beta_s}{2}(\| \lambda^{s}\|^2-\|\lambda^{s+1}\|^2) +\frac{\beta_s}{2}\|x^\star-x^{s-1}\|^2	\\
			&F(x^s)-F^\star	\geq -\beta_s\|\lambda_2^\star\| \| \lambda_2^{s+1}-\lambda_2^s\|,\\
			&	\dist(p_2(x^s), \cK)\leq \beta_s\| \lambda_2^{s+1}-\lambda_2^s\|. 	\end{align}
	\end{theorem}
	
	\begin{coro}\label{coro:them1c}
		Consider Algorithm~\ref{ipALM} with $\beta_s=\beta_0 \rho^s$ and $\epsilon_s=\epsilon_0 \eta^s$ for some $0<\eta<\rho<1$. Then to obtain a solution $x^s$ such that
		\begin{align}\label{a:ef}
			\left |\bE[F(x^s)-F^\star]\right | \leq \epsilon,\enspace \bE[ \dist(p_2(x^s),\cK)]\leq \epsilon,
		\end{align}
		if suffices to run Algorithm~\ref{ipALM} for 
		\begin{align}\label{a:sbound}
			s\geq \frac{\ln(c_1/\epsilon)}{\ln 1/\rho}
		\end{align}
		number of outer iterations where 
		\begin{align}\label{a:defc1}
			c_1:= \max(\epsilon_0+ 2L_{h_1}^2 \beta_0 +c_0\beta_0, \beta_0\|\lambda_2^\star\| \sqrt{c_0}, \beta_0 \sqrt{c_0})
		\end{align}
		with $c_0$ defined in~\eqref{a:defc0}.
	\end{coro}

	\section{Recursive Relation of Inexactness}\label{sec:rri}
	The main objective of this section is to show that the  initial error  $H_{s+1}(x^s)-H_{s+1}^\star$ of the inner problem at iteration $s+1$ in Algorithm~\ref{ipALM} can be upper bounded using the last step error  $H_s(x^s)-H_s^\star$ and some computable quantities.   The bound will yield a way to control the number of inner iterations.
	The key proposition of this section is as follows.
	\begin{proposition}\label{prop:main2}
		Consider Algorithm~\ref{ipALM}.
		If $\beta_s\geq \beta_{s+1}> \beta_s/2$, then
		\begin{equation}\label{eq:mainpop2}
		\begin{array}{ll}
		& H_{s+1}(x^s) -H_{s+1}^\star
		\\&\leq 2(H_{s}(x^s) -H_{s}^\star)+{{\beta_s}}\|\lambda^{s+1}-\lambda^{s}\|^2+
		\frac{\beta_s-\beta_{s+1}}{2}\| \Lambda(p(x^s); \lambda^{s+1},\beta_{s+1})-\lambda^{s+1}\|^2+
		\frac{\beta_s^2}{2\beta_{s+1}-\beta_s}\|x^{s-1}-x^s\|^2\\&\quad+ \|\lambda^{s+1}-\lambda^{s}\|\sqrt{ \left((\beta_s+\beta_{s+1})L_{h_1}+\|\beta_s \lambda_1^s-\beta_{s+1}\lambda_1^{s+1}\| \right)^2
			+\|\beta_s \lambda_2^s-\beta_{s+1}\lambda_2^{s+1}\|^2 }
		\end{array}.
		\end{equation}
	\end{proposition}
	
	We defer the proof of Proposition~\ref{prop:main2} in Section~\ref{sec:H}.  In the next section we show how to make use of  Proposition~\ref{prop:main2} to get an implementable form of Algorithm~\ref{ipALM}. Hereinafter we assume that we have at our disposal  an algorithm $ \cA$ suitable for solving each inner problem in Algorithm~\ref{ipALM}:
	\begin{align}\label{a:minHs}
		\min_x H_s(x).
	\end{align}
	Denote by $\cA(x, k, H_s)$ the output obtained by running $k$ iterations of Algorithm $\cA$   on problem~\eqref{a:minHs}  starting with  initial solution $x$.
	We only consider those inner solvers $\cA$  satisfying the following requirement. 
	\begin{assumption}[Linearly Convergent Inner Solver]\label{ass:innersolver}
		For any outer iteration $s\in\{0,1,\dots\}$ of Algorithm~\ref{ipALM},  there is $ K_s\geq 1$ such that for any $x\in \dom(g)$,
		\begin{align}\label{a:etasdef}
			\bE\left[H_s\left(  \cA(x, k, H_s)\right)-H_s^\star| \cF_{s-1}\right]\leq {2^{-\lfloor k/ K_s\rfloor}}   \left( H_s(x)-H_s^\star\right).
		\end{align}
	\end{assumption}
	We will give in Section~\ref{sec:total} some examples of algorithms satisfying these properties.
	
	\subsection{Inner Iteration Complexity Control for ALM}
	In this section we apply Proposition~\ref{prop:main2} to derive an implementable form of Algorithm~\ref{ipALM}.  
	\begin{coro}\label{coro:dfd2}
		Consider Algorithm~\ref{ipALM} with $\beta_s\geq \beta_{s+1}> \beta_s/2$.
		Let $m_{s+1}>0$ be an integer satisfying
		\begin{equation}\label{eq:ms2}
		\begin{array}{ll}&2\epsilon_s+{{\beta_s}}\|\lambda^{s+1}-\lambda^{s}\|^2+
		\frac{\beta_s-\beta_{s+1}}{2}\| \Lambda(p(x^s); \lambda^{s+1},\beta_{s+1})-\lambda^{s+1}\|^2+ \frac{\beta_s^2}{2\beta_{s+1}-\beta_s}\|x^{s-1}-x^s\|^2\\&+ \|\lambda^{s+1}-\lambda^{s}\|\sqrt{ \left((\beta_s+\beta_{s+1})L_{h_1}+\|\beta_s \lambda_1^s-\beta_{s+1}\lambda_1^{s+1}\| \right)^2
			+\|\beta_s \lambda_2^s-\beta_{s+1}\lambda_2^{s+1}\|^2 } \leq {2^{\lfloor m_{s+1}/ K_{s+1}\rfloor}}  {\epsilon_{s+1}/2 }
		\end{array}.
		\end{equation}
		If $\bE\left [{H_s}(x^s)- H_s^\star\right ]\leq \epsilon_{s}$, then 
		\begin{align}\label{a:dererdf2}
			\bE\left [{H_{s+1}}(x^{s+1})- H_{s+1}^\star\right ]\leq \epsilon_{s+1}, 
		\end{align}
		is guaranteed by letting
		$$x^{s+1}= \cA(x^{s}, m_{s+1}, H_{s+1}).$$
	\end{coro}
	\begin{proof}
		Denote
		\begin{equation}\label{eq:Ms}
		\begin{array}{l}
		M_s:={{\beta_s}}\|\lambda^{s+1}-\lambda^{s}\|^2+
		\frac{\beta_s-\beta_{s+1}}{2}\| \Lambda(p(x^s); \lambda^{s+1},\beta_{s+1})-\lambda^{s+1}\|^2 + \frac{\beta_s^2}{2\beta_{s+1}-\beta_s}\|x^{s-1}-x^s\|^2
		\\
		\qquad\qquad\qquad+\|\lambda^{s+1}-\lambda^{s}\|\sqrt{ \left((\beta_s+\beta_{s+1})L_{h_1}+\|\beta_s \lambda_1^s-\beta_{s+1}\lambda_1^{s+1}\| \right)^2
			+\|\beta_s \lambda_2^s-\beta_{s+1}\lambda_2^{s+1}\|^2 } \end{array}.
		\end{equation}
		By~\eqref{a:etasdef}, we have
		$$
		\bE\left[H_{s+1}\left( x^{s+1}\right)-H_{s+1}^\star|\cF_s\right ]\leq {2^{-\lfloor m_{s+1}/K_{s+1}\rfloor}} \left( H_{s+1}(x^s)-H_{s+1}^\star\right).
		$$
		Then we apply Proposition~\ref{prop:main2} and obtain
		\begin{equation}
		\begin{array}{l}
		\bE\left [H_{s+1}\left( x^{s+1}\right)-H_{s+1}^\star| \cF_s \right]\leq  {2^{1-\lfloor m_{s+1}/K_{s+1}\rfloor}}  \left( H_{s}(x^s)-H_{s}^\star\right)+ {2^{-\lfloor m_{s+1}/K_{s+1}\rfloor}} M_s .
		\end{array}
		\end{equation}
		If~\eqref{eq:ms2} holds, then
		$$
		{2^{-\lfloor m_{s+1}/K_{s+1}\rfloor}} \leq \frac{\epsilon_{s+1}}{4\epsilon_s}, \enspace {2^{-\lfloor m_{s+1}/K_{s+1}\rfloor}} M_s\leq \frac{\epsilon_{s+1}}{2}.
		$$
		It follows that
		$$
		\bE\left [H_{s+1}\left( x^{s+1}\right)-H_{s+1}^\star| \cF_s\right ]\leq \frac{\epsilon_{s+1}}{2\epsilon_s} \left( H_{s}(x^s)-H_{s}^\star\right)+ \frac{\epsilon_{s+1}}{2}.
		$$
		Then~\eqref{a:dererdf2} is guaranteed by taking expectation on both sides of the last inequality.
	\end{proof}

	\begin{remark}
		All the values involved in~\eqref{eq:ms2} are computable. 
	\end{remark}
	Then an instantiation of Algorithm~\ref{ipALM} with  inner solver $ \cA$ and explicit number of inner iterations is given in Algorithm~\ref{ipALM_m}.
	\begin{algorithm}[H]
		\caption{IPALM($ \cA$)}
		\label{ipALM_m}
		\begin{algorithmic}
			\Require   $ \beta_0>0$, $\rho\in(1/2,1)$, $\eta\in (0,1)$, $m_0\in\bN_{++}$
			\Ensure $x^{-1}\in \dom(g)$, $\lambda^0\in \dom(h^{*})$ 
			\State $x^{0}\leftarrow  \cA( x^{-1}, m_0, H_0 )$
			\State $\epsilon^0\geq H_0(x^0)-H_0^\star$ 
			\For {$s=0,1, 2, \ldots$}
			\State $\lambda^{s+ 1}\leftarrow \Lambda(p(x^{s});\lambda^s,\beta_s)$
			\State $\beta_{s+1}=\rho\beta_s$
			\State $\epsilon_{s+1}=\eta\epsilon_s$
			\State choose $m_{s+1}$ to be the smallest integer satisfying~\eqref{eq:ms2}
			\State $x^{s+1}\leftarrow  \cA( x^{s}, m_{s+1}, H_{s+1} )$ 
			\EndFor
		\end{algorithmic}
	\end{algorithm}
	
	\subsection{Overall Iteration Complexity Bound}\label{sec:overall}
	To analyze the total complexity of Algorithm~\ref{ipALM_m}, we will evaluate  $\bE[m_s]$.
	The key step is to show that the expectation of the quantity $M_s$ defined in~\eqref{eq:Ms} can be bounded by some constant times $\beta_s$, provided that the primal and dual sequence $\{(x^s, \lambda^s)\}$ is bounded.
	\begin{lemma}\label{l:erer2}
		Consider Algorithm~\ref{ipALM_m}.
		If there is a constant $c>0$ such that
		\begin{align}\label{a:fgert}
			\max(\bE[\|\lambda^s-\lambda^{s+1}\|^2], \bE[\|x^{s-1}-x^{s}\|^2], \bE[\|\lambda^s\|^2])\leq c,
		\end{align}
		then
		$$
		\bE[M_{s}]\leq \beta_s \left((11+2\rho^{-2})(L^2_{h_1}+c)+(2\rho-1)^{-1}c\right).$$	\end{lemma}
	To ensure condition~\eqref{a:fgert}, we can rely on the result from  Corollary~\ref{coro:bound}.
	\begin{proposition}\label{prop:maincomplexity}
		Consider Algorithm~\ref{ipALM_m} with parameters satisfying $\eta<\rho$.
		Then,
		\begin{align}\label{a:mainnbi}
			\sum_{t=1}^s \bE[ m_t] \leq  s+ \sum_{t=1}^s K_t   \left(t\log_2 \frac{\rho}{\eta}+c_2\right)\leq \left( 1+\log_2\frac{\rho}{\eta}+c_2\right)s\sum_{t=1}^s K_t
		\end{align}
		where 
		\begin{align}\label{a:defc2}
			c_2:=\log_2\left( \frac{4}{\eta}+\frac{2\beta_0 \left((11+2\rho^{-2})(L^2_{h_1}+4c_0)+4(2\rho-1)^{-1}c_0\right)}{\epsilon_0 \eta}\right)+1
		\end{align}
		with $c_0$ is defined as in~\eqref{a:defc0}.
	\end{proposition}
	
	\begin{theorem}\label{main:theoemalgo2}
		Consider Algorithm~\ref{ipALM_m} with parameters satisfying $\eta<\rho$. If there are three constants $\varsigma\geq 0$, $\omega>0$ and $\ell>0$ such that
		\begin{align}\label{a:Ksbound}
			K_s\leq \frac{\omega}{\beta_s^\ell}+\varsigma,\enspace \forall s\geq 1.
		\end{align}
		Let $\epsilon\leq \epsilon_0$.
		Then to obtain a solution $x^s$ such that
		\begin{equation}\label{eq:eps}
		\left |\bE[F(x^s)-F^\star]\right | \leq \epsilon,\enspace \bE[ \dist(p_2(x^s),\cK)]\leq \epsilon,
		\end{equation}
		the total  expected number of calls of Algorithm $\cA$ is bounded by
		\begin{align}\label{a:wrrrrr}
			\sum_{t=0}^s \bE[m_t] \leq m_0+\frac{c_3}{\epsilon^{\ell}
			} \ln\frac{c_1}{\epsilon\rho}
		\end{align}
		where $c_1$ is defined in~\eqref{a:defc1} and
		\begin{align}\label{a:defc3}
			c_3:=\frac{1+\log_2({\rho}/{\eta})+c_2}{\ln(1/\rho)}\left( \frac{ \varsigma c^\ell_1}{\rho^\ell \ell\ln(1/\rho)}+ \frac{\omega c_1^\ell}{\beta_0^\ell(1-\rho^\ell)} \right),
		\end{align}
		with $c_2$ defined in~\eqref{a:defc2}.
	\end{theorem}
	\begin{proof}
		By Corollary~\ref{coro:them1c},~\eqref{eq:eps} holds if $$s\geq \frac{\ln(c_1/\epsilon)}{\ln(1/\rho)}.$$   Thus~\eqref{eq:eps} is true for some integer $s$ satisfying
		\begin{align}\label{a:rrrr}
			s\leq  \frac{\ln(c_1/\epsilon)}{\ln(1/\rho)}+1=\frac{\ln(c_1/(\epsilon \rho))}{\ln(1/\rho)}.
		\end{align}
		Since $\epsilon\leq \epsilon_0$, we know that $\epsilon\leq c_1$ and
		\begin{align}\label{a:trc}
			s\leq \frac{\ln(c_1/(\epsilon \rho))}{\ln(1/\rho)}=\frac{\ln(c^\ell_1/(\epsilon^\ell \rho^\ell))}{\ell\ln(1/\rho)}\leq \frac{c^\ell_1}{\epsilon^\ell \rho^\ell \ell\ln(1/\rho)},
		\end{align}
		where in the last inequality we used $\ln a\leq a$ for any $a\geq 1$.
		In view of~\eqref{a:Ksbound}, we have
		$$
		\sum_{t=1}^s K_t\leq  \varsigma s+\frac{\omega}{\beta_0^\ell} \sum_{t=1}^s   \rho^{-\ell t}\le \varsigma s+\frac{\omega \rho^{-\ell s}}{\beta_0^\ell(\rho^{-\ell}-1)}  \overset{\eqref{a:rrrr}}{\leq}
		\varsigma s+  \frac{\omega c_1^\ell}{\beta_0^\ell(1-\rho^\ell)\epsilon^\ell}
		\overset{\eqref{a:trc}}{\leq} \left( \frac{ \varsigma c^\ell_1}{\rho^\ell \ell\ln(1/\rho)}+ \frac{\omega c_1^\ell}{\beta_0^\ell(1-\rho^\ell)} \right)\frac{1}{\epsilon^\ell}
		$$
		Then we apply Proposition~\eqref{prop:maincomplexity} to obtain
		\begin{align*}
			\sum_{t=1}^s \bE[m_t]&\leq s \left( 1+\log_2({\rho}/{\eta})+c_2\right)  \left( \frac{ \varsigma c^\ell_1}{\rho^\ell \ell\ln(1/\rho)}+ \frac{\omega c_1^\ell}{\beta_0^\ell(1-\rho^\ell)} \right)\frac{1}{\epsilon^\ell} \\& \overset{\eqref{a:rrrr}}{\leq} 
			\frac{1+\log_2({\rho}/{\eta})+c_2}{\ln(1/\rho)}\left( \frac{ \varsigma c^\ell_1}{\rho^\ell \ell\ln(1/\rho)}+ \frac{\omega c_1^\ell}{\beta_0^\ell(1-\rho^\ell)} \right)\frac{1}{\epsilon^\ell}\ln\frac{c_1}{\epsilon\rho}.
		\end{align*}
	\end{proof}
	
	To facilitate the comparison of complexity of different inner solvers, hereinafter we hide the logarithm terms and those constants independent with the inner solver into the $\tilde O$ notation. We also hide the constant $\ell$ since we only compare inner solvers with the same  order $\ell$.
	\begin{coro}\label{coro:main}Under the premise of Theorem~\ref{main:theoemalgo2},  to obtain an $\epsilon$-optimal solution in the sense of~\eqref{eq:eps}, the number of calls of the inner solver $\cA$ is bounded by
		$$
		\tilde O\left(\frac{\omega+\varsigma}{\epsilon^\ell}\right)
		$$
		where the $\tilde O$ hides logarithm terms, and
		constants related to  the inner solver convergence order $\ell$ and other inner solver independent constants  $c_1, c_2$, $\rho$, $\eta$, $\beta_0$, $\epsilon_0$ and  $m_0$. 
	\end{coro}

	\subsection{Proof of Proposition~\ref{prop:main2}}\label{sec:H}
	In the following we denote
	\begin{align}\label{a:erdfgtff}
		L^\star(y,\lambda,\beta):=\min_x  L(x;y,\lambda,\beta),\enspace x^\star(y,\lambda, \beta):=\arg\min_x  L(x;y,\lambda,\beta),\enspace
		p^\star(y,\lambda,\beta):=p(x^\star(y,\lambda,\beta)).
	\end{align}
	We first state a few useful lemmas. Their proofs are mostly based on standard duality theory and can be found in Appendix~\ref{app:pse}.
	\begin{lemma}\label{l:Lxbound2}
		For any $x\in \R^n$,  $\lambda,\lambda'\in \R^d$ and $\beta,\beta'\in \R_+$ we have,
		\begin{align}\notag
			&L(x;y,\lambda,\beta)-L(x;y',\lambda',\beta')+\frac{\beta}{2}\| \Lambda(p(x);\lambda,\beta)-\lambda\|^2-\frac{\beta'}{2}\| \Lambda(p(x);\lambda',\beta')-\lambda'\|^2
			\\&\leq \< \Lambda(p(x);\lambda,\beta)-\Lambda(p(x);\lambda',\beta'), \beta'( \Lambda(p(x);\lambda',\beta')-\lambda') >+\frac{\beta}{2}\|x-y\|^2-\frac{\beta'}{2}\|x-y'\|^2,
		\end{align}
		and
		\begin{align}\notag
			&L(x;y,\lambda,\beta)-L(x;y',\lambda',\beta')+\frac{\beta}{2}\| \Lambda(p(x);\lambda,\beta)-\lambda\|^2-\frac{\beta'}{2}\| \Lambda(p(x);\lambda',\beta')-\lambda'\|^2
			\\&\geq \< \Lambda(p(x);\lambda,\beta)-\Lambda(p(x);\lambda',\beta'), \beta( \Lambda(p(x);\lambda,\beta)-\lambda) >+\frac{\beta}{2}\|x-y\|^2-\frac{\beta'}{2}\|x-y'\|^2
		\end{align}
	\end{lemma}
	\begin{lemma}\label{l:sL2} For any $x\in \R^n$ we have,
		\begin{align}\label{a:sL2}
			L(x;y,\lambda,\beta)-L^\star(y,\lambda,\beta)\geq  \frac{\beta}{2}\|x-x^\star(y,\lambda,\beta)\|^2+\frac{\beta}{2} \| \Lambda(p(x);\lambda,\beta)-
			\Lambda(p^\star(y,\lambda,\beta);\lambda,\beta)\|^2.
		\end{align}
	\end{lemma}
	\begin{lemma}\label{l::boundbetabetap}
		Let any $u, \lambda, \lambda'\in \R^d$ and $\beta,\beta'\in \R_+$.
		Condition~\ref{ass:hp1} in Assumption~\ref{ass:handp} ensures:
		\begin{align}\label{a:boundbetabetap}\|\beta(\Lambda(u;\lambda, \beta)-\lambda)-\beta'(\Lambda(u;\lambda', \beta')-\lambda')\|\leq \sqrt{ ((\beta+\beta')L_{h_1}	 + \|\beta \lambda_1-\beta'\lambda_1'\|)^2+\|\beta \lambda_2-\beta'\lambda_2'\|^2}.\end{align}
	\end{lemma}
	\begin{remark}\label{rem:hc}
		If $$
		h(u)=\left\{\begin{array}{ll}0 & \mathrm{~if~} u=b\\
		+\infty & \mathrm{~otherwise~}
		\end{array}\right.
		$$
		for some constant vector $b\in \R^d$, then
		by~\eqref{a:optimalitycondition} we have
		$$
		u-\beta(\Lambda(u;\lambda, \beta)-\lambda)=b,
		$$
		for any $u,\lambda\in\R^d$ and $\beta\geq 0$. In this special case  a refinement of Lemma~\ref{l::boundbetabetap} can be stated as follows:
		$$\|\beta(\Lambda(u;\lambda, \beta)-\lambda)-\beta'(\Lambda(u;\lambda', \beta')-\lambda')\|=0.$$
	\end{remark}
	
	\begin{lemma}\label{l:erfgrte}
		Let any $0<\beta/2<\beta'$ and any $w,w',y,y'\in \R^n$. We have
		\begin{align}\label{a:dger}
			-\frac{\beta}{2}\| w'-w\|^2+\frac{\beta}{2}\| w'-y\|^2-\frac{\beta'}{2}\|w'-y'\|^2\leq \frac{\beta}{2}\|w-y'\|^2+ \frac{\beta(2\beta'+\beta)}{2(2\beta'-\beta)}\|y-y'\|^2.
		\end{align}
	\end{lemma}
	Using the above four lemmas, we establish a relation between  $L(x; y',\lambda', \beta') -L^\star(y',\lambda',\beta')$ and $L(x; y, \lambda, \beta)-L^\star(y,\lambda,\beta)$.
	\begin{proposition}\label{prop:m2}
		For any $x,y,y'\in \R^n$,  $\lambda,\lambda'\in \R^d$ and $0<\beta/2<\beta'$, we have
		\begin{equation}\label{eq:main2}\begin{array}{ll}
		&L(x; y',\lambda', \beta') -L^\star(y',\lambda',\beta')\\
		&\leq L(x; y, \lambda, \beta)-L^\star(y,\lambda,\beta)+ \|\lambda-\lambda'\|\sqrt{ ((\beta+\beta')L_{h_1}	 + \|\beta \lambda_1-\beta'\lambda_1'\|)^2+\|\beta \lambda_2-\beta'\lambda_2'\|^2}\\&\quad+{\beta}\|\lambda-\lambda'\|^2+
		\frac{\beta-\beta'}{2}\| \Lambda(p(x); \lambda',\beta')-\lambda'\|^2+
		\frac{\beta'-\beta}{2}\| \Lambda(p^\star(y',\lambda',\beta'); \lambda,\beta)-\lambda\|^2\\&\quad+\frac{ \beta}{2} \|\Lambda(p^\star(y,\lambda,\beta);\lambda,\beta)-\Lambda(p(x);\lambda,\beta)\|^2+\frac{\beta}{2}\|x^\star(y,\lambda,\beta)-y'\|^2+ \frac{\beta(2\beta'+\beta)}{2(2\beta'-\beta)}\|y-y'\|^2\\&\quad-\frac{\beta}{2}\|x-y\|^2+\frac{\beta'}{2}\|x-y'\|^2 . \end{array}
		\end{equation}
	\end{proposition}

	\begin{proof}
		We first separate $L(x; y', \lambda', \beta') -L^\star(y',\lambda',\beta')$ into four parts:
		\begin{align*}
			&L(x; y',\lambda', \beta') -L^\star(y',\lambda',\beta')\\
			&=\underbrace{L(x; y,\lambda, \beta)-L^\star(y,\lambda,\beta)}_{\Delta_1}+\underbrace{L(x;y', \lambda', \beta')-L(x; y,\lambda, \beta)}_{\Delta_2}
			\\&+\underbrace{L(x^\star(y',\lambda',\beta');y, \lambda,\beta) -L^\star(y',\lambda',\beta')}_{\Delta_3}+\underbrace{L^\star(y,\lambda,\beta)-L(x^\star(y',\lambda',\beta'); y,\lambda,\beta)}_{\Delta_4}.
		\end{align*}
		By Lemma~\ref{l:Lxbound2}, 
		\begin{align*}
			\Delta_2\leq \enspace&
			\frac{\beta}{2}\| \Lambda(p(x);\lambda,\beta)-\lambda\|^2-\frac{\beta'}{2}\| \Lambda(p(x);\lambda',\beta')-\lambda'\|^2
			\\&+ \< \Lambda(p(x);\lambda',\beta')-\Lambda(p(x);\lambda,\beta), \beta( \Lambda(p(x);\lambda,\beta)-\lambda) >-\frac{\beta}{2}\|x-y\|^2+\frac{\beta'}{2}\|x-y'\|^2,
		\end{align*}
		and
		\begin{align*}
			\Delta_3\leq \enspace&
			\frac{\beta'}{2}\| \Lambda(p^\star(y',\lambda',\beta');\lambda',\beta')-\lambda'\|^2-\frac{\beta}{2}\| \Lambda(p^\star(y',\lambda',\beta');\lambda,\beta)-\lambda\|^2
			\\&+ \< \Lambda(p^\star(y',\lambda',\beta');\lambda,\beta)-\Lambda(p^\star(y',\lambda',\beta');\lambda',\beta'), \beta'( \Lambda(p^\star(y',\lambda',\beta');\lambda',\beta')-\lambda') >\\&+\frac{\beta}{2}\|x^\star(y',\lambda',\beta')-y\|^2-\frac{\beta'}{2}\|x^\star(y',\lambda',\beta')-y'\|^2.
		\end{align*}
		We then get
		\begin{align*}
			\Delta_2+\Delta_3 
			&\leq-\frac{\beta}{2}\| \Lambda(p(x);\lambda,\beta)-\lambda\|^2-\frac{\beta'}{2}\| \Lambda(p(x);\lambda',\beta')-\lambda'\|^2
			+ \< \Lambda(p(x);\lambda',\beta') -\lambda', \beta( \Lambda(p(x);\lambda,\beta)-\lambda) >\\
			&\quad -\frac{\beta'}{2}\| \Lambda(p^\star(y',\lambda',\beta');\lambda',\beta')-\lambda'\|^2-\frac{\beta}{2}\| \Lambda(p^\star(y',\lambda',\beta');\lambda,\beta)-\lambda\|^2
			\\&\quad+ \< \Lambda(p^\star(y',\lambda',\beta');\lambda,\beta)-\lambda, \beta'( \Lambda(p^\star(y',\lambda',\beta');\lambda',\beta')-\lambda') >
			\\&\quad+\<\lambda-\lambda',\beta'( \Lambda(p^\star(y',\lambda',\beta');\lambda',\beta')-\lambda') -\beta( \Lambda(p(x);\lambda,\beta)-\lambda)>\\&\quad-\frac{\beta}{2}\|x-y\|^2+\frac{\beta'}{2}\|x-y'\|^2+\frac{\beta}{2}\|x^\star(y',\lambda',\beta')-y\|^2-\frac{\beta'}{2}\|x^\star(y',\lambda',\beta')-y'\|^2\\
			&\leq \frac{\beta-\beta'}{2}\|\Lambda(p(x);\lambda',\beta') -\lambda'\|^2+ \frac{\beta'-\beta}{2}\|  \Lambda(p^\star(y',\lambda',\beta');\lambda,\beta)-\lambda\|^2
			\\&\quad+\<\lambda-\lambda',\beta'( \Lambda(p^\star(y',\lambda',\beta');\lambda',\beta')-\lambda') -\beta( \Lambda(p(x);\lambda,\beta)-\lambda)>\\&\quad-\frac{\beta}{2}\|x-y\|^2+\frac{\beta'}{2}\|x-y'\|^2+\frac{\beta}{2}\|x^\star(y',\lambda',\beta')-y\|^2-\frac{\beta'}{2}\|x^\star(y',\lambda',\beta')-y'\|^2,
		\end{align*}
		where the last inequality simply relies on $2\<x,y>\leq \|x\|^2+\|y\|^2$.
		Further, according to Lemma~\ref{l:sL2},
		$$
		\Delta_4 \leq -\frac{\beta}{2} \| \Lambda(p^\star(y',\lambda',\beta');\lambda,\beta)-
		\Lambda(p^\star(y,\lambda,\beta);\lambda,\beta)\|^2-\frac{\beta}{2}\| x^\star(y',\lambda',\beta')- x^\star(y,\lambda,\beta)\|^2.
		$$
		Therefore,
		\begin{align} \notag
			& \Delta_2+\Delta_3+\Delta_4- \frac{\beta-\beta'}{2}\|\Lambda(p(x);\lambda',\beta') -\lambda'\|^2- \frac{\beta'-\beta}{2}\|  \Lambda(p^\star(y',\lambda',\beta');\lambda,\beta)-\lambda\|^2 
			\\ \notag &\leq \<\lambda-\lambda',\beta'( \Lambda(p^\star(y',\lambda',\beta');\lambda',\beta')-\lambda') -\beta( \Lambda(p(x);\lambda,\beta)-\lambda)>\\\notag &\quad
			-\frac{\beta}{2} \| \Lambda(p^\star(y',\lambda',\beta');\lambda,\beta)-
			\Lambda(p^\star(y,\lambda,\beta);\lambda,\beta)\|^2-\frac{\beta}{2}\| x^\star(y',\lambda',\beta')- x^\star(y,\lambda,\beta)\|^2\\\notag &\quad-\frac{\beta}{2}\|x-y\|^2+\frac{\beta'}{2}\|x-y'\|^2+\frac{\beta}{2}\|x^\star(y',\lambda',\beta')-y\|^2-\frac{\beta'}{2}\|x^\star(y',\lambda',\beta')-y'\|^2\notag
			\\ \notag
			&=\<\lambda-\lambda',\beta'( \Lambda(p^\star(y',\lambda',\beta');\lambda',\beta')-\lambda') -\beta( \Lambda(p^\star(y',\lambda',\beta');\lambda,\beta)-\lambda)>\\ \notag &\quad+\beta\<\lambda-\lambda',\Lambda(p^\star(y,\lambda,\beta);\lambda,\beta)-\Lambda(p(x);\lambda,\beta)>+\beta\<\lambda-\lambda', \Lambda(p^\star(y',\lambda',\beta');\lambda,\beta)-\Lambda(p^\star(y,\lambda,\beta);\lambda,\beta)>\\\notag &\quad
			-\frac{\beta}{2} \| \Lambda(p^\star(y',\lambda',\beta');\lambda,\beta)-
			\Lambda(p^\star(y,\lambda,\beta);\lambda,\beta)\|^2-\frac{\beta}{2}\| x^\star(y',\lambda',\beta')- x^\star(y,\lambda,\beta)\|^2\\ \notag &\quad-\frac{\beta}{2}\|x-y\|^2+\frac{\beta'}{2}\|x-y'\|^2+\frac{\beta}{2}\|x^\star(y',\lambda',\beta')-y\|^2-\frac{\beta'}{2}\|x^\star(y',\lambda',\beta')-y'\|^2 \notag
			\\
			& \leq \|\lambda-\lambda'\|\sqrt{ ((\beta+\beta')L_{h_1} + \|\beta \lambda_1-\beta'\lambda_1'\|)^2+\|\beta \lambda_2-\beta'\lambda_2'\|^2}+{\beta}\|\lambda-\lambda'\|^2 \notag \\ \notag &\quad+\frac{ \beta}{2} \|\Lambda(p^\star(y,\lambda,\beta);\lambda,\beta)-\Lambda(p(x);\lambda,\beta)\|^2-\frac{\beta}{2}\|x^\star(y',\lambda',\beta')- x^\star(y,\lambda,\beta)\|^2\\& \quad-\frac{\beta}{2}\|x-y\|^2+\frac{\beta'}{2}\|x-y'\|^2+\frac{\beta}{2}\|x^\star(y',\lambda',\beta')-y\|^2  -\frac{\beta'}{2}\|x^\star(y',\lambda',\beta')-y'\|^2  \label{eq:dgere}
		\end{align}
		where the last inequality follows from Lemma~\ref{l::boundbetabetap} and Cauchy Schwartz inequality.
		Now we apply Lemma~\ref{l:erfgrte} with $w=x^\star(y,\lambda,\beta)$ and $w'=x^\star(y',\lambda',\beta')$ to obtain:
		\begin{align}\label{eq:aaplerfgrte}
			& -\frac{\beta}{2}\| x^\star(y',\lambda',\beta')-x^\star(y,\lambda,\beta)\|^2+\frac{\beta}{2}\| x^\star(y',\lambda',\beta')-y\|^2-\frac{\beta'}{2}\|x^\star(y',\lambda',\beta')-y'\|^2 \\&\leq \frac{\beta}{2}\|x^\star(y,\lambda,\beta)-y'\|^2+ \frac{\beta(2\beta'+\beta)}{2(2\beta'-\beta)}\|y-y'\|^2. \notag
		\end{align}
		Plugging~\eqref{eq:aaplerfgrte} into~\eqref{eq:dgere} with we derive~\eqref{eq:main2}.
	\end{proof}
	Next we give a proof for Proposition~\ref{prop:main2}.
	\begin{proof}[proof of Proposition~\ref{prop:main2}]
		We apply  Proposition~\ref{prop:m2} with  $\lambda=\lambda^s$, $\lambda'=\lambda^{s+1}$, $\beta=\beta_s$, $\beta'=\beta_{s+1}$, $x=x^s$, $y=x^{s-1}$ and $y'=x^s$ to obtain
		\begin{align*}
			&H_{s+1}(x^s) -H_{s+1}^\star\\
			&\leq H_{s}(x^s) -H_{s}^\star+ \|\lambda^s-\lambda^{s+1}\|\sqrt{ \left((\beta_s+\beta_{s+1})L_{h_1}+\|\beta_s \lambda_1^s-\beta_{s+1}\lambda_1^{s+1}\| \right)^2
				+\|\beta_s \lambda_2^s-\beta_{s+1}\lambda_2^{s+1}\|^2 }\\&\quad+{\beta_s}\|\lambda^s-\lambda^{s+1}\|^2+
			\frac{\beta_s-\beta_{s+1}}{2}\| \Lambda(p(x^s); \lambda^{s+1},\beta^{s+1})-\lambda^{s+1}\|^2\\&\quad+
			\frac{\beta_{s+1}-\beta_s}{2}\| \Lambda(p^\star(x^s,\lambda^{s+1},\beta_{s+1}); \lambda^s,\beta_s)-\lambda^s\|^2+\frac{ \beta_s}{2} \|\Lambda(p^\star(x^{s-1},\lambda^s,\beta_s);\lambda^s,\beta_s)-\Lambda(p(x^s);\lambda^{s},\beta_s)\|^2\\&\quad+\frac{\beta_s}{2}\|x^\star(x^{s-1},\lambda^s,\beta_s)-x^s\|^2+ \frac{\beta_s(2\beta_{s+1}+\beta_s)}{2(2\beta_{s+1}-\beta_s)}\|x^{s-1}-x^s\|^2-\frac{\beta_s}{2}\|x^s-x^{s-1}\|^2.
		\end{align*}
		We apply Lemma~\ref{l:sL2} with $x=x^s$, $y=x^{s-1}$, $\lambda=\lambda^s$ and $\beta=\beta_s$ and get:
		$$
		\frac{ \beta_s}{2} \|\Lambda(p^\star(x^{s-1},\lambda^s,\beta_s);\lambda^s,\beta_s)-\Lambda(p(x^s);\lambda^{s},\beta_s)\|^2+\frac{\beta_s}{2}\|x^\star(x^{s-1},\lambda^s,\beta_s)-x^s\|^2\leq H_{s}(x^s) -H_{s}^\star.
		$$
		Furthermore, since $\beta_{s+1}\leq \beta_s$ we have,
		$$ \frac{\beta_{s+1}-\beta_s}{2}\| \Lambda(p^\star(x^s,\lambda^{s+1},\beta_s); \lambda^s,\beta_s)-\lambda^s\|^2\leq 0.$$
		We then derive~\eqref{eq:mainpop2} by  the latter three bounds.
	\end{proof}
	\begin{remark}
		If $$
		h(u)=\left\{\begin{array}{ll}0 & \mathrm{~if~} u=b\\
		+\infty & \mathrm{~otherwise~}
		\end{array}\right.
		$$
		for some constant vector $b\in \R^d$, for the reason stated in Remark~\ref{rem:hc},  the number of inner iterations $m_{s+1}$ in Algorithm~\ref{ipALM_m} can be taken as  the smallest integer satisfying
		\begin{equation*}
			\begin{array}{ll}&2\epsilon_s+{{\beta_s}}\|\lambda^{s+1}-\lambda^{s}\|^2+
				\frac{\beta_s-\beta_{s+1}}{2}\| \Lambda(p(x^s); \lambda^{s+1},\beta_{s+1})-\lambda^{s+1}\|^2+
				\frac{\beta_s^2}{2\beta_{s+1}-\beta_s}\|x^{s-1}-x^s\|^2 \leq {2^{\lfloor m_{s+1}/ K_{s+1}\rfloor}} {\epsilon_{s+1}/2 } .
			\end{array}
		\end{equation*}
	\end{remark}

	\section{Inner Solvers}\label{sec:total}
	In this section we recall  some algorithms satisfying Assumption~\ref{ass:innersolver} so that they can be used as inner solvers. Note that due to space limit we do not give the explicit form of the algorithms and refer the readers to the given references for details. This section is independent with the previous sections.
	
	Consider the following convex minimization problem:
	\begin{equation}\label{eq-prob}
	\begin{array}{ll}
	G^\star:= \displaystyle\min_{x\in \R^n} & \left[G(x)\equiv \phi(x)+P(x)\right],
	\end{array}
	\end{equation}
	where $P: \R^n\rightarrow \R\cup\{+\infty\}$ is a convex, proper and closed function  and $\phi:\R^n\rightarrow \R\cup\{+\infty\}$ is a  convex function differentiable on an open set containing $\dom(P)$. For any differentiable point $x\in \dom(\phi)$ and any $y\in \dom(\phi)$, we denote by $D_{\phi}(x;y)$ the Bregman distance from $x$ to $y$ with respect to the function $\phi$:
	$$
	D_{\phi}(y;x):=\phi(y)-\phi(x)-\<\nabla \phi(x),y-x>.
	$$
	\subsection{Accelerated Proximal Gradient}\label{sec:APG}
	Assume that there is $L>0$ such that
	\begin{align}\label{a:ftwq}
		D_{\phi}(y;x)\leq \frac{L}{2}\|x-y\|^2.
	\end{align}
	In addition, assume that there is $\mu>0$  such that for any $y\in \dom(P)$ there is $y^\star\in \arg\min \{ G(y): y\in \R^n\}$   satisfying 
	$$
	G(y)-G^\star\geq \frac{\mu}{2}\|y-y^\star\|^2.
	$$
	The accelerated proximal gradient (APG) method~\cite{nesterov1983method,beck2009fista,tseng2008accelerated} can be applied to solve problem~\eqref{eq-prob}.  If $\{x^k\}$ is the output after $k$ iterations of APG starting with $x^0$ as initial solution, then
	$$
	G(x^k) -G^\star \leq \frac{1}{2} \left(G(x^0)-G^\star\right),\enspace \forall k\geq 2\sqrt{2L/\mu},
	$$
	see e.g.~\cite{NecoaraNesGli,FercoqQu17}.

	\subsection{Accelerated Randomized Coordinate Descent}\label{sec:APPROX}
	There exist some variants of APG which may be more efficient when the problem dimension is high and the objective function has  certain separability.  
	If $P$ is separable, i.e., 
	$$
	P(x)\equiv \sum_{i=1}^n P_i (x_i),
	$$
	then the randomized coordinate extension of APG, known as APPROX~\cite{FR:2013approx},  can also be applied to solve~\eqref{eq-prob}. At each iteration, APPROX only updates a randomly selected set of coordinates. For simplicity let us consider the case when one coordinate is chosen uniformly at each iteration.  In this case denote by $v_i>0$ the constant satisfying the following condition:
	\begin{align}\label{a:erqwe}
		D_{\phi}(x+h e_i; x)\leq \frac{v_i}{2}h^2,\enspace \forall  x\in \dom(P), i\in [n], x+h e_i\in \dom(P).
	\end{align}
	In addition, assume that there is $\mu>0$  such that for any $y\in \dom(P)$ there is $y^\star\in \arg\min \{ G(y): y\in \R^n\}$   satisfying 
	$$
	G(y)-G^\star\geq \frac{\mu}{2}\|y-y^\star\|^2.
	$$
	If $\{x^k\}$ is the output after $k$ iterations of APPROX starting with $x^0$ as initial solution, then
	$$
	\bE\left [G(x^k)-G^\star\right ]\leq \frac{1}{2}\left (G(x^0)-G^\star\right ),\enspace \forall k\geq 2n\sqrt{2\max_i v_i/ \mu+2},
	$$
	see e.g.~\cite{FercoqQu18}. Note that when carefully implemented APPROX could have significantly reduced per-iteration cost than its deterministic origin APG, see~\cite{FR:2013approx}. The total computational saving is more  important when number of coordinates $n$ is larger.

	\subsection{Accelerated Stochastic Variance Reduced Method}\label{sec:Katy}
	If $P$ is $\mu$-strongly convex and $\phi$ is written as a large sum of convex functions, for example when
	$$
	\phi(x)=\frac{1}{m}\sum_{j=1}^m \phi^j (x), 
	$$
	where each $ \phi^j$ is convex and there is $L_j>0$ such that
	$$D_{\phi^j} (y;x)\leq \frac{L_i}{2}\|x-y\|^2,\enspace \forall x,y \in \dom(P).$$
	Then the accelerated stochastic variance reduced  methods, known as Katyusha~\cite{Katyusha,LKatyusha}, can be applied to solve~\eqref{eq-prob}.  Katyusha combined the techniques from stochastic gradient descent, variance reduction and Nesterov's acceleration method. In particular,  at each step, Katyusha randomly select a subset $S\subset [m]$ and use $\{\nabla \phi_j(\cdot): j\in S\}$ to form a stochastic estimator of the gradient $\nabla \phi(\cdot)$.  The  convergence rate depends on  the way we choose $S$, as shown in~\cite{LKatyusha}.  We will apply L-Katyusha\footnote{L-Katyusha stands for Loopless Katyusha. The algorithm Katyusha was first proposed by Allen-Zhu~\cite{Katyusha}. The loopless variants~\cite{LSVRG, LKatyusha} have the same complexity order as the original one but has simpler implementation form and improved practical efficiency. } using nonuniform sampling with replacement~\cite{LKatyusha}. More precisely,  we use the following stochastic gradient estimator:
	$$
	\sum_{j=1}^{\tau} p_j^{-1} \nabla \phi_{\sigma_j}(\cdot),
	$$
	where  $\sigma_j$ is a random integer equal to $j$ with probability $p_j:=L_j/(L_1+\dots L_m)$. Here $\tau \in [m]$ is the batch size. We shall consider the case when $\tau \leq \sqrt{m}$, for which there is linear speedup with respect to the increasing batch size.
	In this case, if $\{x^k\}$ is the output after $k$ iterations of L-Katyusha starting with $x^0$ as initial solution, then we know from~\cite{LKatyusha} that
	\begin{align}\label{a:katy}
		\bE\left [G(x^k)-G^\star\right]\leq \frac{1}{2}\left (G(x^0)-G^\star\right ),\enspace \forall k\geq 10\max\left(m, \sqrt{(L_1+\dots+L_m)/{\mu}}\right)/\tau.
	\end{align}
	Similarly, L-Katyusha becomes more efficient than APG when  $m$ is larger. Moreover it enjoys linear speedup with increasing batch size $\tau$.

	\subsection{Bregman Proximal Gradient}\label{sec:Breg}
	In this section, we recall the Bregman proximal gradient method for solving~\eqref{eq-prob}. This algorithm is an extension of the classical proximal gradient method in the case when $\phi$ does not have a Lipschitz continuous gradient but satisfies the so-called relative smoothness condition~\cite{Bolte16,LuFreudNesterov}. The latter means the existence of a  convex function $\xi(\cdot)$ differentiable on $\dom(P)$ and $L>0$ such that
	$$
	D_{\phi}(y;x)\leq L D_\xi(y;x), \enspace \forall x,y\in \dom(P).
	$$
	In addition, assume that there is $\mu>0$ such that  for any $y\in \dom(P)$, there is $y^\star\in \arg\min_y \{ G(y):y\in \R^n\}$ satisfying
	$$
	G(y)-G^\star\geq \mu D_\xi(x;y^\star).
	$$
	Let $\{x^k\}$ be the output after $k$ iterations of the Bregman proximal gradient method  starting with $x^0$ as initial solution. 
	Then by~\cite[Theorem 3.1]{LuFreudNesterov}, we have $$
	G(x^k)-G^\star\leq \frac{1}{2}\left( G(x^0)-G^\star\right),\enspace \forall k\geq 2L/\mu.
	$$
	Note that this method requires that the following problem
	$$
	\arg\min\{P(y)+\<\nabla \phi(x),y-x>+L D_{\xi}(y;x):y \in \R^n\},
	$$
	is easily solvable for any $x\in \dom(P)$.
	
	Before we end this section, we note that the above four methods, with appropriate restart if necessary, are linearly convergent. 
	
	\section{Applications}\label{sec:rterr}
	In this section we apply Algorithm~\ref{ipALM_m} in different circumstances using the inner solvers discussed in Section~\ref{sec:total}. We denote by $\mu_g\geq 0$ the strong convexity parameter of  the function $g$.  
	Recall that the objective function to be minimized at outer iteration $s$ is:
	\begin{align}\label{ass:dehin1}
		H_s(x)\equiv f(x)+g(x)+h(p(x);\lambda^s,\beta_s)+\frac{\beta_s}{2}\|x-x^{s-1}\|^2
	\end{align}
	which can be written in the form of~\eqref{eq-prob} as follows:
	$$
	H_s(x)= \phi_s(x)+P_s(x),
	$$
	with
	\begin{align}\label{a:phisdef1}
		\phi_s(x)\equiv f(x)+h(p(x);\lambda^s,\beta_s),\enspace P_s(x) \equiv g(x)+\frac{\beta_s}{2}\|x-x^{s-1}\|^2.
	\end{align}
	Note that due to Lemma~\ref{l:rterrr}, we have
	\begin{align}
		H_s(x)-H_s^\star\geq \frac{\beta_s+\mu_g}{2}\|x-y^\star\|^2+D_f(x;y^\star),\enspace \forall x \in \dom(g),\enspace y^\star= \arg\min_y H_s(y) . \label{a:dferefff}
	\end{align}
	\subsection{Composition with Linear Functions}\label{sec:li}
	Throughout this subsection we consider the special case when $p(x)$ is a linear function.  More precisely we focus on the following problem:
	\begin{align}
		\label{prime_affine}
		\min_{x\in \mathbb{R}^n} ~~&F(x)\equiv f(x)+ g(x)+ h_1(A_1 x)\\ \notag
		&~\mathrm{s.t.} \qquad A_2x\in \cK
	\end{align}
	where $A_1\in \R^{d_1\times  n}$ and $A_2\in \R^{d_2\times n}$. Recall that in this special case condition~\ref{ass:pq2} and~\ref{ass:pq3} automatically holds. In addition we have
	\begin{align}\label{a:defphied}\phi_s(x)\equiv  f(x)+h(Ax;\lambda^s,\beta_s), \end{align}
	where $A=:\begin{pmatrix}
	A_1 \\ A_2
	\end{pmatrix}.
	$
	In view of Lemma~\ref{l:rterrr} and~\eqref{a:nablahLip}, the function $ \phi_s(\cdot)$ in~\eqref{a:defphied} is differentiable with respect to $x$ and
	\begin{align}
		\label{a:erdfgg}
		&D_{\phi_s}(y;x)\leq \frac{\|A\|^2}{2\beta_s}\|x-y\|^2+D_f(y;x),\enspace \forall x,y\in\dom(g).
	\end{align}
	We next consider three subcases based on three different assumptions on the functions $f$ and $h$.

	\subsubsection{APG as inner solver}
	In this subsection we consider the case when the function $f$ satisfies the following additional assumption.
	\begin{assumption}\label{ass:add1}
		There is $ L>0 $ such that
		$$
		D_f(y;x)\leq \frac{L}{2}\|x-y\|^2,\enspace \forall x,y\in\dom(g).
		$$
	\end{assumption}
	Under Assumption~\ref{ass:add1}, it is clear from~\eqref{a:erdfgg} that 
	$$
	D_{\phi_s}(y;x)\leq  \frac{L+\beta^{-1}_s\|A\|^2}{2}\|x-y\|^2,\enspace \forall x,y\in\dom(g).
	$$
	Together with~\eqref{a:dferefff}, we know from Section~\ref{sec:APG} that   in this case APG~\cite{nesterov1983method,beck2009fista,tseng2008accelerated} can be used as an inner solver with
	\begin{align}\label{a:KsAPG}
		K_s\leq 2\sqrt{\frac{2(L+\beta^{-1}_s\|A\|^2)}{\mu_g+\beta_s}}+1.
	\end{align}
	Then the following result follows directly from Corollary~\ref{coro:main}.
	\begin{coro}\label{coro:lAPG}
		Consider problem~\eqref{prime_affine} under Assumption~\ref{ass:handp} and~\ref{ass:add1}.  Let us apply Algorithm~\ref{ipALM_m} with APG~\cite{nesterov1983method,beck2009fista,tseng2008accelerated} as inner solver $\cA$. Then to obtain an $\epsilon$-solution in the sense of~\eqref{eq:eps},  the expected number of APG iterations is bounded by
		\begin{equation}\label{eq:APGin}\left\{ \begin{array}{ll}
		\tilde O\left(\frac{\sqrt{L\beta_0+\|A\|^2}+ \sqrt{\mu_g}}{\sqrt{\mu_g\epsilon}} \right) & \mathrm{if} ~\mu_g>0
		\\
		\tilde O\left(\frac{\sqrt{L\beta_0+\|A\|^2}}{{\epsilon}} \right) & \mathrm{if} ~\mu_g=0
		\end{array}\right.
		\end{equation}
	\end{coro}
	
	\subsubsection{Large scale structured problem}
	In this subsection we consider the following structured special case of~\eqref{prime_affine}.
	\begin{align}
		\label{prime:special1}
		\min_{x\in \mathbb{R}^n} ~~&F(x)\equiv \sum_{j=1}^{m_1} f_j(B_j x)+ \sum_{i=1}^n g_i(x^i)+ \sum_{j={m_1+1}}^{m_2}\psi_j(B_j x)\\ \notag
		&~\mathrm{s.t.} \qquad B_jx \in \cK_j,\enspace j\in\{m_2+1,\dots, m-m_2\}
	\end{align}
	Here $ \{B_j: j\in [m]\}$  are matrices/vectors of appropriate dimensions.   In addition we make the following assumption.
	\begin{assumption}\label{ass:add2}
		The functions $g_i, \psi_j$ are all convex, proper closed  and simple functions. The sets $\cK_j$ are all convex,  closed and simple sets. Moreover, for  each  $j\in [m_1]$, the function $f_j$ is convex and 
		$$
		D_{f_j}(y;x)\leq \frac{1}{2}\|x-y\|^2,\enspace \forall x,y\in \dom(g).
		$$
		For each $j\in [m_2]$, the function $\psi_j$ is Lipschitz continuous.
	\end{assumption}
	In this case, the function $\phi_s$ defined as in~\eqref{a:defphied} can be written in the form of finite sum problem:
	$$
	\phi_s(x)\equiv \frac{1}{m}\sum_{j=1}^m \phi_s^j(x).
	$$
	Here, the functions $\phi_s^j$ are such that that  for each $j\in [m_1]$,
	\begin{align}\label{a:w23}
		& D_{\phi_s^j}(x+h e_i; x)\leq \frac{m(B_j^\top B_j)_{i,i}}{2}h^2,\enspace  \forall  x\in \dom(g), i\in [n], x+h e_i\in \dom(g),
		\\\label{a:w24}
		& D_{\phi_s^j} (y;x)\leq \frac{m\|B_j\|^2}{2}\|x-y\|^2,\enspace \forall x,y \in \dom(g),
	\end{align}
	and for each $j\in \{m_1+1,\dots,m\}$,
	\begin{align}\label{a:w25}
		& D_{\phi_s^j}(x+h e_i; x)\leq \frac{m(B_j^\top B_j)_{i,i}}{2\beta_s}h^2,\enspace \forall  x\in \dom(g), i\in [n], x+h e_i\in \dom(g),
		\\ \label{a:w26}
		& D_{\phi_s^j} (y;x)\leq \frac{m\|B_j\|^2}{2\beta_s}\|x-y\|^2,\enspace \forall x,y \in \dom(g).
	\end{align}
	Combining~\eqref{a:w23} and~\eqref{a:w25} we get
	$$ D_{\phi_s}(x+h e_i; x)\leq \frac{\sum_{j=1}^{m_1}(B_j^\top B_j)_{i,i}+ \beta_s^{-1}\sum_{j=m_1+1}^{m}(B_j^\top B_j)_{i,i} }{2}h^2,\enspace  \forall  x\in \dom(g), i\in [n], x+h e_i\in \dom(g).
	$$
	In view of Section~\ref{sec:APPROX}, APPROX can be used as inner solver with
	$$
	K_s\leq 2n\sqrt{\frac{2\max_i  \left(\sum_{j=1}^{m_1}(B_j^\top B_j)_{i,i}+ \beta_s^{-1}\sum_{j=m_1+1}^{m}(B_j^\top B_j)_{i,i}\right) }{ \mu_g+\beta_s}+2}+1.
	$$
	In view of \eqref{a:w24},~\eqref{a:w26} and Section~\ref{sec:Katy}, L-Katyusha with batch size $\tau\leq \sqrt{m}$ can also be used as inner solver with
	$$
	K_s\leq {10}\max\left(m, \sqrt{\frac{m\sum_{j=1}^{m_1} \|B_j\|^2 +m \beta_s^{-1}\sum_{j={m_1+1}}^{m} \|B_j\|^2}{{\mu_g+\beta_s}}}\right)/\tau+1.
	$$
	
	\begin{coro}\label{coro:LL}
		Consider problem~\eqref{prime:special1} under Assumption~\ref{ass:handp} and~\ref{ass:add2}.  Let us apply Algorithm~\ref{ipALM_m} with restart APPROX~\cite{FercoqQu18} as inner solver $\cA$. Then to obtain an $\epsilon$-solution in the sense of~\eqref{eq:eps},  the expected number of APPROX iterations is bounded by
		\begin{equation}\label{eq:APPin}\left\{ \begin{array}{ll}
		\tilde O\left(\frac{n\sqrt{\max_i  \left(\beta_0 \sum_{j=1}^{m_1}(B_j^\top B_j)_{i,i}+\sum_{j=m_1+1}^{m}(B_j^\top B_j)_{i,i}\right)}+{n}\sqrt{\mu_g}}{\sqrt{\mu_g\epsilon}} \right) & \mathrm{if} ~\mu_g>0
		\\
		\tilde O\left(\frac{n\sqrt{\max_i  \left(\beta_0 \sum_{j=1}^{m_1}(B_j^\top B_j)_{i,i}+\sum_{j=m_1+1}^{m}(B_j^\top B_j)_{i,i}\right)}+n}{{\epsilon}} \right) & \mathrm{if} ~\mu_g=0
		\end{array}\right.
		\end{equation}
		If we apply Algorithm~\ref{ipALM_m} with L-Katyusha~\cite{LKatyusha} as inner solver $\cA$ and  mini batch size $\tau\leq \sqrt{m}$, then to obtain an $\epsilon$-solution in the sense of~\eqref{eq:eps},  the expected number of L-Katyusha iterations is bounded by
		\begin{equation}\label{eq:LKin}\left\{ \begin{array}{ll}
		\tilde O\left(\frac{ \sqrt{{m\beta_0\sum_{j=1}^{m_1} \|B_j\|^2 +m \sum_{j={m_1+1}}^{m} \|B_j\|^2}}+ m\sqrt{\mu_g}}{\tau \sqrt{\mu_g\epsilon}} \right) & \mathrm{if} ~\mu_g>0
		\\
		\tilde  O\left(\frac{ \sqrt{{m\beta_0\sum_{j=1}^{m_1} \|B_j\|^2 +m \sum_{j={m_1+1}}^{m} \|B_j\|^2}}+ m}{\tau {\epsilon}} \right)&\mathrm{if} ~\mu_g=0
		\end{array}\right.
		\end{equation}
		
	\end{coro}

	Since~\eqref{prime:special1} is a special case of~\eqref{prime_affine}, we can also use APG as inner solver and apply  Corollary~\ref{coro:lAPG}. However, note that the bounds provided by~\eqref{eq:APGin}, ~\eqref{eq:APPin} and~\eqref{eq:LKin} are not directly comparable since the iteration cost of APG, APPROX and L-Katyusha are different. 
	When carefully implemented, $n$ iterations of APPROX or $m/\tau$ iterations of L-Katyusha with mini-batch size $\tau$ has the same order of computational complexity as one iteration of APG. Indeed, $n$ iterations of APPROX or $m/\tau$ iterations of L-Katyusha is in expectation  equivalent to one full gradient evaluation (i.e., computation of the gradient of $\phi_s$), which is required in every iteration of APG. We provide in Table~\ref{tab_0} a comparison of the three inner solvers in terms of \textit{batch complexity}, i.e. the number of full gradient evaluation. To simplify we consider the case when $\|B_j\|^2=1$ for all $j\in [m]$ and let $\beta_0=1$ and let
	\begin{equation}\label{eq:matrixB}
	\cB:=\begin{pmatrix}
	B_1^\top & \cdots & B_m^\top
	\end{pmatrix}
	\begin{pmatrix}
	B_1 \\  \vdots \\ B_m
	\end{pmatrix}\in \R^{n\times n}.
	\end{equation}
	Note that 
	\begin{equation}\label{a:erddd}
	\begin{array}{l}
	\lambda_{\max}(\cB)\geq \max_i  \lambda_{\max} (B_i^\top B_i )= \max_i   \|B_i \|^2 =1,\\
	\lambda_{\max}(\cB) \geq \max_i \cB_{i,i}\geq \frac{\trace(\cB)}{n}=\frac{1}{n}\sum_{i=1}^m \trace(B_i^\top B_i)\geq \frac{1}{n} \sum_{i=1}^m \|B_i \|^2 = \frac{m}{n}
	\end{array}
	\end{equation}
	In addition, note that the bound in~\eqref{a:erddd} is conservative. Indeed, the maximal eigenvalue is often much larger than the maximal diagonal element. 
	We then draw the following conclusion from Table~\ref{tab_0}.
	\begin{enumerate}
		\item Using APPROX or L-Katyusha as inner solver yields better batch complexity bound than using APG as inner solver.
		\item When $m\gg n$, using L-Katyusha as inner solver yields better batch complexity bound than using APPROX as inner solver. 
	\end{enumerate}
	\begin{table}
		\centering
		\begin{tabular}{|c|c|}
			\hline
			Inner solver & strongly convex case ($\mu_g>0$) \\
			\hline
			APG (Corollary~\ref{coro:lAPG})& $\tilde O\left(\frac{\lambda_{\max} (\cB)}{\sqrt{  \mu_g\epsilon}} \right)$ \\
			\hline
			APPROX (Corollary~\ref{coro:LL})& $\tilde O\left(\frac{\max_i \cB_{i,i}}{\sqrt{  \mu_g\epsilon}} \right)$  \\
			\hline
			L-Katyusha (Corollary~\ref{coro:LL})& $\tilde O\left(\frac{1}{\sqrt{  \mu_g\epsilon}} \right)$ \\
			\hline 
			& non-strongly convex case ($\mu_g=0$) \\
			\hline
			APG (Corollary~\ref{coro:lAPG})& $\tilde O\left(\frac{\lambda_{\max} (\cB)}{{ \epsilon}} \right)$ \\ 
			\hline
			APPROX (Corollary~\ref{coro:LL})& $\tilde O\left(\frac{\max_i \cB_{i,i}}{{ \epsilon}} \right)$\\
			\hline
			L-Katyusha (Corollary~\ref{coro:LL}) &  $\tilde  O\left(\frac{1} {\epsilon} \right)$ \\
			\hline
		\end{tabular}
		\caption{Comparison of batch complexity bounds  of Algorithm~\ref{ipALM_m} applied on problem~\eqref{prime:special1} using different inner solvers. Here we consider the special case when $\|B_j\|^2=1$ for all $j\in [m]$ and let $\beta_0=1$.   The matrix $\cB$ is defined as in~\eqref{eq:matrixB} and $\cB_{i,i}$ denotes the $i$th diagonal element of $\cB$. 
		}
		\label{tab_0}
	\end{table}	
	\begin{remark}[Parallel Linear Speedup]\label{rem:pls}
		Note that  the batch complexity bound for L-Katyusha in Table~\ref{tab_0} is independent of the mini batch size $\tau$. 
		This means that  Algorithm~\ref{ipALM_m} with L-Katyusha as inner solver enjoys a parallel linear speedup  when $\tau\leq \sqrt{m}$. 
	\end{remark}
	
	\begin{remark}When we compare the bounds in Table~\ref{tab_0} with other related work, some additional transformation is needed due to different problem formulation. Here we provide one example of comparing the bounds of our Table~\ref{tab_0} with the complexity bound established in~\cite{li2018complexity} for one special case of problem~\eqref{prime:special1} when $m_1=0$ and $m_2=m$. Consider the following regularized empirical risk minimization model with $\mu_g>0$:
		\begin{equation}
		\label{prime:specialerm}
		\min_{x\in \mathbb{R}^n} ~~F(x)\equiv  \sum_{i=1}^n g_i(x^i)+ \frac{1}{m}\sum_{j={1}}^{m}m\psi_j(B_j x)
		\end{equation}
		which corresponds to problem (1.2) in~\cite{li2018complexity}.
		W.l.o.g. we assume that each $\psi_j$ in is $1$-Lipschitz continuous so that we know $L_{h_1}\leq \sqrt{m}$. Then by \cite[Corollary 3]{li2018complexity}, the number of iterations of Algorithm 4 in~\cite{li2018complexity} is bounded by $O\left(m\sqrt{\frac{m}{\mu_g \epsilon}}\right)$, which corresponds to a batch complexity bound $O\left(\sqrt{\frac{m}{\mu_g \epsilon}}\right)$.
		The $\tilde O$ in Table~\ref{tab_0} hides the constant $c_3$ defined in~\eqref{a:defc3} which is proportional to $\sqrt{c_1}$ when $\mu_g>0$. Recall the definition of $c_1$ in~\eqref{a:defc1}, which is bounded by $O(L_{h_1}^2)=O(m)$.
		Hence the batch complexity bound of our Algorithm~\ref{ipALM_m} with L-Katyusha as inner solver for problem~\eqref{prime:specialerm} is $\tilde O\left(\sqrt{\frac{m}{\mu_g \epsilon}}\right)$, which differs from the bound of~\cite{li2018complexity} by a logarithm term. Nevertheless, note that our Algorithm~\ref{ipALM_m} with L-Katyusha
		can enjoy a  linear speedup up to $\tau\leq \sqrt{m}$ if parallel implementation is used, see Remark~\ref{rem:pls}.
	\end{remark}

	\subsubsection{Bregman proximal gradient as inner solver}\label{sec:LBPG}
	In this subsection we consider the case when $f$  is relatively smooth.
	\begin{assumption}\label{ass:rt4}There is  a  convex function $\xi$ differentiable on an open set containing $\dom(g)$ and $L>\mu>0$ such that
		$$
		\mu D_\xi(y;x)   \leq D_f(y;x )\leq L D_\xi(y;x), \enspace \forall x,y\in \dom(g).
		$$
		Moreover, for any $\alpha, \beta>0$, $x\in \dom(\xi)$ and $x'\in \R^n$ the problem 
		$$
		\min_y\left\{g(y)+\frac{\beta}{2}\|y-x'\|^2+\alpha D_\xi(y;x)\right\}
		$$
		is easily solvable.
	\end{assumption}
	In this case, by~\eqref{a:erdfgg} we know that
	\begin{align*}
		D_{\phi_s}(y;x)&\leq \frac{\beta^{-1}_s\|A\|^2}{2}\|x-y\|^2+L D_\xi(y;x)\\&\leq 
		\max\left(\beta^{-1}_s\|A\|^2, L \right) \left(\frac{1}{2}\|x-y\|^2+D_{\xi}(y;x)\right)\enspace \forall x,y\in\dom(g).
	\end{align*}
	Moreover, by~\eqref{a:dferefff} we know that
	\begin{align*}
		H_s(x)-H_s^\star&\geq \frac{\beta_s+\mu_g}{2}\|x-y^\star\|^2+\mu D_\xi(x;y^\star)\\&\geq \min\left(\beta_s+\mu_g,\mu\right) \left(\frac{1}{2}\|x-y^\star\|^2+D_{\xi}(x;y^\star)\right),\enspace \forall x \in \dom(g),\enspace y^\star\in \arg\min_y H_s(y) . 
	\end{align*}
	Therefore, based on Section~\ref{sec:Breg}, the Bregman proximal gradient can be used as an inner solver with 
	$$
	K_s\leq \frac{2\max\left(\beta^{-1}_s\|A\|^2, L \right)}{ \min\left(\beta_s+\mu_g,\mu\right)}+ 1.
	$$
	\begin{coro}\label{coro:BPG}
		Consider problem~\eqref{prime_affine} under Assumption~\ref{ass:handp} and~\ref{ass:rt4}.  Let us apply Algorithm~\ref{ipALM_m} with Bregman proximal gradient~\cite{Bolte16,LuFreudNesterov} as inner solver $\cA$. Then to obtain an $\epsilon$-solution in the sense of~\eqref{eq:eps},  the expected number of Bregman proximal gradient iterations is bounded by
		$$\left\{ \begin{array}{ll}
		\tilde O\left(\frac{\max\left( \|A\|^2, L\beta_0\right)+ \min(\mu_g,\mu)}{ \min(\mu_g,\mu)\epsilon} \right) & \mathrm{if} ~\min(\mu_g,\mu)>0
		\\
		\tilde  O\left(\frac{ \max\left( \|A\|^2, L\beta_0\right) }{{\epsilon^2}} \right)&\mathrm{if} ~\min(\mu_g,\mu)=0
		\end{array}\right.
		$$
	\end{coro}

	\subsection{Composition with Nonlinear Functions}\label{sec:nonlinear}
	In this section we consider the general case when $p(x)$ is possibly nonlinear.
	\begin{assumption}\label{ass:ewggg}
		There is $M_{p_2}>0$, $M_{\nabla p}>0$ , $L>0$ and $L_{\nabla p}>0$ such that
		\begin{align}\label{a:ererrrrewr}
			&\|p_2(x)\|\leq M_{p_2},\enspace \forall x\in \dom(g),\\
			&\|\nabla p(x)\|\leq M_{\nabla p}, \enspace \forall x\in \dom(g),\\
			&\|\nabla p(x)-\nabla p(y)\|\leq L_{\nabla p}\|x-y\|,\enspace \forall x,y\in \dom(g)
		\end{align}
	\end{assumption}
	Note that the same type of assumptions was used in~\cite[Section 2.4]{lu18}. In particular as mentioned in~\cite{lu18}, if the domain of $g$ is compact then Assumption~\ref{ass:ewggg} holds. Assumption~\ref{ass:ewggg}  is made in order to obtain the smoothness of the function $\nabla \phi_s$.
	Recall from~\eqref{a:phisdef1} and Lemma~\ref{l:rterrr} that
	$$
	\nabla \phi_s(x)=\nabla f(x)+\nabla p(x) \Lambda(p(x);\lambda^s,\beta_s).
	$$
	\begin{lemma}\label{l:ersdgg}
		Under Assumption~\ref{ass:ewggg},
		$$
		\|\nabla \phi_s(x)-\nabla \phi_s(y)\|\leq  \|\nabla f(x)-\nabla f(y)\|+  L_{\nabla p} \left(L_{h_1}+ \beta_s^{-1}d_s\right)+ M^2_{\nabla p}\beta_s^{-1}\|x-y\|,\enspace \forall x,y \in\dom(g),
		$$
		where
		$$
		d_s:= \max_y \min_x \{ \| x-y\|: x\in \cK, \|y\|\leq M_{p_2}+\beta_s\|\lambda_2^s\| \} <+\infty.
		$$
	\end{lemma}
	Further, 
	let Assumption~\ref{ass:add1} hold.
	Then Lemma~\ref{l:ersdgg} implies
	$$
	D_{\phi_s}(y;x)\leq \frac{\left( L+ L_{\nabla p} \left(L_{h_1}+ \beta_s^{-1}d_s\right)+ M^2_{\nabla p}\beta_s^{-1}\right)}{2}\|x-y\|^2,\enspace \forall x,y \in\dom(g).
	$$
	Together with~\eqref{a:dferefff}, we know from Section~\ref{sec:APG} that   in this case APG~\cite{nesterov1983method,beck2009fista,tseng2008accelerated} can be used as an inner solver with
	$$
	K_s\leq 2\sqrt{\frac{2\left( L+ L_{\nabla p} \left(L_{h_1}+ \beta_s^{-1}d_s\right)+ M^2_{\nabla p}\beta_s^{-1}\right)}{\mu_g+\beta_s}}+1.
	$$
	\begin{coro}\label{coro:ggt}
		Consider problem~\eqref{prime} under Assumption~\ref{ass:handp},~\ref{ass:add1} and~\ref{ass:ewggg}.  Let us apply Algorithm~\ref{ipALM_m} with  APG~\cite{nesterov1983method,beck2009fista,tseng2008accelerated} as inner solver $\cA$. Then to obtain an $\epsilon$-solution in the sense of~\eqref{eq:eps},  the expected number of APG iterations is bounded by
		$$\left\{ \begin{array}{ll}
		\tilde O\left(\frac{\sqrt{ L\beta_0+ L_{\nabla p} \left(L_{h_1}\beta_0+ d_s\right)+ M^2_{\nabla p}}+ \sqrt{\mu_g}}{{\sqrt{\mu_g\epsilon}}} \right) & \mathrm{if} ~ \mu_g>0
		\\
		\tilde  O\left(\frac{\sqrt{ L\beta_0+ L_{\nabla p} \left(L_{h_1}\beta_0+ d_s\right)+ M^2_{\nabla p}}}{{\epsilon}} \right) &\mathrm{if} ~\mu_g=0
		\end{array}\right.
		$$
	\end{coro}
	\begin{remark}
		Corollary~\ref{coro:ggt} recovers Corollary~\ref{coro:lAPG} as a special case with $L_{\nabla p}=0$ and $M_{\nabla p}=\|A\|$. 
	\end{remark}
	Similarly,  we could consider the large-scale structured problem as~\eqref{prime:special1} but with nonlinear composite terms, or the relatively smooth assumption as in Section~\ref{sec:LBPG} instead of Assumption~\ref{ass:add1},.  The same order of iteration complexity bound as Corollary~\ref{coro:LL} and~\ref{coro:BPG} can be derived for the nonlinear composite case under Assumption~\ref{ass:ewggg}.
	\section{Further Discussion}\label{sec::dis}
	\subsection{Efficient Inner Problem Stopping Criteria}
	In Algorithm $\ref{ipALM_m}$, 
	we provide an upper bound on the number of inner iterations $m_s$ needed in order to obtain an solution $x^s$
	such that
	$$
	\bE[H_{s}(x^s)-H_s^\star]\leq \epsilon_s.
	$$
	In some cases,  it is possible to have a computable upper bound $U_s(x^s)$ such that $U_s(x^s)\geq H_{s}(x^s)-H_s^\star$. Then we can check the value of  $U_s(x^s)$ and stop the inner solve either when $U_s(x^s)\leq \epsilon_s$ or when the number of inner iterations exceeds $m_s$. Note that the solution $x^s$ obtained in this way satisfies
	$\bE[H_{s}(x^s)-H_s^\star]\leq 2\epsilon_s$, which is equivalent to a change from $\epsilon_0$ to $2\epsilon_0$ in the previous analysis and hence all the previous complexity bounds apply.   
	In particular, in the case of structured problem~\eqref{prime:special1}, the inner problem takes the following form
	\begin{equation}
	\label{prime:specialerm44}
	\min_{x\in \mathbb{R}^n}    \left[\digamma(x)\equiv \Psi(x)+ \sum_{j={1}}^{m} \Phi_j (B_j x)\right],
	\end{equation}
	to which we can associate the following dual problem:
	\begin{equation}
	\label{prime:fual44}
	\max_{y\in \mathbb{R}^m}    \left[ \mathcal{D}(y)\equiv -\Psi^*(-B^\top y)- \sum_{j={1}}^{m} \Phi^*_j (y_j)\right],
	\end{equation}
	where $B=:\begin{pmatrix}
	B_1 \\ \vdots\\ B_m
	\end{pmatrix}$.
	In this case a computable upper bound is given by $\digamma(x^s)-\mathcal{D}(y^s)$
	where $y^s$ is a dual feasible solution constructed from $x^s$.

	\subsection{KKT Solution}\label{sec:kkt}
	The convergence of ALM can also be measured through the KKT residual. 
	Recall that a solution is said to be an $\epsilon$-KKT solution if there exists $(u,v)\in \partial L(x,\lambda)$ such that $\|u\|\leq \epsilon$ and $\|v\|\leq \epsilon$, see e.g.~\cite{lu18}. Due to the possible randomness of the iterates in our algorithm, we shall measure the expected distance of the partial gradient of the Lagrangian to 0.  
	\begin{algorithm}[H]
		\caption{IPALM\_KKT($ \cA$)}
		\label{ipALM_m_kkt}
		\begin{algorithmic}[1]
			\Require   $ \beta_0>0$, $\rho\in(1/2,1)$, $\eta\in (0,\rho^3]$, $m_0\in\bN_{++}$, $\{L_s\}_{s\geq 0}$ satisfying~\eqref{a:etfggtt}
			\Ensure $x^{-1}\in \dom(g)$, $\lambda^0\in \dom(h^*)$ 
			\State $\tilde x^{0}\leftarrow  \cA( x^{-1}, m_0, H_0 )$
			\State $\epsilon^0\geq H_0(x^0)-H_0^*$ 
			\For {$s=0,1, 2, \ldots$}
			\State $ x^{s}\leftarrow\arg\min_{y\in \R^n} \left\{\< \nabla \phi_s (\tilde x^{s}), y-\tilde x^{s}>+\frac{L_{s}}{2} \|y-\tilde x^{s}\|^2+\frac{\beta_s}{2}\|y-x^{s-1}\|^2+ g(y) \right\}$
			\State $\lambda^{s+ 1}\leftarrow \Lambda(p(x^{s});\lambda^s,\beta_s)$
			\State $\beta_{s+1}=\rho\beta_s$
			\State $\epsilon_{s+1}=\eta\epsilon_s$
			\State choose $m_{s+1}$ to be the smallest integer satisfying~\eqref{eq:ms2}
			\State $\tilde x^{s+1}\leftarrow  \cA( x^{s}, m_{s+1}, H_{s+1} )$ 
			\EndFor
		\end{algorithmic}
	\end{algorithm}
	
	For simplicity we restrict the discussion for the case when for any outer iteration $s$ there is a constant $L_s>0$ such that
	\begin{align}\label{a:etfggtt}
		D_{\phi_s}(y;x)\leq \frac{L_s}{2}\|x-y\|^2,\enspace \forall x,y\in \dom(g).
	\end{align}
	We modify slightly Algorithm~\ref{ipALM_m} by adding one additional proximal gradient step (Line 5 in Algorithm~\ref{ipALM_m_kkt}) into each outer iteration. In addition, in Algorithm~\ref{ipALM_m_kkt} we require $\eta$ to be smaller than $\rho^3$.      	Since the proximal gradient step is guaranteed to decrease the objective value, we have
	\begin{align}\label{a:rttddd}
		\bE[H_s(x^s)-H_s^\star]\leq \bE[H_s(\tilde x^s)-H_s^\star]\leq \epsilon_s,\enspace \forall s\geq 0.
	\end{align}
	Hence Algorithm~\ref{ipALM_m_kkt} falls into the class of Algorithm~\ref{ipALM} and all the results in Section~\ref{sec:prox_ALM} can be applied. Moreover,
	in analogue to Theorem~\ref{rock2}, we have the following  bounds for the KKT residual.
	\begin{theorem}\label{thm:kkt2} Consider Algorithm~\ref{ipALM_m_kkt}.
		For any $s\geq 0$ we have
		\begin{align}\label{a:kkt1}
			&	\dist(0,\partial_x L(x^s, \lambda^{s+1})) \leq  \sqrt{16 L_{s}\left(H_s(\tilde x^s)-H_s^\star\right)+2\beta_s^2\| x^s-x^{s-1}\|^2}\enspace,\\ \label{a:kkt2}
			&\dist(0,\partial_{\lambda} L(x^s,  \lambda^{s+1})) \leq  \beta_s \| \lambda^{s+1}-\lambda^s\|.
		\end{align}
	\end{theorem}

	\begin{coro}\label{coro:kkt}
		Consider Algorithm~\ref{ipALM_m_kkt}.
		Assume that there is $\gamma>0$ such that $L_s\leq \gamma \beta_s^{-1}$. Then to obtain a solution such that
		\begin{align}\label{a:tdff}
			\bE\left[\dist(0,\partial_x L(x^s, \lambda^{s+1})) \right]\leq \epsilon,\enspace \bE\left[\dist(0,\partial_{\lambda}L(x^s, \lambda^{s+1})) \right]\leq \epsilon
		\end{align}
		it suffices to run Algorithm~\ref{ipALM_m_kkt} for 
		\begin{align}\label{a:kktsbound}
			s\geq \frac{\ln(c_4/\epsilon)}{\ln( 1/\rho)}
		\end{align}
		number of outer iterations where 
		$$
		c_4:=\max\left(\sqrt{16\gamma\epsilon_0/\beta_0+8c_0\beta_0}, \beta_0\sqrt{c_0}\right).
		$$
	\end{coro}
	Note that the outer iteration bound~\eqref{a:kktsbound} for the KKT convergence~\eqref{a:tdff} only differs from the bound for the objective value convergence~\eqref{a:sbound} by a constant in the logarithm term.
	For each outer iteration, Algorithm~\ref{ipALM_m_kkt} has one more proximal gradient step to execute than Algorithm~\ref{ipALM_m} and this will only add a term with logarithm dependence with respect to $\epsilon$ into the total complexity bound. In particular, we can derive $\tilde O(1/\epsilon)$ complexity bound to obtain $\epsilon$-KKT convergence in the sense of~\eqref{a:tdff}, and  $\tilde O(1/\sqrt{\epsilon})$ if the function $g$ is strongly convex. For brevity we omit the details which are highly similar to Section~\ref{sec:rterr}.

	\subsection{Bounded Primal and Dual Domain}\label{sec:bound}
	The bound  $\tilde O(1/\epsilon^\ell)$ can be improved to  $O(1/\epsilon^\ell)$ if both the primal and dual domain are bounded.  Indeed, we require $\eta<\rho$ in Algorithm~\ref{ipALM_m} to ensure the boundedness (in expectation) of the sequence $\{(x^s,\lambda^{s})\}$.  If the domain of $g$ is bounded and there is no constraint, i.e., $\cK= \R^{d_2}$, then $\{(x^s,\lambda^{s})\}$ of Algorithm~\ref{ipALM} is bounded for any choice of $\{\epsilon_s\}$ and $\{\beta_s\}$. In this case we can let $\eta=\rho$ and the bound in~\eqref{a:mainnbi} can be improved to
	$$
	\sum_{t=1}^s \bE[m_t]\leq s+c_2\sum_{t=1}^s K_t.
	$$
	Consequently, the bound in~\eqref{a:wrrrrr} can be improved to
	$$
	\sum_{t=0}^s\bE[m_t]\leq m_0+ \frac{c^\ell_1}{\epsilon^\ell \rho^\ell \ell\ln(1/\rho)}+ c_2\left( \frac{ \varsigma c^\ell_1}{\rho^\ell \ell\ln(1/\rho)}+ \frac{\omega c_1^\ell}{\beta_0^\ell(1-\rho^\ell)} \right)\frac{1}{\epsilon^\ell},
	$$
	and we get the $O(1/\epsilon^\ell)$ iteration complexity bound for an $\epsilon$-solution in the sense of~\eqref{eq:eps}. However, for the $\epsilon$-KKT solution in the sense of~\eqref{a:tdff} we still only have $\tilde O(1/\epsilon^\ell)$ iteration complexity bound.
	\section{Numerical Experiments}\label{sec::num}
	We will test the performance of Algorithm~\ref{ipALM_m} with APPROX and L-Katyusha as inner solver, 
	which are referred to as IPALM-APPROX and IPALM-Katyusha. We mainly compare with first-order primal dual solvers
	ASGARG-DL \cite{tran18ada} and SMART-CD \cite{ala17}. 
	Note that  comparison with linearized ADMM~\cite{cham11} and ASGARD \cite{tran18} are not included as they were compared with ASGARG-DL in~\cite{tran18ada}.  
	Since all the three algorithms depends on the choice of $\beta_0$, we test $\beta_0\in\{10^{-2}, 10^{-1}, 1,10,100\}$ and choose the best result to compare. (Note that  the problem data used are all scaled so that the  row vectors all have norm 1.) We also run CVX  so as to obtain a good approximation of the optimal value $F^\star$, which is needed in the computation of the error term:
	\begin{align}\label{a:eraaa}
		\log_{10}\left|\frac{F(x)-F^\star}{F^\star}\right|.
	\end{align}
	However, due to the large-scale problem that we solve, CVX  may return inaccurate solution or even fail. 
	To solve the issue on unknown $F^\star$, note that either CVX or our algorithm can provide a lower bound $F_l$ and an upper bound $F_u$ so that $F^\star\in [F_{l}, F_{u}]$.  Define $$\epsilon_{c}:= (F_{u}- F_{l})/F_{l}$$
	as the confidence error level.  Then for any $x$ such that $F(x)=(1+ \epsilon)F_{u} $ for some $1> \epsilon> \epsilon_c$, we have
	$$\epsilon= \frac{F(x)- F_{u}}{F_{u}}\le \frac{F(x)- F^\star}{F^\star}\le \frac{F(x)- F_{l}}{F_{l}}= \epsilon_{c}+ (1+ \epsilon_{c})\epsilon< 3\epsilon.$$
	So we use $ \frac{F(x)- F_{u}}{F_{u}} $ as an approximation of $\frac{F(x)- F^\star}{F^\star}$ for those $x$ such that $F(x)> (1+ \epsilon_c)F_{u}$.

	\subsection{Least Absolute Deviation }
	The first problem we solve is of the form:
	\begin{align*}
		\min_{x\in\mathbb{R}^n} \norm{Ax- b}_1+ \lambda \norm{x}_1
	\end{align*}
	which is also know as Least Absolute Deviation (LAD) problem~\cite{wang2007robust}. We use training data of three different datasets from libsvm \cite{chang2011libsvm} as $A$ and modify $b$ such that $Ax= b$ has a sparse solution. We set $\lambda= 0.01$. The details about the datasets are given in Table $\ref{tab_1}$. 
	The result is shown in Figure $\ref{fig_1}$. We also compare the time of CVX with IPALM-APPROX to get a mid-level accurate solution in Table $\ref{tab_2}$.
	
	As we can see form Figure $\ref{fig_1}$, IPALM-APPROX has the best performance after accuracy $10^{-3}$. ASGARD-DL works with full dimensional variables and therefore has slow convergence in time. SMART-CD has similar performance as IPALM-APPROX but tends to be slower for obtaining more accurate solution. From Table $\ref{tab_2}$ we can see to get a mid-level accurate solution, IPALM-APPROX significantly outperforms CVX for these three datasets.	\begin{table}[H]
		\centering
		\begin{tabular}{l|r|r}
			\hline
			Dataset & Training size ($m$) & Number of features ($n$) \\
			\hline
			news20scale & 15,935 & 62,061 \\
			rcv1 & 20,242 & 47,236 \\
			rcv1mc & 15,564 & 47,236\\
			\hline
		\end{tabular}
		\caption{Datasets from libsvm}
		\label{tab_1}
	\end{table}
	
	\begin{figure}
		\centering
		\begin{subfigure}[t]{0.32\linewidth}
			\centering
			\includegraphics[scale=0.5]{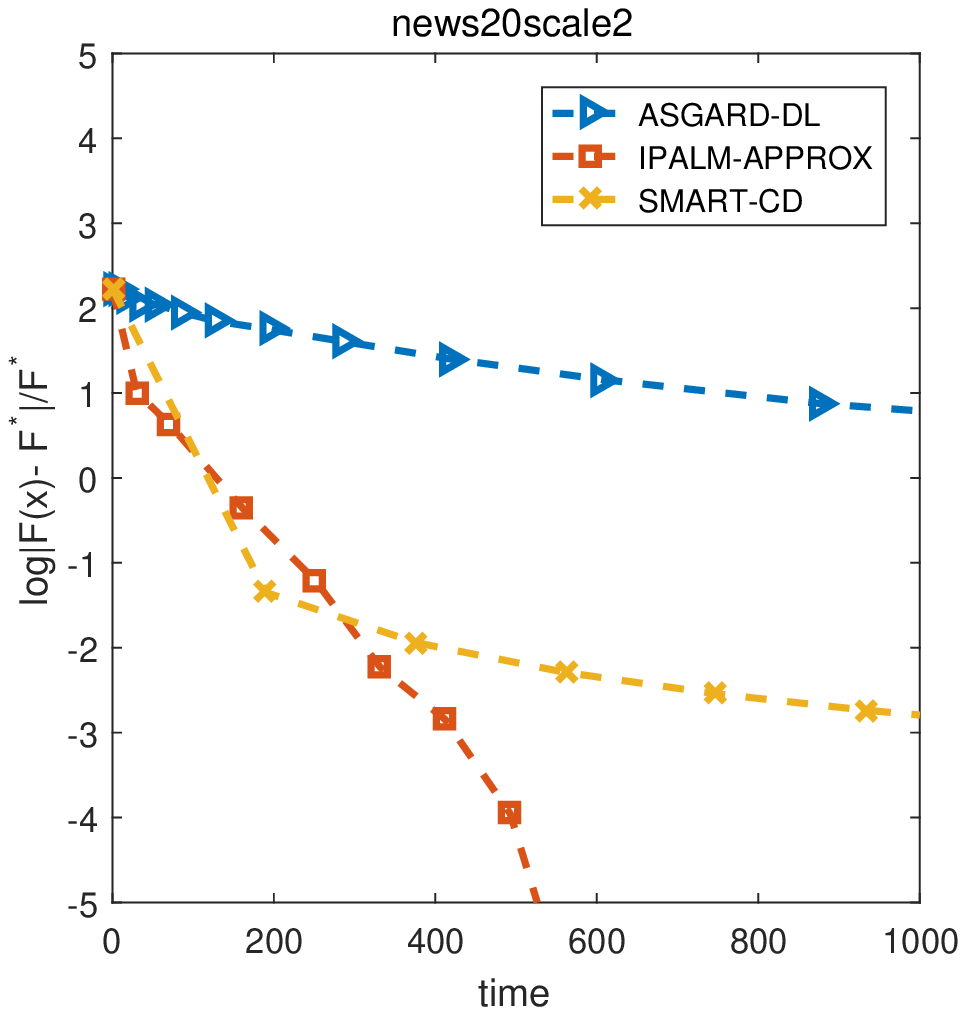}
			\caption{news20scale}
		\end{subfigure}
		\begin{subfigure}[t]{0.32\linewidth}
			\centering
			\includegraphics[scale=0.5]{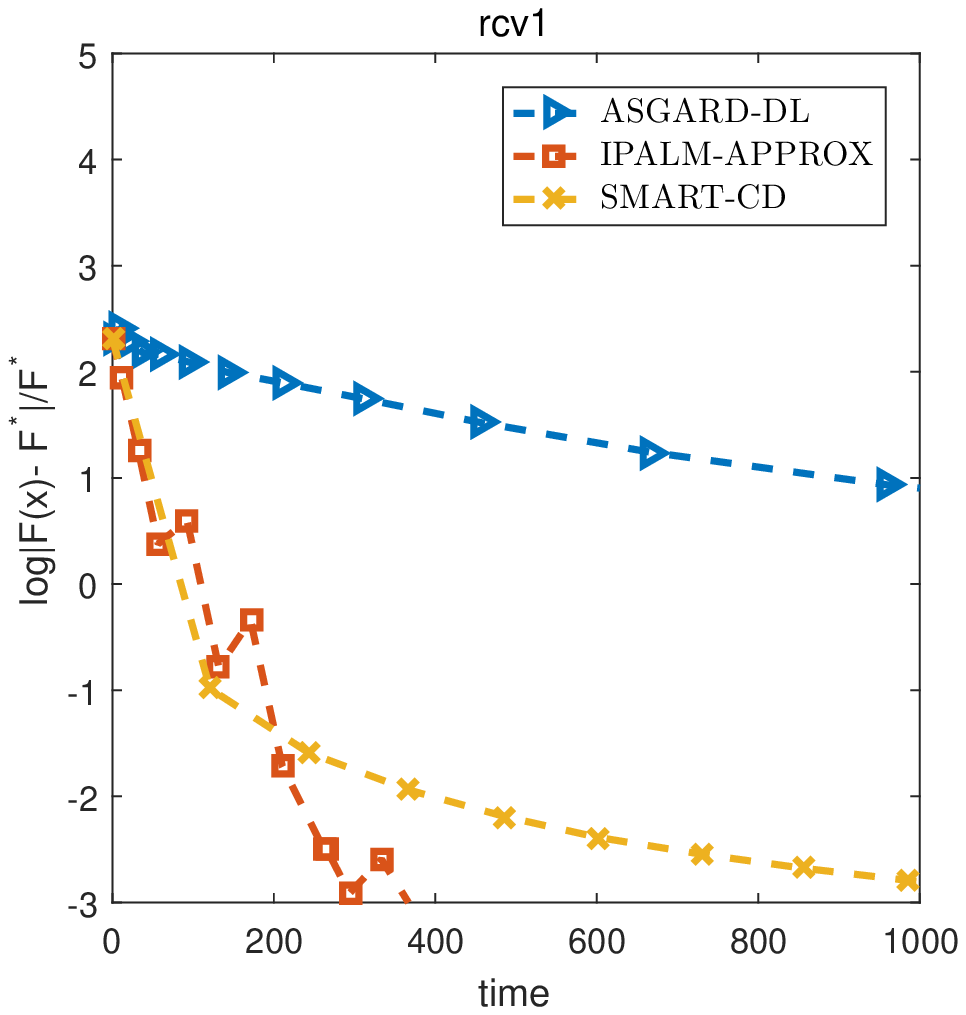}
			\caption{rcv1}
		\end{subfigure}
		\begin{subfigure}[t]{0.32\linewidth}
			\centering
			\includegraphics[scale=0.5]{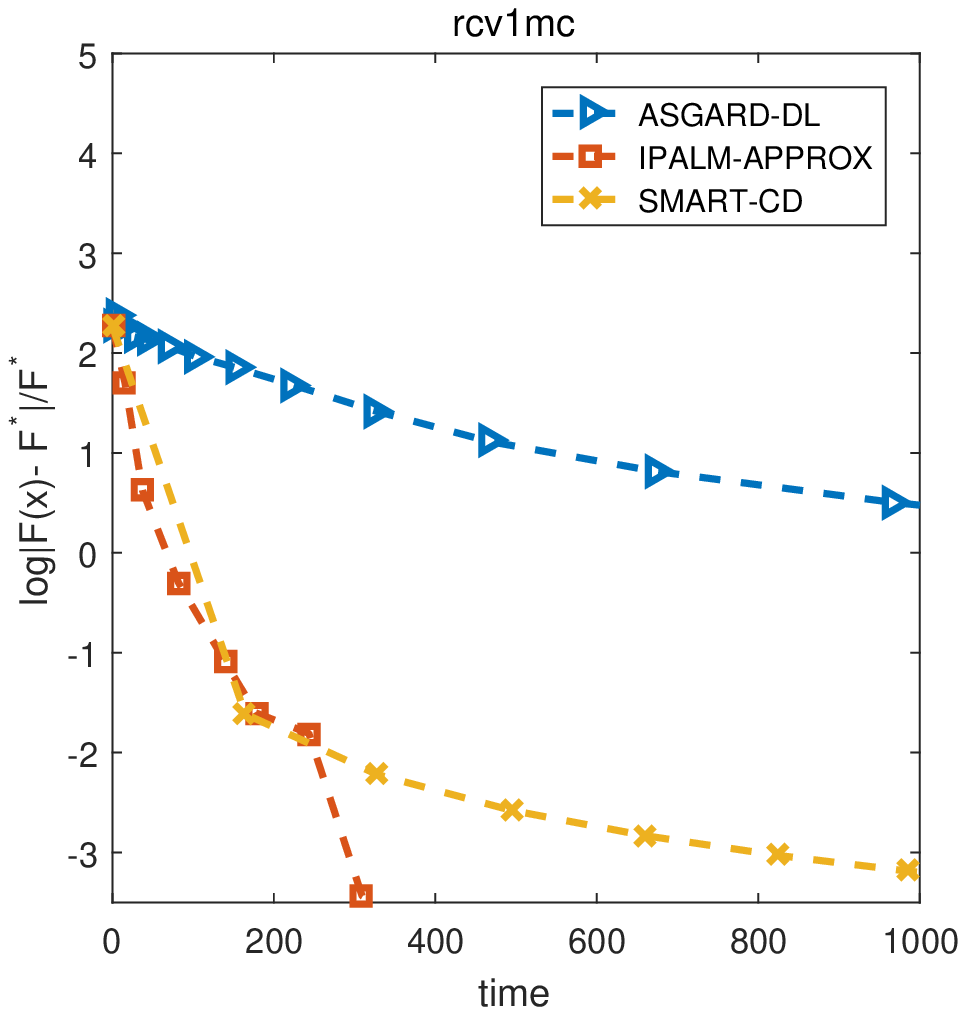}
			\caption{rcv1mc}
		\end{subfigure}
		\caption{Comparison of three algorithms for LAD problem on three datasets. The $x$-axis is time and $y$-axis is $\log((F(x)- F^\star)/F^\star)$. Here we use the result of CVX as an approximation of  $F^\star$.}
		\label{fig_1}
	\end{figure}
	
	\begin{table}[H]
		\centering
		\begin{tabular}{l|l|r|r}
			\hline
			Dataset & accuracy~\eqref{a:eraaa} & CVX time & IPALM-APPROX time \\
			\hline
			news20scale & $10^{-5}$ & $\sim$4200s & $\sim 500$s \\
			rcv1 & $10^{-3}$ & $\sim$4000s & $\sim$400s \\
			rcv1mc & $10^{-3}$ & $\sim$1800s & $\sim$300s \\
			\hline
		\end{tabular}
		\caption{Running time of CVX and IPALM-APPROX for Least absolute deviation problem on three datasets.}
		\label{tab_2}
	\end{table}
	
	\subsection{Basis Pursuit}
	The second problem we solve is of the form:
	\begin{align*}
		\min_{x\in\mathbb{R}^n} \quad&\norm{x}_1 \\
		s.t. \quad& Ax= b
	\end{align*}
	which is known as basis pursuit problem~\cite{ChenDonohoSaubders}.
	The datasets used are shown in Table $\ref{tab_1}$ and we modify $b$ for each dataset to make sure that the problem is feasible.
	The results are shown in Figure~\ref{fig_2} for the objective value gap and in  Figure~\ref{fig_2_inf} for the infeasibility gap. We also compare the time of CVX with IPALM-APPROX to get a mid-level accurate solution in Table $\ref{tab_3}$.
	
	As we can see from Figure~\ref{fig_2} and \ref{fig_2_inf}, IPALM-APPROX works well both in objective value and feasibility. Since SMART-CD reduces $\beta$ much faster than IPALM-APPROX, it has fast convergence at the beginning, but small $\beta$ leading to small stepsize and slow convergence in objective value for high accuracy. From Table~\ref{tab_3}, we see the difference between IPALM-APPROX and CVX if only medium accuracy is required.	\begin{figure}[!h]
		\centering
		\begin{subfigure}[t]{0.3\linewidth}
			\centering
			\includegraphics[scale=0.5]{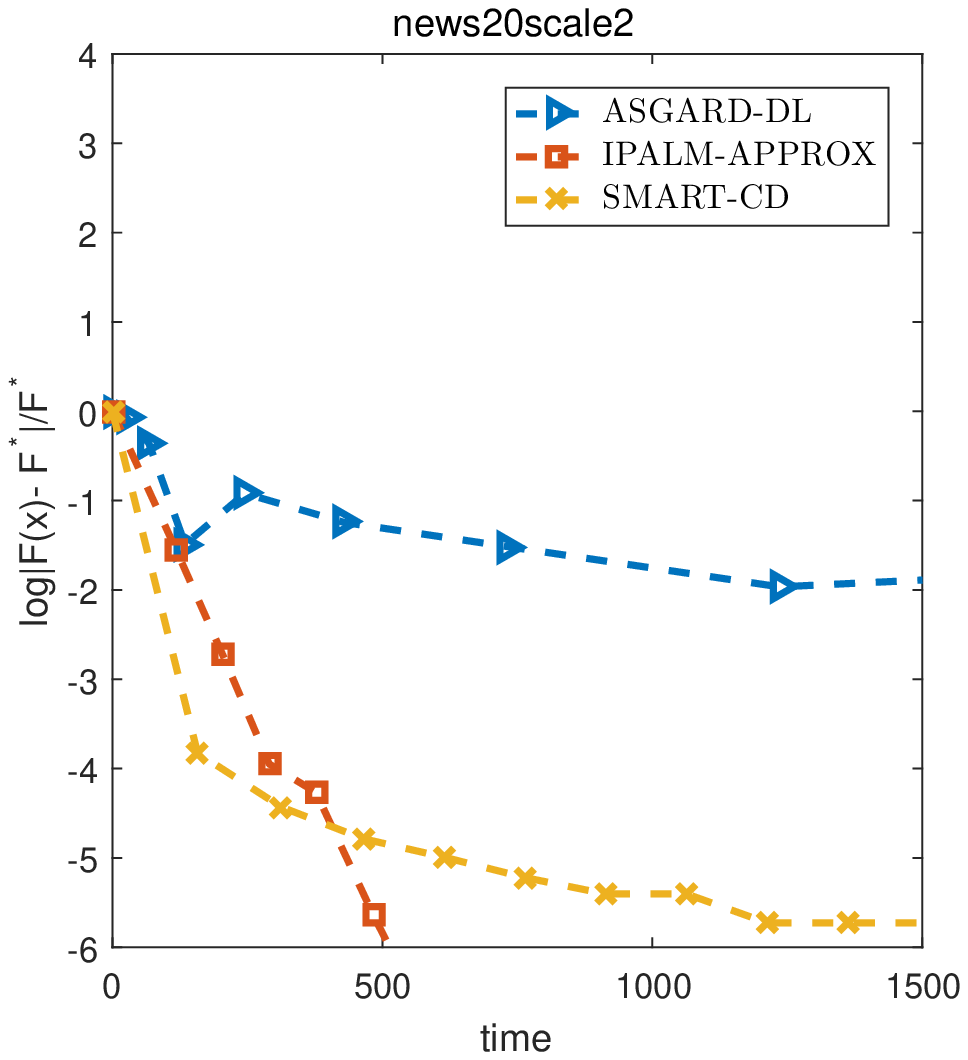}
			\caption{news20scale}
		\end{subfigure}
		\begin{subfigure}[t]{0.3\linewidth}
			\centering
			\includegraphics[scale=0.5]{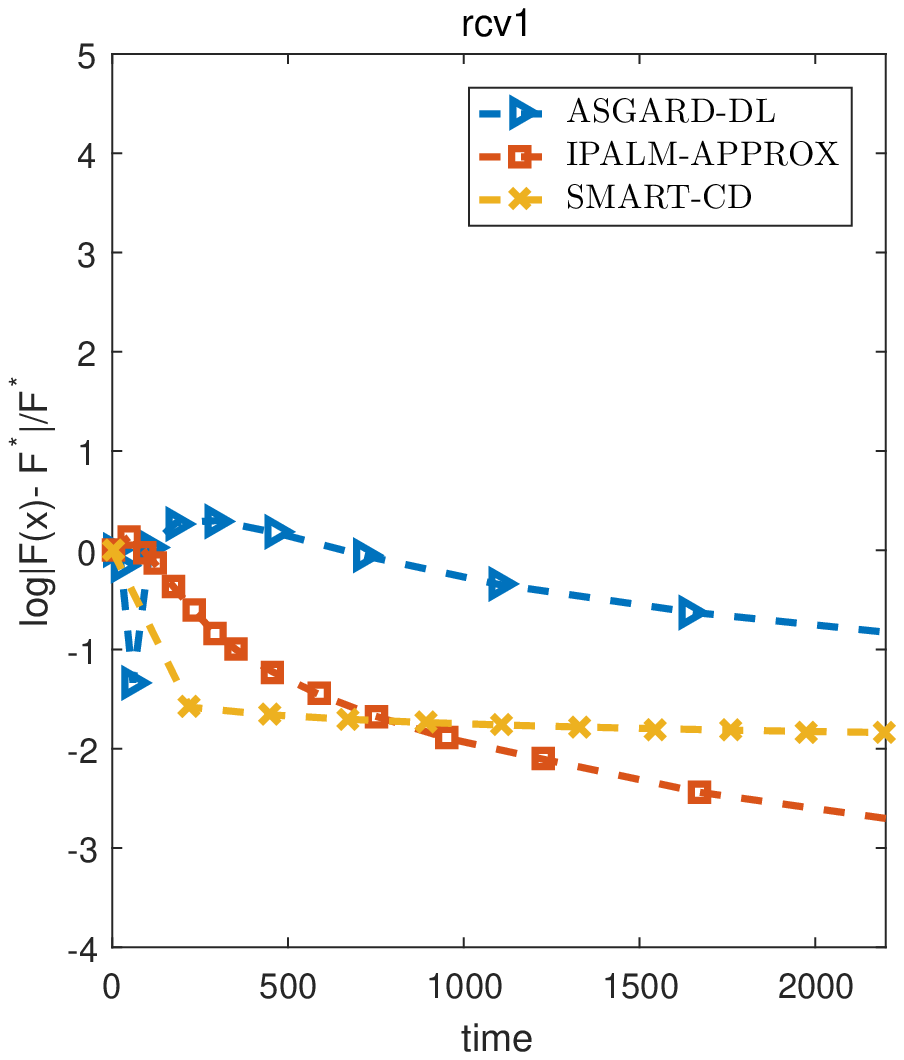}
			\caption{rcv1}
		\end{subfigure}
		\begin{subfigure}[t]{0.3\linewidth}
			\centering
			\includegraphics[scale=0.5]{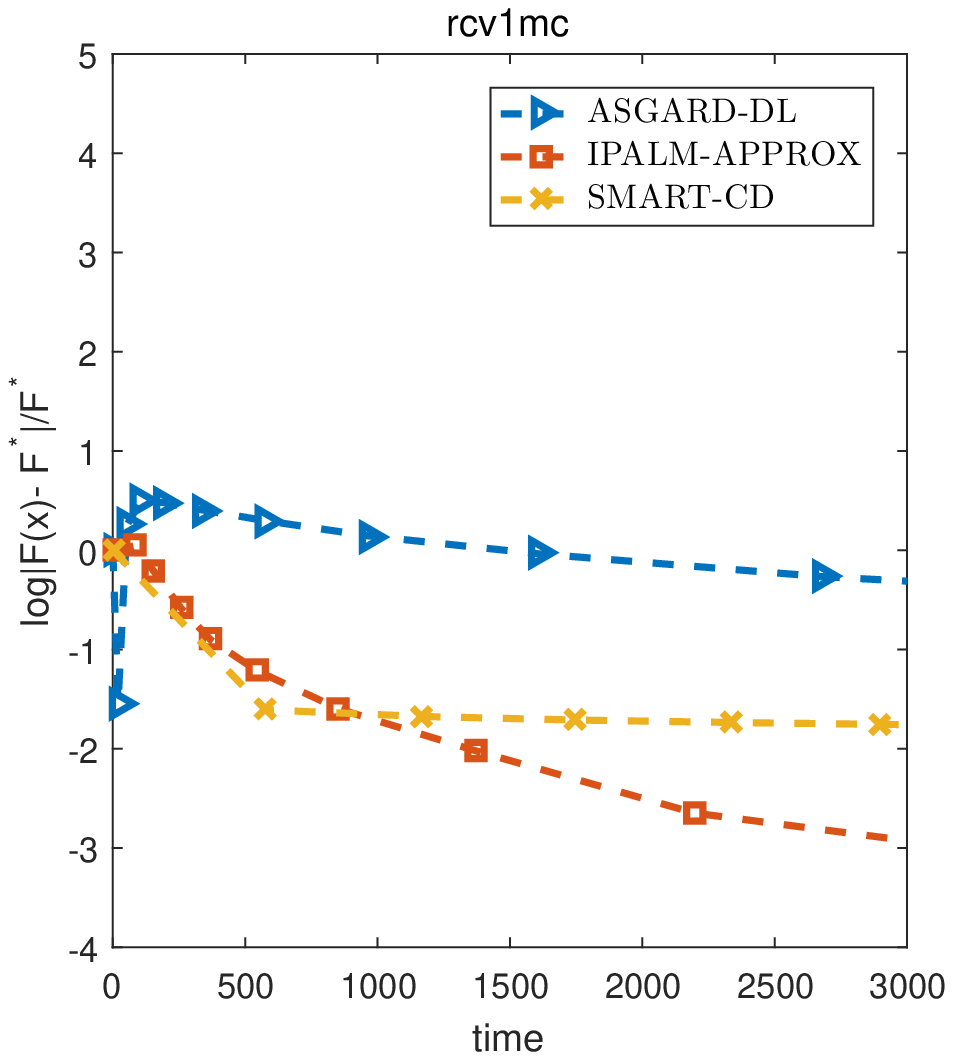}
			\caption{rcv1mc}
		\end{subfigure}
		\caption{Comparison of three algorithms for basis pursuit problem on three datasets. The $x$-axis is time and $y$-axis is $\log((F(x)- F^\star)/F^\star)$. Here we use the result of CVX as an approximation of  $F^\star$.}
		\label{fig_2}
	\end{figure}
	
	\begin{figure}[!h]
		\centering
		\begin{subfigure}[t]{0.3\linewidth}
			\centering
			\includegraphics[scale=0.5]{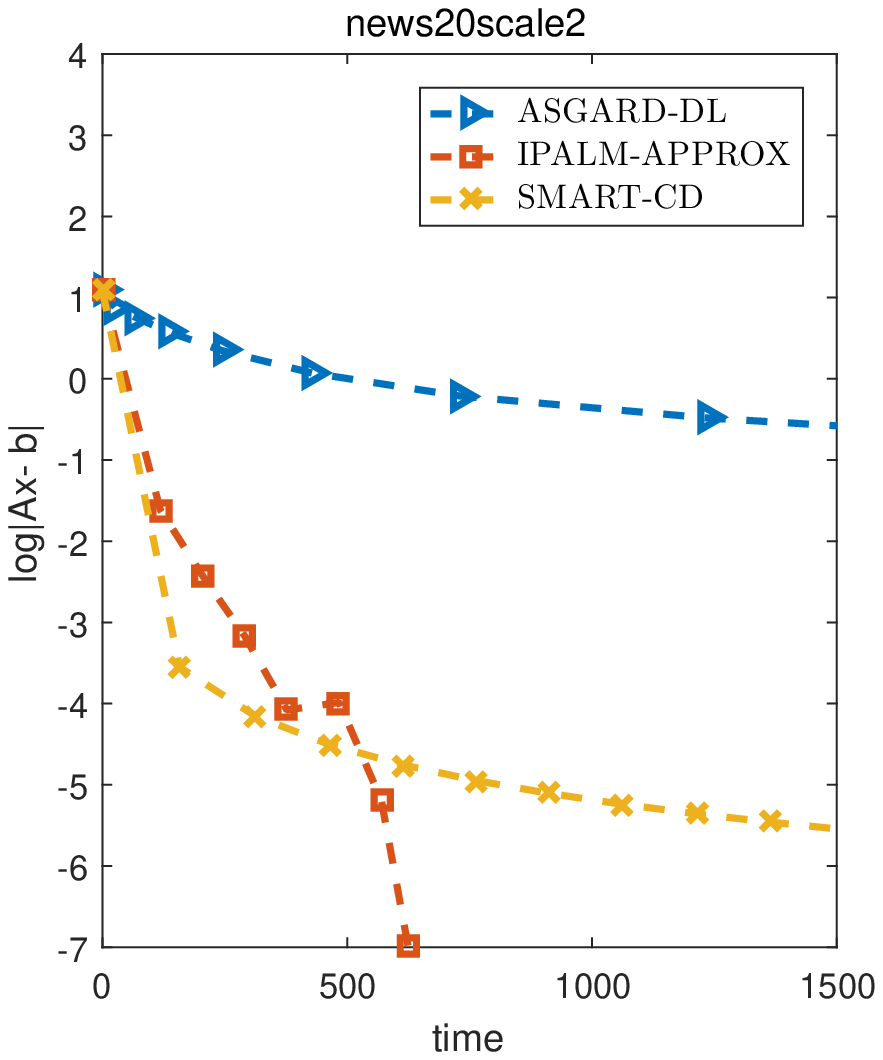}
			\caption{news20scale}
		\end{subfigure}
		\begin{subfigure}[t]{0.3\linewidth}
			\centering
			\includegraphics[scale=0.5]{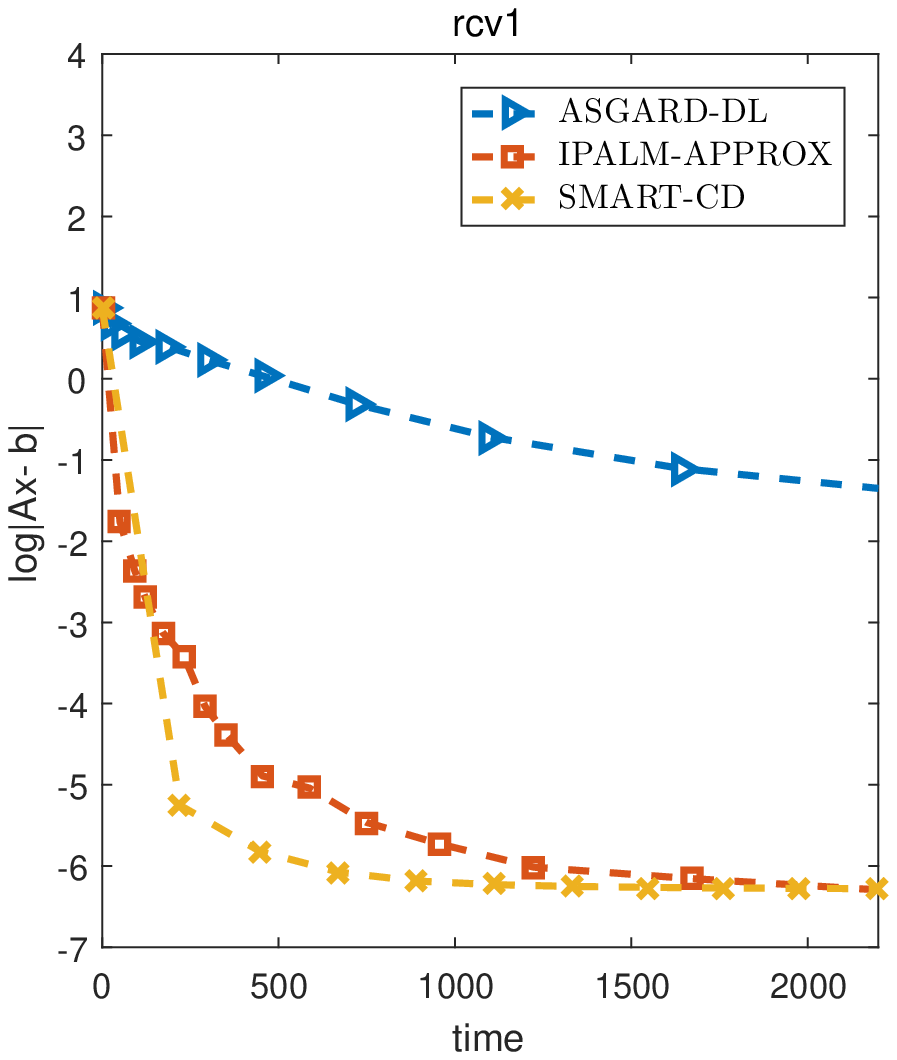}
			\caption{rcv1}
		\end{subfigure}
		\begin{subfigure}[t]{0.3\linewidth}
			\centering
			\includegraphics[scale=0.5]{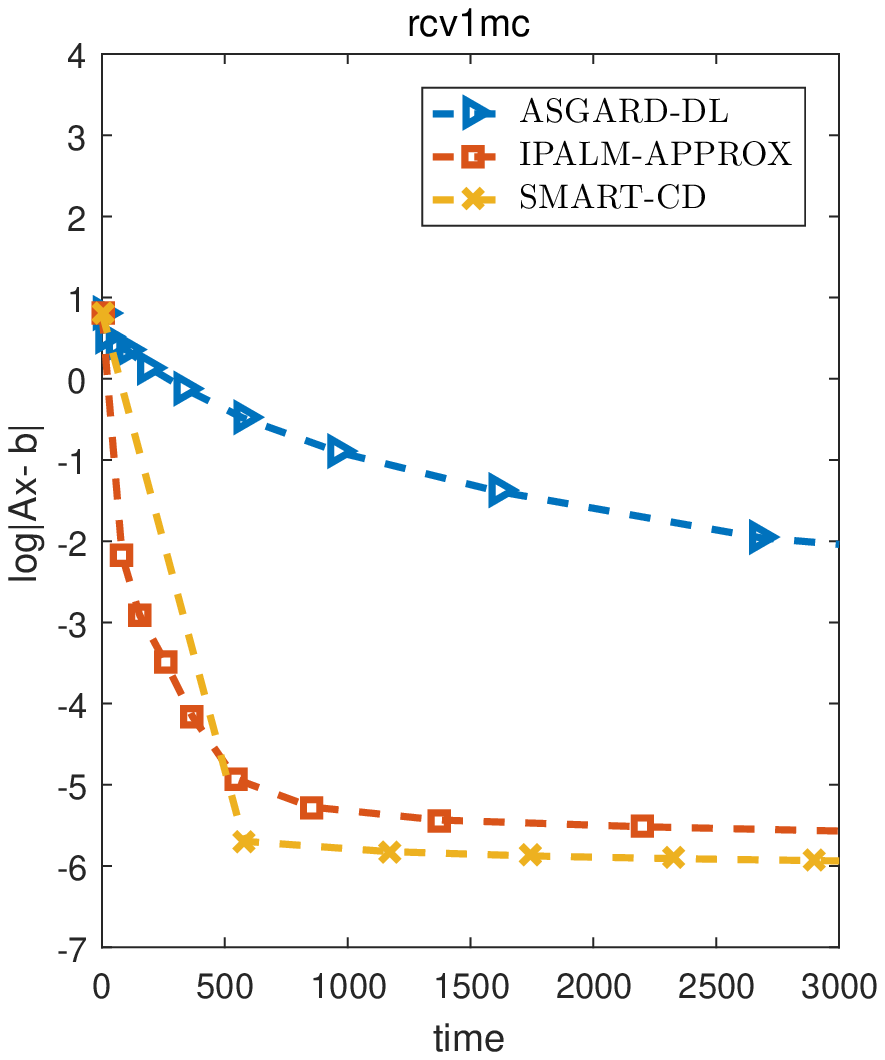}
			\caption{rcv1mc}
		\end{subfigure}
		\caption{Comparison of three algorithms for basis pursuit problem on three datasets. The $x$-axis is time and $y$-axis is infeasibility error $\log\|Ax-b\|$. }
		\label{fig_2_inf}
	\end{figure}
	\begin{table}[H]
		\centering
		\begin{tabular}{l|l|r|r}
			\hline
			Dataset & accuracy~\eqref{a:eraaa}& CVX time & IPALM-APPROX time \\
			\hline
			news20scale & $10^{-6}$ & $\sim$4200s & $\sim$600s \\
			rcv1 & $10^{-2}$ & $\sim$2700s & $\sim$1000s \\
			rcv1mc & $10^{-2}$ &  $\sim$1500s & $\sim$700s \\
			\hline
		\end{tabular}
		\caption{Running time  of CVX and IPALM-APPROX for basis pursuit problem on three datasets.}
		\label{tab_3}
	\end{table}
	
	\subsection{Fused Lasso}
	
	The third problem we solve is of the form:
	\begin{align*}
		\min_{x\in \mathbb{R}^n} \frac{1}{2}\norm{Ax- b}_2^2+ \lambda r\norm{x}_1+ \lambda (1- r)\sum_{i} |x_i- x_{i+ 1}|
	\end{align*}
	which is known as Fused Lasso problem~\cite{Tibshirani05sparsityand}. The datasets used are shown in Table~\ref{tab_1} and we set $\lambda r= \lambda (1- r)= 0.01$.
	
	The results are shown in Figure $\ref{fig_3}$ and Table $\ref{tab_4}$. For this problem, we tested both IPALM-APPROX and IPALM-Katyusha. Note that for the datasets in Table~\ref{tab_1}, we have  $n\leq m\leq 2n$ where $m$ is the problem size in~\eqref{prime:special1}. According to Table~\ref{tab_0}, we should expect IPALM-Katyusha to work similarly as IPALM-APPROX, which is indeed observed in practice. Note that in our implementation we used $\tau=\sqrt{m}$ with single processor.
	Hence the computational time of IPALM-Katyusha can be further reduced when multi-processor and parallel implementation is used.	
	\begin{figure}
		\centering
		\begin{subfigure}[t]{0.32\linewidth}
			\centering
			\includegraphics[scale=0.5]{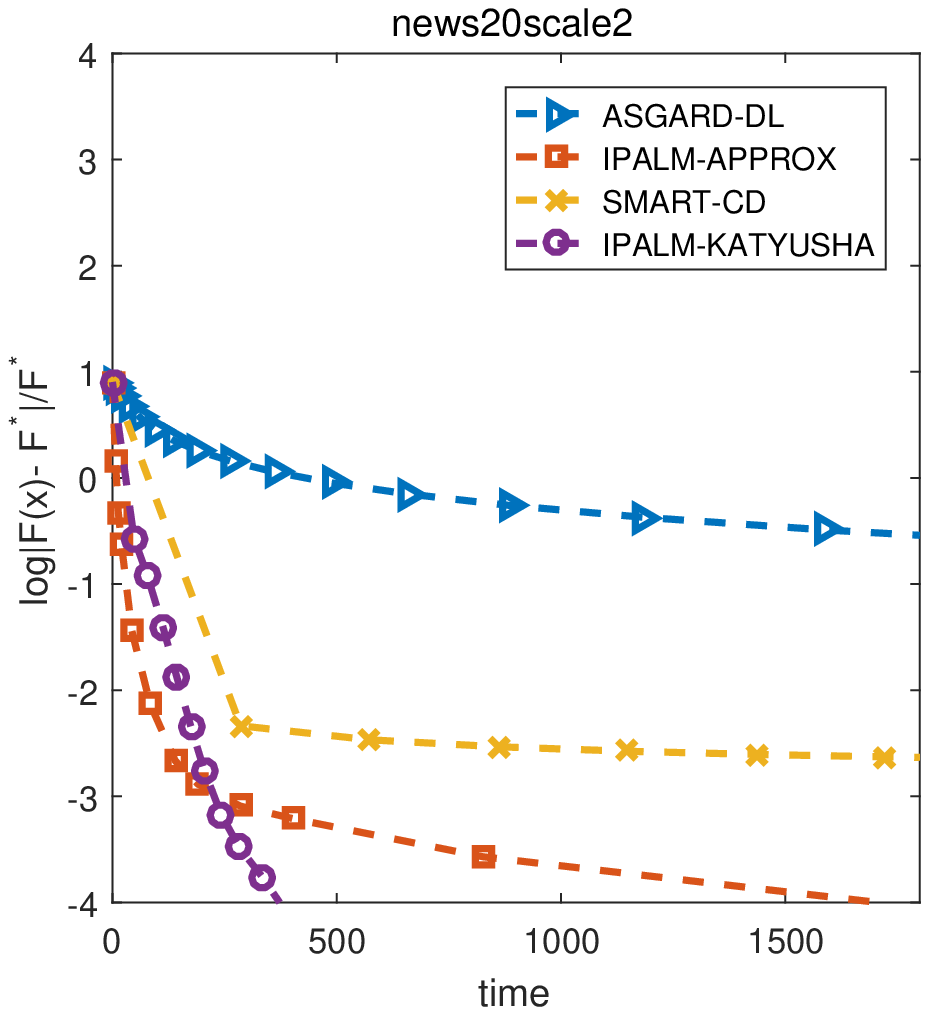}
			\caption{news20scale}
		\end{subfigure}
		\begin{subfigure}[t]{0.32\linewidth}
			\centering
			\includegraphics[scale=0.5]{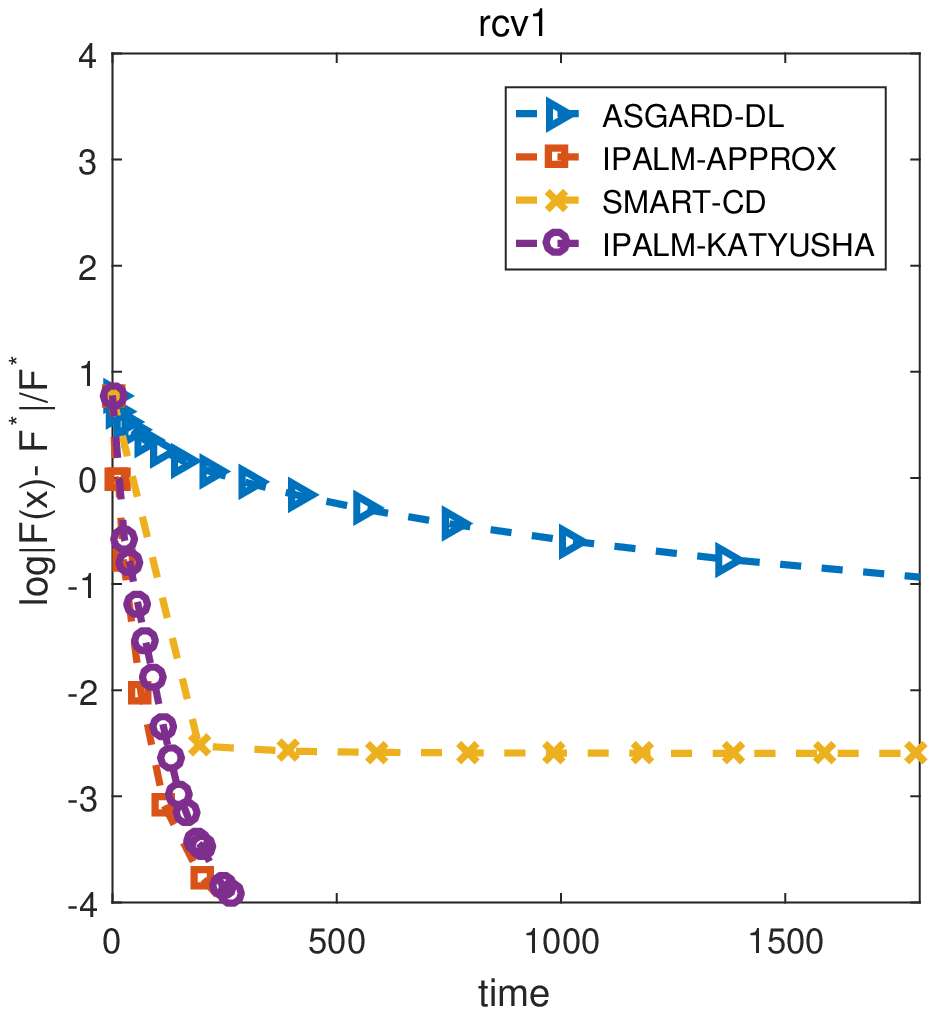}
			\caption{rcv1}
		\end{subfigure}
		\begin{subfigure}[t]{0.32\linewidth}
			\centering
			\includegraphics[scale=0.5]{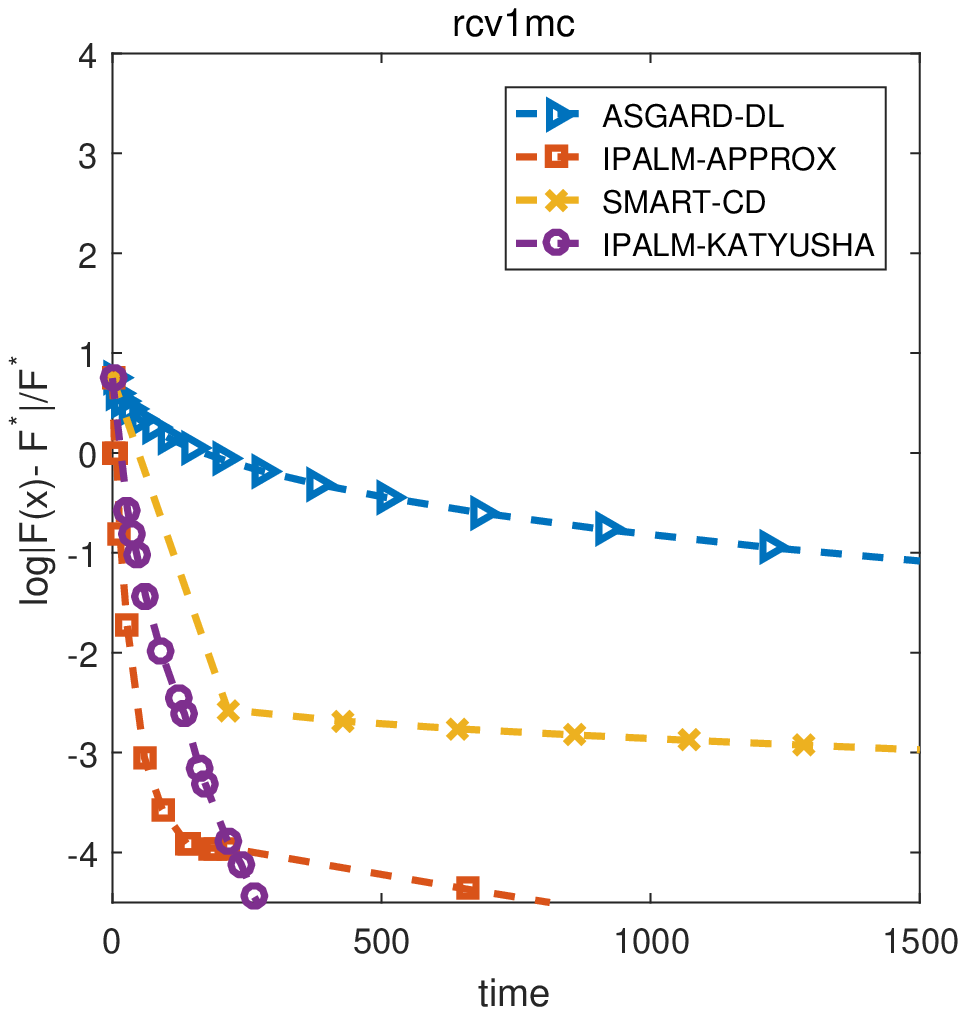}
			\caption{rcv1mc}
		\end{subfigure}
		\caption{Comparison of four algorithms for Fused Lasso problem on three datasets. The $x$-axis is time and $y$-axis is $\log((F(x)- F^\star)/F^\star)$. Here we use the result of CVX as an approximation of  $F^\star$.}
		\label{fig_3}
	\end{figure}
	
	\begin{table}[H]
		\centering
		\begin{tabular}{l|l|r|r|r}
			\hline
			Dataset & accuracy & CVX time & IPALM-APPROX time & IPALM-Katyusha time \\
			\hline
			news20scale & $10^{-4}$ & $\sim$1600s & $\sim$1700s & $\sim400s$\\
			rcv1 & $10^{-4}$ & $\sim$7000s & $\sim$300s & $\sim300$s\\
			rcv1mc & $10^{-4}$ &  $\sim$5500s & $\sim$400s & $\sim300$s\\
			\hline
		\end{tabular}
		\caption{Running time  of CVX,  IPALM-APPROX and IPALM-Katyusha for Fused Lasso on three datasets.}
		\label{tab_4}
	\end{table}
	As we can see form Figure $\ref{fig_3}$, IPALM-APPROX and IPALM-Katyusha both perform better than ASGARD and SMART-CD. From Table $\ref{tab_4}$,  IPALM-Katyusha  significantly outperforms CVX to get a mid-level accurate solution for these three datasets.
	\subsection{Soft Margin SVM}
	The last problem we solve is of the form:
	\begin{align*}
		\min_{x\in\mathbb{R}^n, \omega\in\mathbb{R}} \lambda\norm{x}_1+ \frac{1}{m} \sum_{i= 1}^{m} \max\left(0, 1- b_i\left(\langle a_i, x\rangle- \omega\right)\right)  
	\end{align*}
	which is known as $l_1$ regularized soft margin support vector machine problem~\cite{Zhu:2003:SVM:2981345.2981352}. Here $a_i\in \mathbb{R}^n$ are feature vectors and $b_i\in \{-1, 1\}$ are labels for $i= 1,\ldots,m$. We use three different datasets from libsvm \cite{chang2011libsvm}. The details about the datasets are given in Table $\ref{tab_5}$.
	\begin{table}[H]
		\centering
		\begin{tabular}{l|r|r}
			\hline
			Dataset & Training size ($m$) & Number of features ($n$)\\
			\hline
			w4a & 7366 & 300 \\
			w8a & 49479 & 300 \\
			real-sim & 72309 & 20958\\
			\hline
		\end{tabular}
		\caption{Datasets from libsvm}
		\label{tab_5}
	\end{table}
	\begin{figure}[H]
		\centering
		\begin{subfigure}[t]{0.32\linewidth}
			\centering
			\includegraphics[scale=0.5]{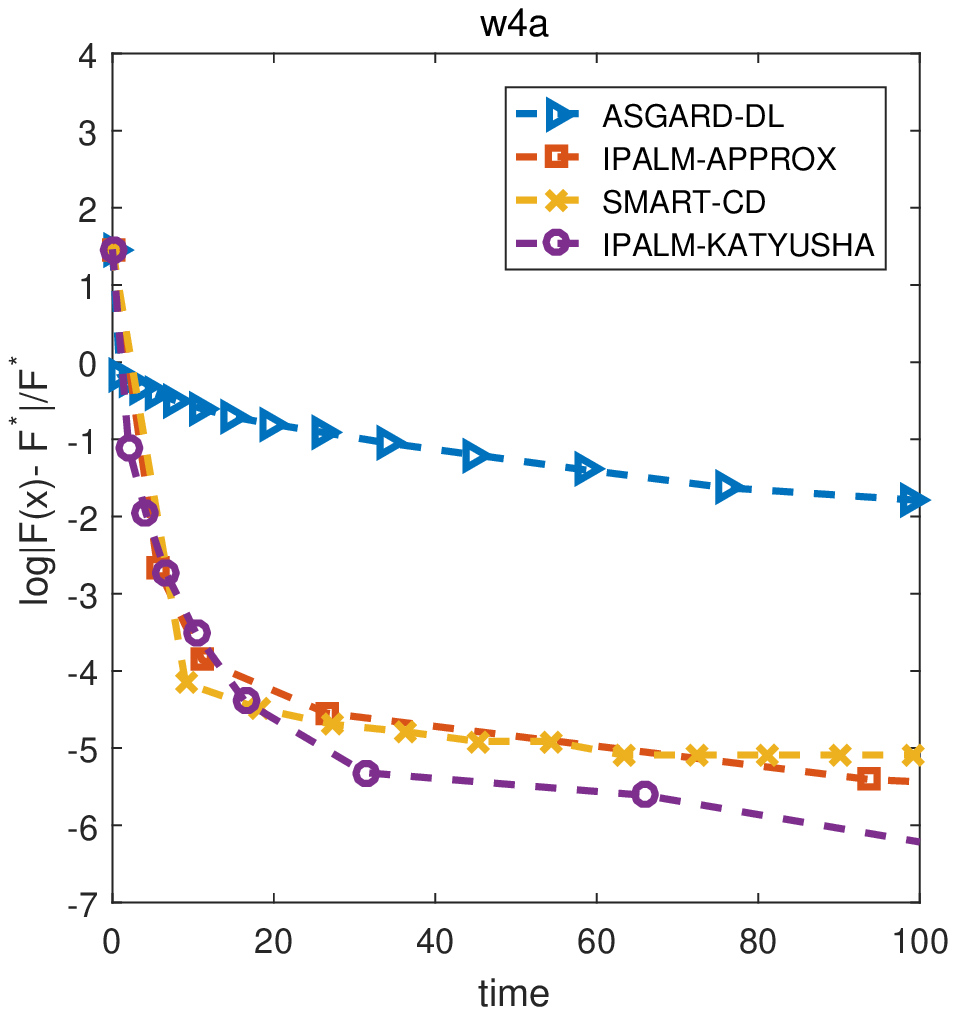}
			\caption{w4a}
		\end{subfigure}
		\begin{subfigure}[t]{0.32\linewidth}
			\centering
			\includegraphics[scale=0.5]{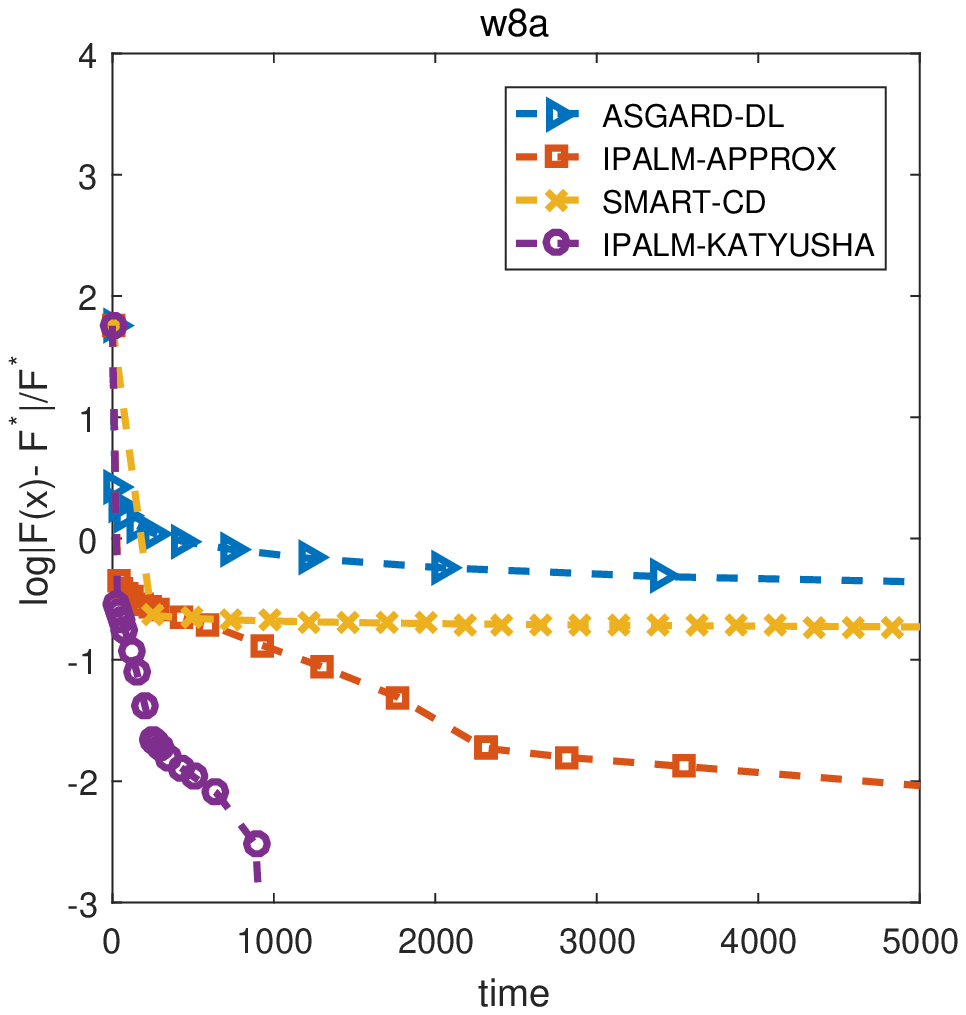}
			\caption{w8a}
		\end{subfigure}
		\begin{subfigure}[t]{0.32\linewidth}
			\centering
			\includegraphics[scale=0.5]{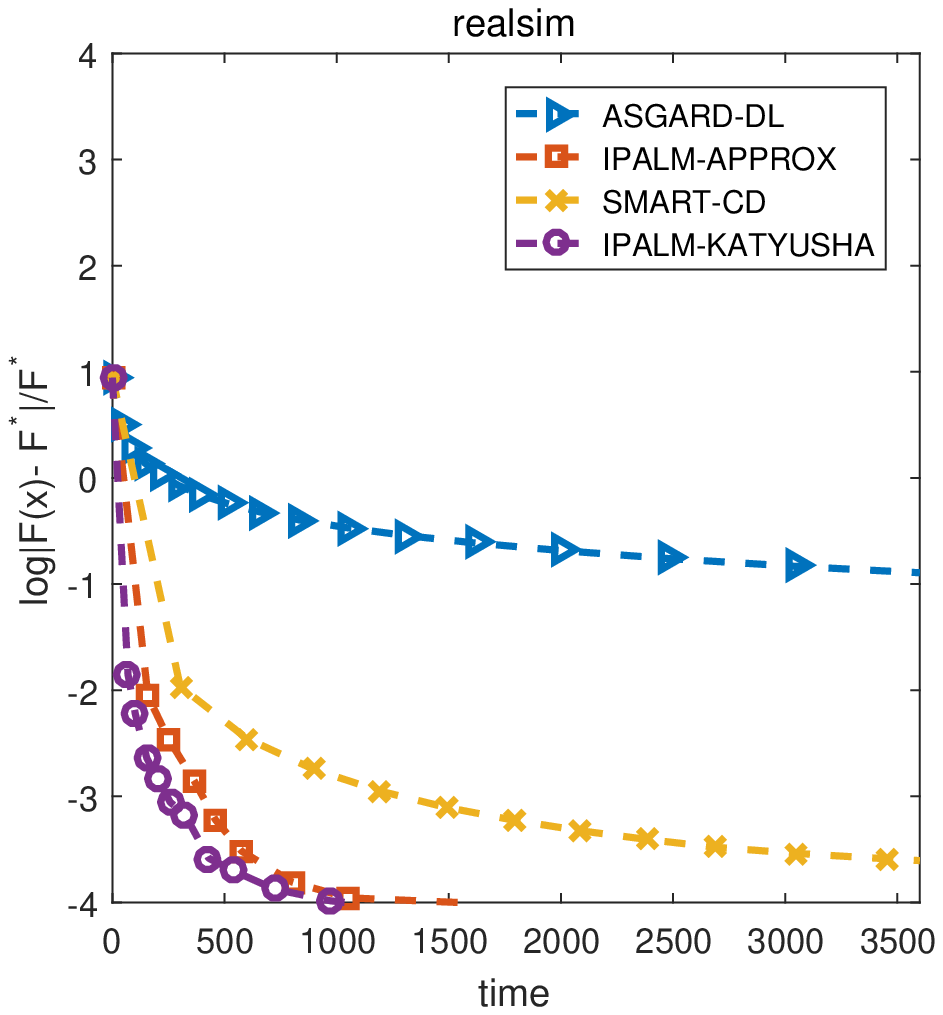}
			\caption{real-sim}
		\end{subfigure}
		\caption{Comparison of four algorithms for SVM problem on three datasets. The $x$-axis is time and $y$-axis is $\log((F(x)- F^\star)/F^\star)$. Here we use the best result of these four algorithms as an approximation of $F^\star$.}
		\label{fig_7}
	\end{figure}	
	Since here $m\geq n$, we expect IPALM-Katyusha to converge faster than IPALM-APPROX, as indicated by our theoretical bounds given in Table~\ref{tab_0}. Indeed, as we observe from Figure $\ref{fig_7}$, IPALM-Katyusha has the best performance for all datasets.   For w4a and real-sim, the difference of IPALM-Katyusha and IPALM-APPROX is small. But for w8a with $m\gg n$, IPALM-Katyusha significantly outperforms IPALM-APPROX, as well as SMART-CD and ASGARD-DL.  Note that for this problem CVX fails so we cannot compare the running time with CVX.

	\section{Conclusion and Future Research}\label{sec::con}
	In this paper we consider a class of structured convex minimization problem and develop an inexact proximal augmented Lagrangian  method with explicit inner termination rule. Our framework allows arbitrary linearly convergent inner solver, including in particular many randomized first-order methods. 
	When $p(\cdot)$ is linear, under the same assumptions as~\cite{NecoaraPatrascuGlineur,Lan:2016:IFA:2874819.2874858,NedelcuNecoaraTran,PatrascuNecoara15,LiuLiuMa19MoR,lu18} but without the boundedness of $\dom(g)$, we obtain nearly optimal $\tilde O(1/\epsilon)$ and  $\tilde O(1/\sqrt{\epsilon})$ complexity bound respectively for the non-strongly convex and strongly convex case. 
	The flexible inner solver choice allows us to deal with large-scale constrained problem more efficiently, with the aid of recent advances in randomized first-order methods for unconstrained problem.
	We provide numerical evidence showing the 
	efficiency of our approach compared with existing ones when the problem dimension is high.
	
	There are several interesting directions to exploit in the future.
	\begin{enumerate}
		\item The complexity bound established in this paper for non-strongly convex problem is  $\tilde O(1/\epsilon)$.
		Throughout the paper we only rely on the fact that the sequence generated by PPA is bounded, whereas it is known that PPA can be linearly convergent if certain metric sub-regularity is satisfied, see e.g.~\cite{Yuan18LinearconvergenceofADMM}. We expect to obtain a linearly convergent rate under these conditions, see e.g.~\cite{li2018complexity}.
		
		\item In numerical experiments, the choice of $\beta_0$ does influence the performance. Can a reasonable guess on $\beta_0$ be derived from the analysis?
		
		\item When $f$ is only relatively smooth (see Section~\ref{sec:LBPG}), we only obtained $\tilde O(1/\epsilon^2)$ and $\tilde O(1/\epsilon)$ complexity bound for non-strongly convex and strongly convex case. Can we improve to $\tilde O(1/\epsilon)$ and $\tilde O(1/\sqrt{\epsilon})$?
		
		\item Can this work be extended to saddle point problem? In particular, \cite{hien2017inexact} discussed an inexact primal-dual method for nonbilinear saddle point problems with  bounded $dom(g)$. Can we get rid of the boundedness assumption for saddle point problem?
		
		\item  Can this work be extended to weakly convex case as in~\cite{dru18,rafique2018non}?
		
	\end{enumerate}

	\begin{appendix}

		\section{Inexact Proximal Point algorithm and inexact Augmented Lagrangian method}
		\subsection{Inexact Proximal Point Method}
		Let $\mathcal{T}:\mathbb{R}^{n+d}\to \mathbb{R}^{n+d}$ be a maximal monotone operator and $\mathcal{J}_{\rho}= (\mathcal{I}+ \rho\mathcal{T})^{-1}$ be the resolvent of $\cT$, where $\mathcal{I}$ denotes the identity operator. Then for any $z^*$ such that $0\in \mathcal{T}(z^*)$ \cite{rock76ppa}, 
		\begin{align}
			\label{PPA}
			\Mnorm{\mathcal{J}_{\rho}(z)- z^*}^2+ \Mnorm{\mathcal{J}_{\rho}(z)- z}^2\le \Mnorm{z- z^*}^2.
		\end{align}
		\begin{algorithm}[H]
			\caption{PPA}
			\label{ppa}
			\begin{algorithmic}
				\State \textbf{Input:} $z_0$, $\{\varepsilon_s\}, \{\rho_s\}$.
				\For {$k= 0,1,\ldots$}
				\State Compute $z^{s+1}\approx \mathcal{J}_{\rho_s}(z^{s})$ such that $\Mnorm{z^{s+1}- \mathcal{J}_{\rho_s}(z^{s})}\le \varepsilon_s$;
				\EndFor
			\end{algorithmic}
		\end{algorithm}
		\begin{lemma}\cite{rock76ppa}
			\label{lu}
			Let $\{z^{s}\}$ be the sequence generated by Algorithm $\ref{ppa}$. Then for any $z^*$ such that $0\in \cT(z^*)$,
			\begin{align*}
				\Mnorm{z^{s+1}- z^*}\le \Mnorm{z_0- z^*}+ \sum_{i= 0}^{s}\varepsilon_i \\
				\Mnorm{z^{s+1}- z^{s}}\le \Mnorm{z_0- z^*}+ \sum_{i= 0}^{s}\varepsilon_i
			\end{align*}
		\end{lemma}
		We now give a stochastic generalization of Algorithm~\ref{ppa}.
		\begin{algorithm}[H]
			\caption{sPPA}
			\label{sPPA}
			\begin{algorithmic}
				\State \textbf{Input:} $z_0$, $\{\varepsilon_s\}, \{\rho_s\}$.
				\For {$k= 0,1,\ldots$}
				\State Compute $z^{s+1}\approx \mathcal{J}_{\rho_s}(z^{s})$ such that $\mathbb{E}\left[\Mnorm{z^{s+1}- \mathcal{J}_{\rho_s}(z^{s})}^2\right]\le \varepsilon_s^2$;
				\EndFor
			\end{algorithmic}
		\end{algorithm}	
		We then extend Lemma $\ref{lu}$ for Algorithm~\ref{sPPA}.
		\begin{lemma}
			\label{li}
			Let $\{z^{s}\}$ be the sequence generated by Algorithm $\ref{sPPA}$. Then for any $z^*$ such that $0\in \cT(z^*)$,
			\begin{align*}
				\mathbb{E}\left[\Mnorm{z^{s+1}- z^*}\right]\le \Mnorm{z_0- z^*}+ \sum_{i= 0}^{s}\varepsilon_i \\
				\mathbb{E}\left[\Mnorm{z^{s+1}- z^{s}}\right]\le \Mnorm{z_0- z^*}+ \sum_{i= 0}^{s}\varepsilon_i \\
				\left(\mathbb{E}\left[\Mnorm{z^{s+1}- z^*}^2\right]\right)^{1/2}\le \Mnorm{z_0- z^*}+ \sum_{i= 0}^{s}\varepsilon_i 
			\end{align*}		
		\end{lemma}
		\begin{proof}
			By the definition of $z^{s}$, we have $\left(\mathbb{E}\Mnorm{z^{s+1}- \mathcal{J}_{\rho_s}(z^{s})}\right)^2\le \mathbb{E}\Mnorm{z^{s+1}- \mathcal{J}_{\rho_s}(z^{s})}^2\le \varepsilon_s^2$. The first and second estimates can be obtained by taking expectation on both sides of the result of Lemma $\ref{lu}$. The third estimate is derived from $\eqref{PPA}$:
			\begin{align*}
				0&\le \Mnorm{\mathcal{J}_{\rho_s}- z^{s}}^2\le \Mnorm{z^{s}- z^*}^2- \Mnorm{\mathcal{J}_{\rho_s}(z^{s})- z^*}^2 \\
				&= \Mnorm{z^{s}- z^*}^2- \Mnorm{\mathcal{J}_{\rho_s}(z^{s})- z^{s+1}+ z^{s+1}- z^*}^2  \\
				&\le \Mnorm{z^{s}- z^*}^2- \Mnorm{z^{s+1}- z^*}^2-  \Mnorm{\mathcal{J}_{\rho_s}(z^{s})- z^{s+1}}^2+ 2 \Mnorm{\mathcal{J}_{\rho_s}(z^{s})- z^{s+1}}\Mnorm{z^{s+1}- z^*}
			\end{align*}
			Taking expectation on both sides we have:
			\begin{align*}
				0&\le \mathbb{E}\left[\Mnorm{z^{s}- z^*}^2\right]- \mathbb{E}\left[\Mnorm{z^{s+1}- z^*}^2\right]-  \mathbb{E}\left[\Mnorm{\mathcal{J}_{\rho_s}(z^{s})- z^{s+1}}^2\right]+ 2 \mathbb{E}\left[\Mnorm{\mathcal{J}_{\rho_s}(z^{s})- z^{s+1}}\Mnorm{z^{s+1}- z^*}\right] \\
				&\le \mathbb{E}\left[\Mnorm{z^{s}- z^*}^2\right]- \mathbb{E}\left[\Mnorm{z^{s+1}- z^*}^2\right]-  \mathbb{E}\left[\Mnorm{\mathcal{J}_{\rho_s}(z^{s})- z^{s+1}}^2\right]+ 2\left(\mathbb{E}\left[\Mnorm{\mathcal{J}_{\rho_s}(z^{s})- z^{s+1}}^2\right]\mathbb{E}\left[\Mnorm{z^{s+1}- z^*}^2\right]\right)^{1/2}  \\
				&= \mathbb{E}\left[\Mnorm{z^{s}- z^*}^2\right]- \left(\left(\mathbb{E}\left[\Mnorm{z^{s+1}- z^*}^2\right]\right)^{1/2}- \left(\mathbb{E}\left[\Mnorm{\mathcal{J}_{\rho_s}(z^{s})- z^{s+1}}^2\right]\right)^{1/2}\right)^2
			\end{align*}		
			where the second inequality we use $\mathbb{E}[XY]\le (\bE[X^2])^{1/2}(\bE[Y^2])^{1/2}$. Therefore
			\begin{align*}
				\left(\mathbb{E}\left[\Mnorm{z^{s+1}- z^*}^2\right]\right)^{1/2}- \varepsilon_s\le \left(\mathbb{E}\left[\Mnorm{z^{s+1}- z^*}^2\right]\right)^{1/2}- \left(\mathbb{E}\left[\Mnorm{\mathcal{J}_{\rho_s}(z^{s})- z^{s+1}}^2\right]\right)^{1/2}\le \left(\mathbb{E}\left[\Mnorm{z^{s}- z^*}^2\right]\right)^{1/2}
			\end{align*}
			Then summing up the latter inequalities from $s= 0$ we obtain the third inequality.
		\end{proof}
		
		\subsection{Inexact ALM}
		We define the maximal monotone operator $\cT_{l}$ as follows.
		\begin{align*}
			\cT_l(x;\lambda)&=\left\{(v;u): (v;-u)\in \partial L(x;\lambda)\right\}\\
			&=\left\{\begin{pmatrix}
				\nabla f(x)+\partial g(x)+\nabla p(x)\lambda\\
				-p(x)+\partial h^*(\lambda)
			\end{pmatrix}\right \}
		\end{align*}
		Recall the definitions in~\eqref{a:erdfgtff}. We further let $\Lambda^\star(y,\lambda,\beta):=\Lambda(p^\star(y,\lambda,\beta);\lambda,\beta)$.
		By first order optimality condition and~\eqref{a:nablehu}, we know that
		$$0\in \nabla f(x^\star(y,\lambda,\beta))+\partial g(x^\star(y,\lambda,\beta))+\nabla p(x^\star(y, \lambda,\beta)) \Lambda^\star(y,\lambda,\beta)+\beta(x^\star(y,\lambda,\beta)-y)$$
		Secondly we know from~\eqref{a:optimalitycondition} that
		\begin{align*}
			p^\star(y,\lambda,\beta)-\beta(\Lambda^\star(y,\lambda, \beta)-\lambda)\in \partial h^*(\Lambda^\star(y,\lambda,\beta)).
		\end{align*}
		It follows that
		\begin{align}\label{a:dfertsd}
			( \cI+\beta^{-1}\cT_l)^{-1}(y;\lambda)=(x^\star(y,\lambda,\beta);\Lambda^\star(y,\lambda,\beta))
		\end{align}
		We can then establish the following well known link between inexact ALM and inexact PPA.
		\begin{proposition}[compare with~\cite{rock76ppa}]\label{prop:3fgrtrt}
			Algorithm~\ref{ipALM} is a special case of Algorithm~\ref{sPPA} with $\cT=\cT_l$,  $\rho_s=1/\beta_s$ and $\varepsilon_s= \sqrt{2\epsilon_s/\beta_s}$.
		\end{proposition}
		\begin{proof}
			This follows from~\eqref{a:dfertsd} and Lemma~\ref{l:sL2}.
		\end{proof}

		\section{Missing proofs}
		\subsection{Proofs in Section~\ref{sec:prox_ALM}}\label{app:pALM}
		\begin{proof}[proof of Lemma~\ref{l:infimalconv}]
			For any $x, y\in \R^n$ and $\alpha\in [0,1]$, let $z= \alpha x+ (1- \alpha)y$.
			By condition~\eqref{a:herffg}, 
			$$h\left(p(z)- \alpha u- (1- \alpha) v\right)\le \alpha h(p(x)- u)+ (1- \alpha)h(p(y)- v),\enspace \forall u,v\in \R^d.$$
			It follows that
			\begin{align*}
				\tilde{\psi}(z)&= \inf_\omega \left\{h(p(z)- \omega)+ \psi(\omega) \right\}= \inf_{u,v} \left\{h\left(p(z)- \alpha u- (1- \alpha)v\right)+ \psi(\alpha u+ (1- \alpha)v) \right\}\\
				&\le \inf_{u,v} \left\{\alpha h(p(x)- u)+ (1- \alpha)h(p(y)- v)+ \alpha\psi(u)+ (1- \alpha)\psi(v) \right\} \\
				&= \alpha\inf_u \left\{h(p(x)- u)+ \psi(u) \right\}+ (1- \alpha)\inf_v \left\{h(p(y)- v)+ \psi(v) \right\} \\
				&= \alpha \tilde{\psi}(x)+ (1- \alpha)\tilde{\psi}(y).
			\end{align*}
		\end{proof}
		\begin{proof}[proof of Lemma~\ref{l:rterrr}]
			The convexity of $\tilde \psi$ follows from~\eqref{a:dualityhbeta} and Lemma~\ref{l:infimalconv} with $\psi(w):=\frac{1}{2\beta}\|w\|^2+\<w,\lambda>$. The gradient formula follows from~\eqref{a:nablehu}.
		\end{proof}
		
		\begin{proof}[proof of Lemma~\ref{bound_y_1.3}]
			This is a direct consequence of Proposition~\ref{prop:3fgrtrt} and Lemma~\ref{li}.
		\end{proof}

		\begin{proof}[proof of Corollary~\ref{coro:bound}]
			By Lemma~\ref{bound_y_1.3}, we have
			$$
			\mathbb{E}\left[\Mnorm{(x^{s},\lambda^{s+ 1})- (x^{s-1},\lambda^{s})}\right]\le \Mnorm{(x^{-1},\lambda^0)- (x^\star,\lambda^\star)}+ \frac{2\sqrt{\epsilon_0/\beta_0}}{1-\sqrt{\eta/\rho}},\enspace \forall s\geq 0,
			$$
			and
			$$
			\bE\left[ \Mnorm{(x^{s},\lambda^{s+ 1})- (x^\star,\lambda^\star)}^2\right] \leq \left(\Mnorm{(x^{-1},\lambda^0)- (x^\star,\lambda^\star)}+ \frac{2\sqrt{\epsilon_0/\beta_0}}{1-\sqrt{\eta/\rho}} \right)^2,\enspace \forall s\geq 0.
			$$
			Consequently,
			$$
			\mathbb{E}\left[\Mnorm{\lambda^{s+ 1}-\lambda^{s}}\right]\le \Mnorm{(x^{-1},\lambda^0)- (x^\star,\lambda^\star)}+ \frac{2\sqrt{\epsilon_0/\beta_0}}{1-\sqrt{\eta/\rho}},\enspace \forall s\geq 0,
			$$
			and
			$$
			\max\left( \bE\left[ \Mnorm{x^{s}-x^\star}^2\right],  \bE\left[ \Mnorm{\lambda^{s+1}-\lambda^\star}^2\right] \right)  \leq \left(\Mnorm{(x^{-1},\lambda^0)- (x^\star,\lambda^\star)}+ \frac{2\sqrt{\epsilon_0/\beta_0}}{1-\sqrt{\eta/\rho}} \right)^2,\enspace \forall s\geq 0.
			$$
			We then conclude.
		\end{proof}
		
		\begin{proof}[proof of Theorem~\ref{rock2}]
			First,
			\begin{equation}\label{eq:fsr}
			\begin{array}{ll}
			h_1(p_1(x^s))- h(p(x^s);\lambda^s,\beta_s) &\overset{\eqref{a:hbeta}}{=} h_1(p_1(x^s))-h_1(p_1(x^s)-\beta_s(\lambda_1^{s+1}-\lambda_1^s))-\frac{\beta_s}{2}(\| \lambda^{s+1}\|^2-\|\lambda^s\|^2)\\&\leq L_{h_1} \beta_s \|\lambda^{s+1}-\lambda^s\|+\frac{\beta_s}{2}(\| \lambda^{s}\|^2-\|\lambda^{s+1}\|^2).
			\end{array}
			\end{equation}
			Then we know that
			\begin{align*}
				F(x^s)-L(x^s; x^{s-1}, \lambda^s, \beta_s)&= h_1(p_1(x^s))- h(p(x^s);\lambda^s,\beta_s)-\frac{\beta_s}{2}\|x^{s}-x^{s-1}\|^2\\& \overset{\eqref{eq:fsr}}{\leq} L_{h_1} \beta_s \|\lambda^{s+1}-\lambda^s\|+\frac{\beta_s}{2}(\| \lambda^{s}\|^2-\|\lambda^{s+1}\|^2)-\frac{\beta_s}{2}\|x^{s}-x^{s-1}\|^2.
			\end{align*}
			Since $H_s(\cdot)$ is $\beta_s$-strongly convex, we know that \begin{align*}
				L^\star(x^{s-1}, \lambda^s,\beta_s)&\leq L(x^\star; x^{s-1}, \lambda^s, \beta^s)-\frac{\beta_s}{2}\|x^\star-x^\star(x^{s-1}, \lambda^s, \beta_s)\|^2 \\&  \overset{\eqref{a:dualityhbeta}}{\leq} F^\star +\frac{\beta_s}{2}\|x^\star-x^{s-1}\|^2-\frac{\beta_s}{2}\|x^\star-x^\star(x^{s-1}, \lambda^s, \beta_s)\|^2.
			\end{align*}
			Combining the latter two bounds we get
			\begin{align*}
				F(x^s)-F^\star&\leq L(x^s; x^{s-1}, \lambda^s, \beta_s)-L^\star(x^{s-1}, \lambda^s,\beta_s)+ L_{h_1} \beta_s \|\lambda^{s+1}-\lambda^s\|+\frac{\beta_s}{2}(\| \lambda^{s}\|^2-\|\lambda^{s+1}\|^2)\\
				&\quad +\frac{\beta_s}{2}\|x^\star-x^{s-1}\|^2-\frac{\beta_s}{2}\|x^\star-x^\star(x^{s-1}, \lambda^s, \beta_s)\|^2
				-\frac{\beta_s}{2}\|x^{s}-x^{s-1}\|^2.
			\end{align*}
			Furthermore, by convexity of $h_1(\cdot)$, 
			$$  \inf_x F(x)+ \<\lambda_2^\star, p_2(x)> -h_2^*(\lambda_2^\star) \geq \inf_x f(x)+g(x)+\<\lambda^\star, p(x)>-h^*(\lambda^\star)=D(\lambda^\star).
			$$
			Now we apply the strong duality assumption~\eqref{a:strd} to obtain:
			$$
			F(x^s)+ \<\lambda_2^\star, p_2(x^s)> -h_2^*(\lambda_2^\star)\geq  \inf_x F(x)+ \<\lambda_2^\star, p_2(x)> -h_2^*(\lambda_2^\star) \geq F^\star.
			$$
			Consequently,
			$$
			F(x^s)-F^\star\geq  \<\lambda_2^\star, -p_2(x^s)> + h_2^*(\lambda_2^\star) \geq 
			\sup_v \<\lambda_2^\star, v-p_2(x^s)>-h_2(v) \geq 
			-\|\lambda_2^\star\|\dist(p_2(x^s),\cK).
			$$
			From~\eqref{a:optimalitycondition} we know 
			$$
			p_2(x^s)-\beta_s(\lambda_2^{s+1}-\lambda_2^s)\in \cK,
			$$
			and thus
			$$
			\dist(p_2(x^s), \cK)\leq \beta_s\| \lambda_2^{s+1}-\lambda_2^s\|.
			$$
		\end{proof}
		
		\begin{proof}[proof of Corollary~\ref{coro:them1c}]
			Using Corollary~\ref{coro:bound},  the bounds in Theorem~\ref{rock2} can be relaxed as:
			\begin{align*}
				& \bE[ F(x^s)-F^\star] \leq \epsilon_s+ 2L_{h_1}^2 \beta_s +c_0\beta_s	\\
				&
				\bE[F(x^s)-F^\star]	\geq -\beta_s\|\lambda_2^\star\| \sqrt{c_0},\\
				&	 \bE[\dist(p_2(x^s), K)]\leq \beta_s \sqrt{c_0}.	\end{align*}
			We then conclude by noting that~\eqref{a:sbound} guarantees
			$$
			\max(\epsilon_0+ 2L_{h_1}^2 \beta_0 +c_0\beta_0, \beta_0\|\lambda_2^\star\| \sqrt{c_0}, \beta_0 \sqrt{c_0})\leq \epsilon \rho^{-s}.
			$$
		\end{proof}

		\subsection{Proofs in Section~\ref{sec:overall}}\label{sec:proofoverall}
		
		\begin{proof}[proof of Lemma~\ref{l:erer2}]
			We first bound
			\begin{align*}
				&\bE[ \left((\beta_s+\beta_{s+1})L_{h_1}+\|\beta_s \lambda_1^s-\beta_{s+1}\lambda_1^{s+1}\| \right)^2
				+\|\beta_s \lambda_2^s-\beta_{s+1}\lambda_2^{s+1}\|^2 ]\\&
				\leq 2(\beta_s+\beta_{s+1})^2 L^2_{h_1}+2\bE[ \|\beta_s \lambda^s-\beta_{s+1}\lambda^{s+1}\|^2]\\
				& \leq 2(\beta_s+\beta_{s+1})^2 L^2_{h_1}+4(\beta_s^2+\beta_{s+1}^2) c\\
				&\leq 4(\beta_s^2+\beta_{s+1}^2)(L^2_{h_1}+c).
			\end{align*}
			Since
			$$
			\lambda^{s+1}=\Lambda(p(x^s);\lambda^{s},\beta_{s}),
			$$
			by Lemma~\ref{l::boundbetabetap} we have
			\begin{align}
				&\| \beta_{s+1}\left(\Lambda(p(x^s);\lambda^{s+1},\beta_{s+1})-\lambda^{s+1}\right)-\beta_s(\lambda^{s+1}-\lambda^{s})\| \\&\leq \sqrt{ \left((\beta_s+\beta_{s+1})L_{h_1}+\|\beta_s \lambda_1^s-\beta_{s+1}\lambda_1^{s+1}\| \right)^2
					+\|\beta_s \lambda_2^s-\beta_{s+1}\lambda_2^{s+1}\|^2 }\enspace.
			\end{align}
			Therefore,
			\begin{align*}&
				\| \Lambda(p(x^s);\lambda^{s+1},\beta_{s+1})-\lambda^{s+1}\| \\&\leq  \beta_{s+1}^{-1} \beta_s \|\lambda^{s+1}-\lambda^s\|+  \beta_{s+1}^{-1}\sqrt{ \left((\beta_s+\beta_{s+1})L_{h_1}+\|\beta_s \lambda_1^s-\beta_{s+1}\lambda_1^{s+1}\| \right)^2
					+\|\beta_s \lambda_2^s-\beta_{s+1}\lambda_2^{s+1}\|^2 }.
			\end{align*}
			If follows that
			\begin{align}\label{a:erwe2}
				&\bE[ \| \Lambda(p(x^s);\lambda^{s+1},\beta_{s+1})-\lambda^{s+1}\|^2]\leq  2 \beta_{s+1}^{-2} \beta^2_s c+8\beta_{s+1}^{-2 } (\beta_s^2+\beta_{s+1}^2)(L^2_{h_1}+c)
			\end{align}
			By $\mathbb{E}[XY]\le (\bE[X^2])^{1/2}(\bE[Y^2])^{1/2}$, we get
			\begin{align}\label{a:fer}
				\bE\left[\|\lambda^{s+1}-\lambda^{s}\|\sqrt{ \left((\beta_s+\beta_{s+1})L_{h_1}+\|\beta_s \lambda_1^s-\beta_{s+1}\lambda_1^{s+1}\| \right)^2
					+\|\beta_s \lambda_2^s-\beta_{s+1}\lambda_2^{s+1}\|^2 }\right]\leq   \sqrt{4c(\beta_s^2+\beta_{s+1}^2)(L^2_{h_1}+c)}.
			\end{align}
			Combining~\eqref{a:fgert},~\eqref{a:erwe2} and~\eqref{a:fer}, we then get an upper bound for $\bE[M_{s}]$:
			\begin{equation*}
				\begin{array}{ll}
					\bE[M_{s+1}]&\leq {{\beta_s}}c+
					\frac{\beta_s-\beta_{s+1}}{2}\left( 2 \beta_{s+1}^{-2} \beta^2_s c+8\beta_{s+1}^{-2 } (\beta_s^2+\beta_{s+1}^2)(L^2_{h_1}+c)\right) + \frac{\beta_s^2}{2\beta_{s+1}-\beta_s}c
					+ \sqrt{4c(\beta_s^2+\beta_{s+1}^2)(L^2_{h_1}+c)} \\
					&\leq {{\beta_s}}c+
					\beta_s\left(  \beta_{s+1}^{-2} \beta^2_s c+4\beta_{s+1}^{-2 } (\beta_s^2+\beta_{s+1}^2)(L^2_{h_1}+c)\right) + \frac{\beta_s^2}{2\beta_{s+1}-\beta_s}c
					+ 2 \beta_s \sqrt{c(1+\beta_{s+1}^2\beta_s^{-2})(L^2_{h_1}+c)} \\
					& \leq 2{{\beta_s}}c  +  \beta_s\left(  \beta_{s+1}^{-2} \beta^2_s c+(5+4\beta_{s+1}^{-2 } \beta_s^2+\beta_{s+1}^2 \beta_s^{-2})(L^2_{h_1}+c)\right) + \frac{\beta_s^2}{2\beta_{s+1}-\beta_s}c
				\end{array},
			\end{equation*}
			where the last inequality used $2\sqrt{ab}\leq a+b $ for any $a,b>0$.  Next we plug in $\beta_s=\beta_0\rho^s$ to obtain
			$$
			\bE[M_s]\leq \beta_s \left(2c+\rho^{-2}c+(9+\rho^{-2})(L^2_{h_1}+c)+(2\rho-1)^{-1}c\right) \leq \beta_s \left((11+2\rho^{-2})(L^2_{h_1}+c)+(2\rho-1)^{-1}c\right).
			$$
		\end{proof}
		
		\begin{proof}[proof of Proposition~\ref{prop:maincomplexity}]
			Since Algorithm~\ref{ipALM_m} is a special case of Algorithm~\ref{ipALM} with $\beta_s=\beta_0 \rho^s$ and $\epsilon_s=\epsilon_0 \eta^s$, we know  from  Corollary~\ref{coro:bound} that~\eqref{a:fgert} holds with $c=4c_0$. Applying Lemma~\ref{l:erer2} we know that  
			$$
			\bE[M_s]\leq C\beta_s ,
			$$
			with $C=(11+2\rho^{-2})(L^2_{h_1}+4c_0)+4(2\rho-1)^{-1}c_0$.
			If $m_{s+1}$ is the smallest integer satisfying~\eqref{eq:ms2}, then
			\begin{align}\label{a:ms2}
				m_{s+1}\leq  K_{s+1}\left(\log_2 \left(4\epsilon_s \epsilon^{-1}_{s+1} +  2M_s    \epsilon^{-1}_{s+1}\right)+1\right)+1.
			\end{align}
			By the concavity of $\log_2$ function we get
			$$
			\bE[m_{s+1}] \leq K_{s+1}\left(\log_2 \left(4\epsilon_s \epsilon^{-1}_{s+1} +  2C\beta_s   \epsilon^{-1}_{s+1}\right)+1\right)+1=K_{s+1}\left(\log_2 \left(4 \eta^{-1} +  2C\beta_0  \epsilon_0^{-1}\eta^{-1}\rho^s\eta^{-s}\right)+1\right)+1.
			$$
			Since $\rho>\eta$, we get 
			\begin{align*}
				\bE[m_{s+1}] &\leq K_{s+1}\left(\log_2 \left( \left(4 \eta^{-1} +  2C\beta_0  \epsilon_0^{-1}\eta^{-1}\right)\rho^s\eta^{-s}\right)+1\right)+1\\&=K_{s+1}\left(\log_2 \left( 4 \eta^{-1} +  2C\beta_0 \epsilon_0^{-1}\eta^{-1}\right)+1+\log_2 \left( \rho^s\eta^{-s}\right) \right)+1\\&=K_{s+1}\left(s\log_2\left(\rho\eta^{-1}\right)+c_2\right)+ 1.
			\end{align*}
		\end{proof}

		\subsection{Proofs in Section~\ref{sec:H}}\label{app:pse}
		
		We first state a lemma similar to Lemma~\ref{l:infimalconv}.
		\begin{lemma}\label{l:infimalconv2}
			Let $\psi(\cdot):\R^n\rightarrow\R \cup\{+\infty\}$ 
			be a convex function.  Define:
			$$
			\tilde \psi(w):=\inf_x\{h(p(x)-w)+\psi(x)\},
			$$
			Then condition~\eqref{a:herffg} ensures the convexity of   $\tilde \psi$.
		\end{lemma}
		\begin{proof}
			The proof is similar to Lemma~\ref{l:infimalconv} in Appendix~\ref{app:pALM}.  
		\end{proof}  
		
		\begin{proof}[proof of Lemma~\ref{l:Lxbound2}]
			By the definitions~\eqref{a:pAL},~\eqref{a:hbetadef} and~\eqref{a:Lambda}, we have
			\begin{align*}
				&L(x;y, \lambda,\beta)-L(x;y',\lambda',\beta')+\frac{\beta}{2}\| \Lambda(p(x);\lambda,\beta)-\lambda\|^2-\frac{\beta'}{2}\| \Lambda(p(x);\lambda',\beta')-\lambda'\|^2
				\\&=\< \Lambda(p(x);\lambda,\beta)-\Lambda(p(x);\lambda',\beta'), p(x) >- h^*( \Lambda(p(x);\lambda,\beta))+
				h^*( \Lambda(p(x);\lambda',\beta'))+\frac{\beta}{2}\|x-y\|^2-\frac{\beta'}{2}\|x-y'\|^2.
			\end{align*}
			Next we apply~\eqref{a:optimalitycondition} to get
			$$
			h^*( \Lambda(p(x);\lambda',\beta')) \geq h^*( \Lambda(p(x);\lambda,\beta))+\<\Lambda(p(x);\lambda,\beta)-\Lambda(p(x);\lambda',\beta') ,\beta(\Lambda(p(x);\lambda, \beta)-\lambda)-p(x)>,
			$$
			and
			$$
			h^*( \Lambda(p(x);\lambda,\beta)) \geq h^*( \Lambda(p(x);\lambda',\beta'))+\<\Lambda(p(x);\lambda,\beta)-\Lambda(p(x);\lambda',\beta') ,p(x)-\beta'(\Lambda(p(x);\lambda', \beta')-\lambda')>.
			$$
		\end{proof}

		\begin{proof}[proof of Lemma~\ref{l:sL2}]
			In this proof we fix $y\in \R^n$, $\lambda\in \R^d$ and $\beta>0$. Recall the definitions in~\eqref{a:erdfgtff}.
			Define
			$$
			L(x, w; y, \lambda,\beta):=f(x)+g(x)+ h(p(x)-w)+\frac{1}{2\beta}\|w\|^2+\<w,\lambda>+\frac{\beta}{2}\|x-y\|^2-\frac{\beta}{2}\|x-x^\star(y,\lambda,\beta)\|^2.
			$$
			Then by~\eqref{a:dualityhbeta}, 
			\begin{align}\label{a:wrfff}
				\min_w L(x, w; y, \lambda,\beta)=
				L(x;y,\lambda,\beta)-\frac{\beta}{2}\|x-x^\star(y,\lambda,\beta)\|^2.
			\end{align}
			Since $L(x;y,\lambda,\beta)-\frac{\beta}{2}\|x-x^\star(y,\lambda,\beta)\|^2$ is a convex function with $x^\star(y,\lambda,\beta)$ being a critical point, 
			it follows that
			\begin{align}\label{a:poi}
				\min_x \min_w L(x, w; y, \lambda,\beta)=L^\star(y, \lambda,\beta).
			\end{align}
			Denote
			\begin{align}\label{a:Hdef}
				H(w;y, \lambda,\beta):=\min_x L(x, w;y, \lambda,\beta).
			\end{align}
			In view of~\eqref{a:wstar},
			\begin{align}\label{a:rtdffrtr}
				L(x;y, \lambda,\beta)-\frac{\beta}{2}\|x-x^\star(y,\lambda,\beta)\|^2=L(x, \beta(\Lambda(p(x); \lambda,\beta)-\lambda); y,\lambda,\beta)\overset{\eqref{a:Hdef}}{\geq} H(\beta(\Lambda(p(x);\lambda,\beta)-\lambda); y,\lambda,\beta).
			\end{align}
			Note that
			\begin{align}\label{a:efwq}
				\min_w H(w;y,\lambda,\beta)=\min_w \min_x L(x,w;y,\lambda,\beta)=\min_x \min_w L(x,w;y,\lambda,\beta)\overset{\eqref{a:poi}}{=}L^\star(y,\lambda,\beta).
			\end{align}
			Denote $\Lambda^\star(y,\lambda,\beta)=\Lambda(p^\star(y,\lambda,\beta);  \lambda,\beta)$. It follows that,
			$$
			H( \beta(\Lambda^\star(y,\lambda,\beta)-\lambda);y,\lambda,\beta) \geq \min_w H(w;y,\lambda,\beta)\overset{\eqref{a:efwq}}{=}L^\star(y,\lambda,\beta)=L(x^\star( y,\lambda,\beta); y,\lambda,\beta).
			$$
			Using again~\eqref{a:rtdffrtr} with $x=x^\star( y,\lambda,\beta)$ we deduce 
			\begin{align}\label{a:rqqsf}
				H( \beta(\Lambda^\star(y,\lambda,\beta)-\lambda);y,\lambda,\beta) = \min_w H(w;y,\lambda,\beta).
			\end{align}
			Moreover, it follows  from Lemma~\ref{l:infimalconv2} that $H( w;y, \lambda,\beta)$ is $1/\beta$-strongly convex with respect to $w$. Thus, 
			\begin{align*}
				L(x;y,\lambda,\beta)-L^\star(y,\lambda,\beta)-\frac{\beta}{2}\|x-x^\star(y,\lambda,\beta)\|^2
				&\overset{\eqref{a:rtdffrtr}+\eqref{a:efwq}}{\geq} H(\beta(\Lambda(p(x);\lambda,\beta)-\lambda);y,\lambda,\beta)-\min_w H(w;y,\lambda,\beta) \\
				&\overset{\eqref{a:rqqsf}}{\geq} \frac{1}{2\beta}\|\beta(\Lambda(p(x);\lambda,\beta)-\lambda)-\beta(\Lambda^\star(y,\lambda,\beta)-\lambda) \|^2\\
				&=\frac{\beta}{2} \| \Lambda(p(x);\lambda,\beta)-
				\Lambda^\star(y,\lambda,\beta)\|^2.
			\end{align*}
		\end{proof}
		\begin{proof}[proof of Lemma~\ref{l::boundbetabetap}]
			Denote
			\begin{align}\label{a:Lambdai}
				\Lambda_i(u_i;\lambda_i, \beta):=\arg\max_{\xi_i}\left\{\<\xi_i, u_i>-h_i^*(\xi_i)- \frac{\beta}{2}\|\xi_i-\lambda_i\|^2 \right\},
				\enspace i=1,2,
			\end{align}
			so that $\Lambda(u;\lambda, \beta)=\left(\Lambda_1(u_1;\lambda_1, \beta); \Lambda_{2}(u_{2};\lambda_{2}, \beta)\right)$.
			We can then decompose~\eqref{a:optimalitycondition} into two independent conditions:
			\begin{align}\label{a:optimalityconditionblock}
				\Lambda_i(u_i;\lambda_i,\beta)\in\partial h_i(u_i-\beta(\Lambda_i(u_i;\lambda_i, \beta)-\lambda_i)),\enspace i=1,2.
			\end{align}
			By condition~\ref{ass:hp1} in Assumption~\ref{ass:handp}, 
			\begin{align}\label{aseff}
				\| \Lambda_1(u_1;\lambda_1, \beta)\|\leq L_{h_1} \end{align}
			which yields directly
			\begin{align}\label{a:lff}
				\|\beta(\Lambda_1(u_1;\lambda_1, \beta)-\lambda_1)-\beta'(\Lambda_1(u_1;\lambda_1', \beta')-\lambda_1')\|
				\leq (\beta+\beta')L_{h_1}+\|\beta \lambda_1-\beta'\lambda_1'\|.
			\end{align}
			On the other hand,  since $h_2$ is an indicator function, $\partial h_2$ is a cone and~\eqref{a:optimalityconditionblock}
			implies
			\begin{align}\label{a:optimalityconditionblockcone}
				\beta\Lambda_2(u_2;\lambda_2,\beta)\in\partial h_2(u_2-\beta(\Lambda_2(u_2;\lambda_2, \beta)-\lambda_2)).
			\end{align}
			The latter condition further leads to
			$$
			\<\beta\Lambda_2(u_2;\lambda_2,\beta)-\beta'\Lambda_2(u_2;\lambda'_2,\beta'),\beta(\Lambda_2(u_2;\lambda_2, \beta)-\lambda_2)-\beta'(\Lambda_2(u_2;\lambda_2', \beta')-\lambda_2')>\leq 0,
			$$
			which  by Cauchy-Schwartz inequality implies
			$$
			\|\beta(\Lambda_2(u_2;\lambda_2,\beta)-\lambda_2)-\beta'(\Lambda_2(u_2;\lambda'_2,\beta')-\lambda_2')\|
			\leq \|\beta \lambda_2-\beta'\lambda_2'\|.
			$$
			Then~\eqref{a:boundbetabetap} is obtained by simple algebra.	
		\end{proof}
		
		\begin{proof}[proof of Lemma~\ref{l:erfgrte}]
			We first recall the following basic inequality:
			\begin{align}\label{a:bascc}
				\| u+v\|^2\leq (1+a)\|u\|^2+(1+1/a)\|v\|^2,\enspace \forall u,v\in \R^n, a>0.
			\end{align}
			In view of~\eqref{a:bascc} and the fact that $\beta'>\beta/2$, we know that
			\begin{align*}
				&	-\frac{\beta}{2}\| w'-w\|^2\leq \frac{\beta}{2}\|w-y'\|^2-\frac{\beta}{4}\|w'-y'\|^2,\\
				&  -\frac{\beta'+\beta/2}{2}\|w'-y'\|^2\leq  \frac{\beta(2\beta'+\beta)}{2(2\beta'-\beta)}\|y-y'\|^2 -\frac{\beta}{2}\|w'-y\|^2.
			\end{align*}
			Combining the latter two inequalities we get~\eqref{a:dger}.
		\end{proof}

		\subsection{Proof in Section~\ref{sec:li} }
		\begin{proof}[proof of Corollary~\ref{coro:lAPG}]
			If $K_s$ satisfies~\eqref{a:KsAPG}, then 
			$$
			K_s\leq 2\sqrt{\frac{2(L\beta_0+\|A\|^2)}{\mu_g\beta_s+\beta^2_s}}+1\leq
			\left\{ \begin{array}{ll}
			\frac{2\sqrt{2(L\beta_0+\|A\|^2)/\mu_g}}{ \sqrt{\beta_s}}+1  & \mathrm{if} ~\mu_g>0 \\
			\frac{2\sqrt{2(L\beta_0+\|A\|^2)}}{{\beta_s}}+1  & \mathrm{if} ~\mu_g=0
			\end{array}\right. 
			$$
			We then apply Corollary~\ref{coro:main}. 
		\end{proof}
		
		The proof of Corollary~\ref{coro:LL} and \ref{coro:BPG} are similar.
		\subsection{Proofs in Section~\ref{sec:nonlinear}}
		We first state a useful Lemma.
		\begin{lemma}
			For any $u,\lambda\in\R^d$, $\beta>0$,
			\begin{align}\label{a:ewssss}
				\|\Lambda(u;\lambda,\beta)\|\leq L_{h_1}+ \beta^{-1}\dist(u_2+\beta\lambda_2, \cK)
			\end{align}
		\end{lemma}
		\begin{proof}
			From~\eqref{a:dualityhbeta}, 
			\begin{align}\label{a:dualityhbeta34}
				& h(u;\lambda,\beta)=\min_z\left\{ h(z)+\frac{1}{2\beta}\|u+\beta \lambda-z\|^2-\frac{\beta}{2}\|\lambda\|^2 \right\} 
			\end{align}
			with optimal solution 
			$$
			z^*=u+\beta\lambda-\beta \Lambda(u;\lambda,\beta).
			$$
			In particular, $\dist(u_2+\beta \lambda_2, \cK)^2=\beta^2\|\Lambda_2(u_2;\lambda_2,\beta)\|^2$.   Together with~\eqref{aseff} we obtain the desired bound.
		\end{proof}
		\begin{proof}[proof of Lemma~\ref{l:ersdgg}]
			\begin{align*}&
				\| \nabla p(x) \Lambda(p(x);\lambda^s,\beta_s)-\nabla p(y) \Lambda(p(y);\lambda^s,\beta_s) \|
				\\& \leq \|\nabla p(x)-\nabla p(y)\|\|  \Lambda(p(x);\lambda^s,\beta_s)\|
				+\|\nabla p(y)\|\|\Lambda(p(x);\lambda^s,\beta_s)-\Lambda(p(y);\lambda^s,\beta_s) \|
				\\& \overset{\eqref{a:ewssss}+\eqref{a:nablahLip}}{\leq} L_{\nabla p} \|x-y\| \left(L_{h_1}+ \beta_s^{-1}\dist(p_2(x)+\beta_s\lambda^s_2, \cK)\right)+ M_{\nabla p}\| p(x)-p(y)\|\beta_s^{-1}
				\\&  \leq \left(L_{\nabla p} \left(L_{h_1}+ \beta_s^{-1}\dist(p_2(x)+\beta_s\lambda^s_2, \cK)\right)+ M^2_{\nabla p}\beta_s^{-1}\right) \|x-y\|.
			\end{align*}
			Note that by~\eqref{a:ererrrrewr} and the definition of $d_s$,
			$$
			\dist(p_2(x)+\beta_s\lambda_2^s, \cK)\leq d_s.
			$$
		\end{proof}
		
		\subsection{Proofs in Section~\ref{sec:kkt}}
		
		\begin{proof}[proof of Theorem~\ref{thm:kkt2}]
			We know from the basic property of proximal gradient step~\cite{ne07} that
			$$
			\|x^{s} -\tilde x^s\|^2  \leq 2\left(H_s(\tilde x^s)-H_s^\star\right)/L_{s}.
			$$
			By Line 4 in Algorithm~\ref{ipALM_m_kkt},
			$$
			0\in \nabla \phi_s(\tilde x^s)+L_{s}(x^s-\tilde x^s)+{\beta_s}( x^s-x^{s-1})+\partial g(x^s).
			$$
			Therefore,
			\begin{align*}
				\dist(0, \nabla \phi_s( x^s)+\partial g(x^s))&\leq L_{s} \|\tilde x^s-x^s\|+\|\nabla \phi_s(x^s)-\nabla \phi_s(\tilde x^s)\|+\beta_s \| x^s-x^{s-1} \| \\&\leq 2L_{s}\|\tilde x^s-x^s\|+\beta_s \|x^s-x^{s-1} \|
			\end{align*}
			Combining the last two bounds and~\eqref{a:nablehu} we get $\nabla \phi_s( x^s)= \nabla f( x^s)+\nabla p(x^s) \lambda^{s+1}$ and	$$
			\dist(0, \nabla f( x^s)+\nabla p(x^s) \lambda^{s+1}+\partial g( x^{s}))^2\leq 16 L_{s} \left(H_s(\tilde x^s)-H_s^\star\right)+2\beta_s^2\|x^s-x^{s-1}\|^2.
			$$
			Secondly we know from~\eqref{a:optimalitycondition} that
			\begin{align*}
				p(x^s)-\beta_s(\lambda^{s+1}-\lambda^s)\in \partial h^*(\lambda^{s+1}).
			\end{align*}
			It follows that 
			$$
			\dist(0, p(x^s)-\partial h^*(\lambda^{s+1}))\leq \beta_s \| \lambda^{s+1}-\lambda^s\|.
			$$
		\end{proof}

		\begin{proof}[proof of Corollary~\ref{coro:kkt}]
			Due to~\eqref{a:rttddd}, we can have the same  bound (in expectation) of the sequence $\{(\tilde x^s, x^s, \lambda^s)\}$ as Corollary~\ref{coro:bound}. Hence,
			\begin{align*}
				&\bE\left[\dist(0,\partial_x L(x^s, \lambda^{s+1}))  \right]\leq \sqrt{16 L_{s}\epsilon_s+8c_0\beta_s^2}\leq \sqrt{16\gamma\epsilon_0/\beta_0+8c_0\beta_0}\rho^s,\\
				&\bE\left[\dist(0,\partial_{\lambda}L(x^s, \lambda^{s+1}))  \right] \leq \beta_0\sqrt{c_0} \rho^s.
			\end{align*}
		\end{proof}

	\end{appendix}

	\bibliographystyle{abbrv}
	\bibliography{paper_ref}

\begin{thebibliography}{10}

\bibitem{ala17}
A.~Alacaoglu, Q.~Tran-Dinh, O.~Fercoq, and V.~Cevher.
\newblock Smooth primal-dual coordinate descent algorithms for nonsmooth convex
  optimization.
\newblock In {\em Advances in Neural Information Processing Systems}, pages
  5852--5861, 2017.

\bibitem{Katyusha}
Z.~Allen-Zhu.
\newblock Katyusha: The first direct acceleration of stochastic gradient
  methods.
\newblock In {\em The Journal of Machine Learning Research}, volume 18(1),
  pages 8194--8244, 2017.

\bibitem{Auslender:2005:IPM:3113613.3114011}
A.~Auslender and M.~Teboulle.
\newblock Interior {Projection-like Methods for Monotone Variational
  Inequalities}.
\newblock {\em Math. Program.}, 104(1):39--68, Sept. 2005.

\bibitem{beck2009fista}
A.~Beck and M.~Teboulle.
\newblock A fast iterative shrinkage-thresholding algorithm for linear inverse
  problems.
\newblock {\em SIAM Journal on Imaging Sciences}, 2(1):183--202, 2009.

\bibitem{BeckTeboulle12}
A.~Beck and M.~Teboulle.
\newblock Smoothing and {First Order Methods: A Unified Framework}.
\newblock {\em SIAM Journal on Optimization}, 22(2):557--580, 2012.

\bibitem{articlesqa}
A.~Belloni, V.~Chernozhukov, and L.~Wang.
\newblock {Square-Root} lasso: {Pivotal} {Recovery} of {Sparse} {Signals} via
  {Conic} {Programming}.
\newblock {\em SSRN Electronic Journal}, 01 2011.

\bibitem{ber14}
D.~P. Bertsekas.
\newblock {\em Constrained optimization and {Lagrange} multiplier methods}.
\newblock Academic press, 2014.

\bibitem{Bolte16}
J.~Bolte, H.~H.~Bauschke, and M.~Teboulle.
\newblock A {Descent} {Lemma} {Beyond} {Lipschitz} {Gradient} {Continuity}:
  {First-Order} {Methods} {Revisited} and {Applications}.
\newblock {\em Mathematics of Operations Research}, 42, 07 2016.

\bibitem{chambolle2018stochastic}
A.~Chambolle, M.~J. Ehrhardt, P.~Richt{\'a}rik, and C.-B. Schonlieb.
\newblock Stochastic primal-dual hybrid gradient algorithm with arbitrary
  sampling and imaging applications.
\newblock {\em SIAM Journal on Optimization}, 28(4):2783--2808, 2018.

\bibitem{cham11}
A.~Chambolle and T.~Pock.
\newblock A first-order primal-dual algorithm for convex problems with
  applications to imaging.
\newblock {\em Journal of mathematical imaging and vision}, 40(1):120--145,
  2011.

\bibitem{chang2011libsvm}
C.-C. Chang and C.-J. Lin.
\newblock {LIBSVM}: {A} library for support vector machines.
\newblock {\em ACM transactions on intelligent systems and technology (TIST)},
  2(3):27, 2011.

\bibitem{ChenDonohoSaubders}
S.~Chen, D.~Donoho, and M.~Saunders.
\newblock Atomic {Decomposition} by {Basis} {Pursuit}.
\newblock {\em SIAM Journal on Scientific Computing}, 20(1):33--61, 1998.

\bibitem{dru18}
D.~Drusvyatskiy and C.~Paquette.
\newblock Efficiency of minimizing compositions of convex functions and smooth
  maps.
\newblock {\em Mathematical Programming}, pages 1--56.

\bibitem{FercoqQu18}
O.~Fercoq and Z.~Qu.
\newblock Restarting the accelerated coordinate descent method with a rough
  strong convexity estimate.
\newblock {\em arXiv:1803.05771}, 2018.

\bibitem{FercoqQu17}
O.~Fercoq and Z.~Qu.
\newblock {Adaptive restart of accelerated gradient methods under local
  quadratic growth condition}.
\newblock {\em IMA Journal of Numerical Analysis}, 03 2019.

\bibitem{FR:2013approx}
O.~Fercoq and P.~Richt{\'a}rik.
\newblock Accelerated, parallel and proximal coordinate descent.
\newblock {\em SIAM Journal on Optimization}, 25(4):1997--2023, 2015.

\bibitem{FriedlanderGoh:2016}
M.~P. Friedlander and G.~Goh.
\newblock Efficient evaluation of scaled proximal operators.
\newblock {\em Electronic Transactions on Numerical Analysis}, 46:1--22, 2017.

\bibitem{BauschkeCombetter09}
H.~H.~Bauschke and P.~Combettes.
\newblock The baillon-haddad theorem revisited.
\newblock {\em Journal of {Convex Analysis}}, 17, 06 2009.

\bibitem{hien2017inexact}
L.~T.~K. Hien, R.~Zhao, and W.~B. Haskell.
\newblock An inexact primal-dual smoothing framework for large-scale
  non-bilinear saddle point problems.
\newblock {\em arXiv preprint arXiv:1711.03669}, 2017.

\bibitem{LSVRG}
D.~Kovalev, S.~Horv\'{a}th, and P.~Richt\'{a}rik.
\newblock Don't {Jump Through Hoops and Remove Those Loops}: {SVRG} and
  {Katyusha} are {Better Without the Outer Loop}.
\newblock 2019.

\bibitem{Lan:2016:IFA:2874819.2874858}
G.~Lan and R.~D. Monteiro.
\newblock Iteration-complexity of {First}-order {Augmented} {Lagrangian}
  {Methods} for {Convex} {Programming}.
\newblock {\em Math. Program.}, 155(1-2):511--547, Jan. 2016.

\bibitem{li2018complexity}
H.~Li and Z.~Lin.
\newblock On the {Complexity Analysis of the Primal Solutions for the
  Accelerated Randomized Dual Coordinate Ascent}.
\newblock {\em arXiv preprint arXiv:1807.00261}, 2018.

\bibitem{LiuLiuMa19MoR}
Y.~Liu, X.~Liu, and S.~Ma.
\newblock On the {Nonergodic} {Convergence} {Rate} of an {Inexact} {Augmented}
  {Lagrangian} {Framework} for {Composite} {Convex} {Programming}.
\newblock {\em Mathematics of Operations Research}, 44(2):632--650, 2019.

\bibitem{LuFreudNesterov}
H.~Lu, R.~Freund, and Y.~Nesterov.
\newblock Relatively {Smooth} {Convex} {Optimization} by {First-Order Methods},
  and {Applications}.
\newblock {\em SIAM Journal on Optimization}, 28(1):333--354, 2018.

\bibitem{lu18}
Z.~Lu and Z.~Zhou.
\newblock Iteration-complexity of first-order augmented {Lagrangian} methods
  for convex conic programming.
\newblock {\em arXiv preprint arXiv:1803.09941}, 2018.

\bibitem{NecoaraNesGli}
I.~Necoara, Y.~Nesterov, and F.~Glineur.
\newblock Linear convergence of first order methods for non-strongly convex
  optimization.
\newblock {\em Mathematical Programming}, Jan 2018.

\bibitem{NecoaraPatrascuGlineur}
I.~Necoara, A.~Patrascu, and F.~Glineur.
\newblock Complexity of first-order inexact {Lagrangian} and penalty methods
  for conic convex programming.
\newblock {\em Optimization Methods and Software}, 34(2):305--335, 2019.

\bibitem{NedelcuNecoaraTran}
V.~Nedelcu, I.~Necoara, and Q.~Tran-Dinh.
\newblock Computational {Complexity} of {Inexact} {Gradient} {Augmented}
  {Lagrangian} {Methods}: {Application} to {Constrained} {MPC}.
\newblock {\em SIAM Journal on Control and Optimization}, 52(5):3109--3134,
  2014.

\bibitem{nesterov1983method}
Y.~Nesterov.
\newblock A method of solving a convex programming problem with convergence
  rate ${O}(1/k^2)$.
\newblock {\em Soviet Mathematics Doklady}, 27(2):372--376, 1983.

\bibitem{Nesterov2005}
Y.~Nesterov.
\newblock Smooth minimization of non-smooth functions.
\newblock {\em Mathematical Programming}, 103(1):127--152, May 2005.

\bibitem{ne07}
Y.~Nesterov et~al.
\newblock Gradient methods for minimizing composite objective function, 2007.

\bibitem{ouyang15}
Y.~Ouyang, Y.~Chen, G.~Lan, and E.~P. Jr.
\newblock An accelerated linearized alternating direction method of
  multipliers.
\newblock {\em SIAM Journal on Imaging Sciences}, 8(1):644--681, 2015.

\bibitem{PatrascuNecoara15}
A.~Patrascu, I.~Necoara, and Q.~Tran-Dinh.
\newblock Adaptive inexact fast augmented {Lagrangian} methods for constrained
  convex optimization.
\newblock {\em Optimization Letters}, 11, 05 2015.

\bibitem{LKatyusha}
X.~Qian, Z.~Qu, and P.~Richt\'{a}rik.
\newblock L-{SVRG} and {L-Katyusha} with arbitrary sampling.
\newblock {\em arXiv:1906.01481}, 2019.

\bibitem{rafique2018non}
H.~Rafique, M.~Liu, Q.~Lin, and T.~Yang.
\newblock Non-convex min-max optimization: {Provable} algorithms and
  applications in machine learning.
\newblock {\em arXiv preprint arXiv:1810.02060}, 2018.

\bibitem{rock76}
R.~T. Rockafellar.
\newblock Augmented {Lagrangians} and applications of the proximal point
  algorithm in convex programming.
\newblock {\em Mathematics of operations research}, 1(2):97--116, 1976.

\bibitem{rock76ppa}
R.~T. Rockafellar.
\newblock Monotone operators and the proximal point algorithm.
\newblock {\em SIAM journal on control and optimization}, 14(5):877--898, 1976.

\bibitem{751369}
P.~O.~M. {Scokaert}, D.~Q. {Mayne}, and J.~B. {Rawlings}.
\newblock Suboptimal model predictive control (feasibility implies stability).
\newblock {\em IEEE Transactions on Automatic Control}, 44(3):648--654, March
  1999.

\bibitem{Simon13asparse-group}
N.~Simon, J.~Friedman, T.~Hastie, and R.~Tibshirani.
\newblock A sparse-group lasso.
\newblock {\em Journal of Computational and Graphical Statistics}, 2013.

\bibitem{Tibshirani05sparsityand}
R.~Tibshirani, M.~Saunders, S.~Rosset, J.~Zhu, and K.~Knight.
\newblock Sparsity and smoothness via the fused lasso.
\newblock {\em Journal of the Royal Statistical Society Series B}, pages
  91--108, 2005.

\bibitem{tran18ada}
Q.~Tran-Dinh, A.~Alacaoglu, O.~Fercoq, and V.~Cevher.
\newblock An {Adaptive Primal-Dual Framework for Nonsmooth Convex
  Minimization}.
\newblock {\em arXiv preprint arXiv:1808.04648}, 2018.

\bibitem{tran18}
Q.~Tran-Dinh, O.~Fercoq, and V.~Cevher.
\newblock A smooth primal-dual optimization framework for nonsmooth composite
  convex minimization.
\newblock {\em SIAM Journal on Optimization}, 28(1):96--134, 2018.

\bibitem{tseng2008accelerated}
P.~Tseng.
\newblock On accelerated proximal gradient methods for convex-concave
  optimization.
\newblock {\em Submitted to SIAM Journal on Optimization}, 2008.

\bibitem{wang2007robust}
H.~Wang, G.~Li, and G.~Jiang.
\newblock Robust regression shrinkage and consistent variable selection through
  the lad-lasso.
\newblock {\em Journal of Business and Economic Statistics}, 25(3):347--355,
  2007.

\bibitem{xu2017first}
Y.~Xu.
\newblock First-order methods for constrained convex programming based on
  linearized augmented {Lagrangian} function.
\newblock {\em arXiv preprint arXiv:1711.08020}, 2017.

\bibitem{Xu2017IterationCO}
Y.~Xu.
\newblock Iteration complexity of inexact augmented {Lagrangian} methods for
  constrained convex programming.
\newblock {\em arXiv:1711.05812}, 2017.

\bibitem{xu2018accelerated}
Y.~Xu and S.~Zhang.
\newblock Accelerated primal--dual proximal block coordinate updating methods
  for constrained convex optimization.
\newblock {\em Computational Optimization and Applications}, 70(1):91--128,
  2018.

\bibitem{Yuan18LinearconvergenceofADMM}
X.~Yuan, S.~Zeng, and J.~Zhang.
\newblock Discerning the linear convergence of {ADMM} for structured convex
  optimization through the lens of variational analysis.
\newblock {\em optimization-online}, 2018.

\bibitem{Zhu:2003:SVM:2981345.2981352}
J.~Zhu, S.~Rosset, T.~Hastie, and R.~Tibshirani.
\newblock 1normm {Support} {Vector} {Machines}.
\newblock In {\em Proceedings of the 16th International Conference on Neural
  Information Processing Systems}, NIPS'03, pages 49--56, Cambridge, MA, USA,
  2003. MIT Press.

\end{thebibliography}

\end{document}